\documentclass[onefignum,onetabnum]{siamart190516}

\usepackage{todonotes}

\def\D{{\rm d}}
\def\dx{\D x}
\def\F(#1,#2,#3;#4){{}_2 F_1\left(\begin{matrix}#1,\quad #2 \\ #3\end{matrix};#4 \right)}

\usepackage{amsmath,amssymb,mathabx,amsfonts,mathtools}
\usepackage{bbm}
\usepackage{graphicx}
\usepackage{epstopdf}
\usepackage{algorithmic}
\usepackage{tikz} 
\usetikzlibrary{matrix}
\usepackage{subfig}
\ifpdf
  \DeclareGraphicsExtensions{.eps,.pdf,.png,.jpg}
\else
  \DeclareGraphicsExtensions{.eps}
\fi
\def\ip<#1>{\left\langle{#1}\right\rangle}

\newsiamremark{remark}{Remark}
\newsiamremark{hypothesis}{Hypothesis}
\crefname{hypothesis}{Hypothesis}{Hypotheses}
\newsiamthm{claim}{Claim}

\headers{Computation of equilibrium measures on balls}{Timon S. Gutleb, Jos\'e Carrillo, Sheehan Olver}

\title{Computation of power law equilibrium measures on balls of arbitrary dimension}

\author{Timon S. Gutleb\thanks{Department of Mathematics, Imperial College London, UK
  (\email{t.gutleb18@imperial.ac.uk}).}
\and  Jos\'e A. Carrillo\thanks{Mathematical Institute, University of Oxford, UK 
  (\email{carrillo@maths.ox.ac.uk)}}
  \and Sheehan Olver\thanks{Department of Mathematics, Imperial College London, UK 
  (\email{s.olver@imperial.ac.uk}).} 
  }

\usepackage{amsopn}

\ifpdf
\hypersetup{
  pdftitle={Computation of power law equilibrium measures balls of arbitary dimension},
  pdfauthor={Timon S. Gutleb, Sheehan Olver, José A. Carrillo}
}
\fi

\def\supp{{\rm supp}}

\begin{document}

\maketitle

\begin{abstract}
We present a numerical approach for computing attractive-repulsive power law equilibrium measures in arbitrary dimension. We prove new recurrence relationships for radial Jacobi polynomials on $d$-dimensional ball domains, providing a substantial generalization of the work started in \cite{gutleb_computing_2020} for the one-dimensional case based on recurrence relationships of Riesz potentials on arbitrary dimensional balls. Among the attractive features of the numerical method are good efficiency due to recursively generated banded and approximately banded Riesz potential operators and computational complexity independent of the dimension $d$, in stark contrast to the widely used particle swarm simulation approaches for these problems which scale catastrophically with the dimension. We present several numerical experiments to showcase the accuracy and applicability of the method and discuss how our method compares with alternative numerical approaches and conjectured analytical solutions which exist for certain special cases. Finally, we discuss how our method can be used to explore the analytically poorly understood gap formation boundary to spherical shell support.
\end{abstract}



\section{Introduction}\label{sec:intro}
In this paper we present a novel banded and approximately banded spectral method for the computation of equilibrium measures with power law kernels on ball domains of arbitrary dimension, with computational cost independent of the dimension. Equilibrium measure problems naturally arise as the continuous limit of particle swarm systems and as such find various applications in the physics and biology, ranging from modelling cellular movement and interactions of charged particles to the flocking behaviour of birds \cite{topaz2006nonlocal,bertozziPRL,parrish1997animal,carrillo2010particle,survey2013,carrillo2018adhesion,carrillo2017review,bertozzi2015ring}. \\ Consider the problem of finding stationary states of a discrete system of $N$ particles with a pairwise interaction potential $K(|x_i-x_j|)$:
\begin{align*}
\frac{\mathrm{d}^2 x_i}{\mathrm{d}t^2} = f\left(\left\lvert\frac{\mathrm{d}x_i}{\mathrm{d}t}\right\rvert\right)\frac{\mathrm{d}x_i}{\mathrm{d}t}- \frac{1}{N}\sum_{j\neq i}\nabla K(|x_i-x_j|),
\end{align*}
where the $f$ term correponds to self-propulsion and friction, see \cite{carrillo_explicit_2017}. In the continuous limit $N\rightarrow \infty$ and in the absence of friction the evolution of the system is governed by the aggregation equation
\begin{align}\label{eq:aggregationevolution}
\rho_t = \nabla \cdot (\rho \nabla K * \rho),
\end{align}
the equilibrium states of which are characterized by non-negative densities $\rho$ such that $$\nabla K * \rho=0$$ on $\supp(\rho)$. The stable stationary states are found among the local minimizers of the energy $\int \rho(K*\rho)\dx$ with a given mass condition $\int \rho \dx = M$. An Euler-Lagrange approach, cf. \cite{balague2013dimensionality,carrillo_explicit_2017}, shows that global minimizers are among non-negative densities satisfying
\begin{align*}
K*\rho = E \quad \text{on} \quad \supp(\rho),
\end{align*}
with the constant $E/2$ given by the minimal total energy, which constitutes a more tractable problem. From a numerical perspective this minimization problem remains extremely difficult, however, as naively it would require us to minimize over the space of all positive densities of unknown radius. For the important special case problem of power law kernels, that is
\begin{align}\label{eq:powerlawkernel}
K(|x-y|) = \frac{|x-y|^\alpha}{\alpha}-\frac{|x-y|^\beta}{\beta},
\end{align}
there are some existence and uniqueness results \cite{lopes2017uniqueness,canizo2015existence,choksi2015minimizers,carrillo_preprint} and for a few specific parameter combinations even analytic solutions \cite{carrillo_explicit_2017,CHM2}. The radial symmetry of these kernels suggests that the equilibrium states are themselves radially symmetric. However, this is not true and the loss of radial symmetry for local and global minimizers is very subtle and not theoretically shown even in the case of the power-law kernels \eqref{eq:powerlawkernel}, see \cite{bertozzi2015ring}. Only very recently \cite{carrillo_preprint}, we have been able to show that the explicit radially symmetric solutions found in particular ranges in \cite{carrillo_explicit_2017} are in fact the global minimizers of the total interaction energy with $K$ given by \eqref{eq:powerlawkernel} with parameters inside the uniqueness range in \cite{lopes2017uniqueness}. The authors in \cite{carrillo_preprint} also give a generic family of kernels leading to non radially symmetric global minimizers. 

In the simplest cases, the support of $\rho$ is just a single interval in one-dimension \cite{carrillo_explicit_2017,gutleb_computing_2020}, a disk in two dimensions and a $d$-dimensional ball in higher dimensions \cite{carrillo_explicit_2017}. For different, in particular higher, powers $\alpha$ and $\beta$ different support domains for the equilibrium measure may be observed when performing particle simulations, such as two interval systems in one-dimension \cite{carrillo_explicit_2017,gutleb_computing_2020}, annuli in two dimensions and hyperspherical shells in higher dimensions \cite{carrillo_explicit_2017}. While this is known from numerical particle swarm simulations, there are as of now very few analytic results on this gap formation phenomenon and to our knowledge the only numerical results dealing with the continuous gap formation is the one-dimensional ultraspherical spectral method described in \cite{gutleb_computing_2020}. The present paper sets the foundations for studying these phenomena in higher dimensions and can be seen as an arbitrary dimensional generalization of \cite{gutleb_computing_2020}, with many of the obtained results mirroring the one-dimensional case. As such, the method is also related to \cite{olver2011computation}, in which the last author introduced a Chebyshev spectral method for the computation of potential theory equilibrium measures with kernels $K_0 = \log(|x-y|)$. 
The path to justifying similar methods in higher dimensions is fundamentally different because the one-dimensional method was able to leverage the structure of Riemann--Liouville fractional integrals. In constrast, their higher dimensional analogues called Riesz potentials are significantly more restrictive.

A lot of the material required to develop this method is dispersed across different subdisciplines and fields. Equilibrium measure problems, themselves often discussed in biology or physics publications due to interest in their applications \cite{CARRILLO201875,mogilner_non-local_1999}, have close ties to mathematical potential theory \cite{saff_logarithmic_1997}. The power law potentials used in this paper are in particular related to Riesz potentials and fractional Laplacians on balls \cite{carrillo_explicit_2017,dyda_fractional_2017}. Much of the groundwork for the modern theory of these concepts was laid by Rubin and Samko \cite{samko_fractional_1993,rubin_fractional_1996} and while some of their papers and books are available in English much of their earlier important work appears to remain untranslated from their Russian originals, see e.g. \cite{karapetyants_radial_1982,rubin_method_1986}. Furthermore, the theory of Zernike polynomials, a disk generalization of Jacobi polynomials, has been developed in tandem with applied fields such as optical abberrations research \cite{born_principles_2019,mahajan_zernike_1994,tango_circle_1977,thibos_standards_2000}, which has led to a number of different notations and conventions for radially shifted Jacobi polynomials for $d$-dimensional balls which feature prominently in our numerical method. As a result of all of the above, we provide a detailed overview of the relevant results in the early sections of the paper, while providing references to both standard monographs and important recent results.

The sections of this paper are organized as follows: Section \ref{sec:intro} serves as an introduction, giving an overview of the required concepts for power law equilibrium measures, sparse spectral methods using Jacobi polynomials in one and higher dimensions as well as the theory of fractional Laplacians and Riesz potentials. After collecting auxilliary lemmas and theorems section \ref{sec:resultsandproofs} presents new results and proofs for recurrence relationships and explicit formulae for power law integrals of Jacobi polynomials and Gaussian hypergeometric functions on $d$-dimensional balls. Section \ref{sec:methoddescription} contains a discussion of how these results may be applied to compute equilibrium measures, discusses the banded and approximately banded structures found for appropriate basis choices and finishes with a discussion of regularization and stability for this approach. Section \ref{sec:numexp} contains extensive numerical validation by means of comparison to special case analytic solutions as well as particle simulations and explores conjectures and other analytically unknown properties such as the uniqueness of solutions and the boundary of the gap forming behaviour for high parameter regimes. We close with a review of the method's properties and future research pathways.

\subsection{Power law equilibrium measures in higher dimensions}\label{sec:introequilibrium}
\begin{definition}[Equilibrium measure]
Given a kernel $K:\mathbb{R}^d \rightarrow \mathbb{R}$, $x,y\in\mathbb{R}^d$ we define the equilibrium mesaure to be the measure $\mathrm{d}\rho = \rho(x) \mathrm{d}x$ such that the expression
\begin{equation}\label{eq:energyfull}
 \frac{1}{2} \int \int K(x-y) \mathrm{d}\rho(x) \mathrm{d}\rho(y)
\end{equation}
is minimized given the mass condition
$$ \int_{\supp(\rho)}\rho(y) \mathrm{d}y = M.$$
When $M=1$ the equilibrium measure is a probability measure.
\end{definition}
Note that the energy as defined in \eqref{eq:energyfull} is a Lyapunov functional of the evolution equation stated in \eqref{eq:aggregationevolution}. For the power law kernel in \eqref{eq:powerlawkernel} one can derive the Euler-Lagrange conditions as outlined above, more details and references on this approach may be found in \cite{balague2013dimensionality,canizo2015existence,carrillo_explicit_2017}. In case the global minimizers are unique \cite{lopes2017uniqueness,carrillo_preprint}, then the equilibrium measure may instead be found by solving the following governing equation
\begin{equation*}
\frac{1}{\alpha}\int_{\supp(\rho)} |x-y|^\alpha \rho(y) dy - \frac{1}{\beta}\int_{\supp(\rho)} |x-y|^\beta \rho(y) dy = E.
\end{equation*}
In any case, they are candidates for being global minimizers of the interaction energy \eqref{eq:energyfull}.
By the radial symmetry of the problem, assuming that the measure $\rho(y)$ is supported on the ball $B_R = \{ x\in \mathbb{R}^d, |x|\leq R \}$, the $d$ dimensional governing equation takes the form \cite{carrillo_explicit_2017}:
\begin{equation*}
\frac{1}{\alpha}\int_{B_R} |x-y|^\alpha \rho(y) dy - \frac{1}{\beta}\int_{B_R} |x-y|^\beta \rho(y) dy = E,
\end{equation*}
as before subject to a mass condition
\begin{align*}
\int_{B_R} \rho(y) dy = M.
\end{align*}
Our spectral method will be developed to find the minimizer among radial densities supported on a ball. General conditions on the interaction potential leading to the radial symmetry of local/global minimizers of \eqref{eq:energyfull} have been recently shown in \cite{carrillo_preprint}.

\subsection{Jacobi polynomials and sparse spectral methods}\label{sec:jacobi}
We say polynomials $p_n(x)$ are orthogonal on a given domain $\Omega \subseteq \mathbb{R}$ with respect to some weight function $w(x)$ if they satisfy
\begin{equation*}
\int_\Omega w(x) p_n(x) p_m(x) \mathrm{d} x = c_{n,m} \delta_{n,m},
\end{equation*}
where $c_{n,m}$ are constants and $\delta_{n,m}$ is the Kronecker delta.
A complete orthogonal polynomial basis  may be used to expand sufficiently smooth functions via
\begin{equation*}
f(x) = \sum_{n=0}^\infty f_n p_n(x)= \mathbf{P}(x)^\mathsf{T} \mathbf{f},
\end{equation*}
with constant coefficients $f_n$ and
\begin{align*}
     \mathbf{P}(x) := \begin{pmatrix}
           p_{0}(x) \\
           p_{1}(x) \\
           \vdots
         \end{pmatrix} &,\hspace{5mm}
     \mathbf{f} := \begin{pmatrix}
           f_{0} \\
           f_{1} \\
           \vdots
         \end{pmatrix}.
\end{align*}
This expansion can be used to numerically approximate functions by only summing over the first $N$ number of terms. There has been much interest in the mathematical properties of such polynomials, leading to a multitude of high quality formula collections \cite{nist_2018} and monographs \cite{dunkl_orthogonal_2014}, to which we refer for details beyond the scope of this paper.\\
Various complete orthogonal polynomial bases, most notably among them the ultraspherical, Jacobi and Chebyshev polynomials, are used to numerically approximate functions and solve differential and integral equations using sparse, in particular banded, operators. A recent in-depth survey may be found in \cite{olver_slevinsky_townsend_2020}. These methods converge fast with competitive computational cost due to sparsity. Open source software implementations are available to perform such computations in the form of ApproxFun.jl \cite{noauthor_juliaapproximation/approxfun.jl_2019} and FastTransforms.jl \cite{slevinsky_conquering_2017,slevinsky_fasttransforms_2020}, which also form the framework within which the method introduced in this paper has been implemented for the numerical experiments in section \ref{sec:numexp}.

In the remainder of this section we will describe properties of one of the most ubiquitously used one-dimensional bases for these methods, the Jacobi polynomials, which also often feature as components in various higher dimensional bases, including ball domains of arbitrary dimension. The Jacobi polynomials are conventionally denoted $P_n^{(a,b)}(x)$ and are orthogonal on the interval $[-1,1]$ with respect to the weight function 
\begin{equation*}
w^{(a,b)}(x)=(1-x)^a (1+x)^b,
\end{equation*}
satisfying
\begin{align}\label{eq:orthogonalityconditionJacobi}
\int_{-1}^1(1-x)^a(1+x)^bP_n^{(a,b)}\dx = \frac{2^{a+b+1} \Gamma (a+n+1) \Gamma (b+n+1)}{n! (a+b+2 n+1) \Gamma (a+b+n+1)}  \delta _{n,m},
\end{align}
with $a	,b>-1$. The special case of $a=b$ corresponds to the ultraspherical or Gegenbauer polynomials with slightly different normalization, while the case $a=b=0$ corresponds to the Legendre polynomials \cite[Table 18.3.1]{nist_2018}. We will find the following explicit representation of the $n$-th Jacobi polynomial to be useful for our purposes \cite[18.5.7]{nist_2018}:
\begin{align}\label{eq:jacobihyper}
P^{(a,b)}_{n}\left(x\right)&=\sum_{k=0}^{n}\frac{{\left(n+a+%
b+1\right)_{k}}{\left(a+k+1\right)_{n-k}}}{k!\;(n-k)!}%
\left(\frac{x-1}{2}\right)^{k}\\ \nonumber &=\frac{{\left(a+1\right)_{n}}}{n!}\F(n+a+b+1,-n,\alpha+1; \frac{1-x}{2}),
\end{align}
where $\F(a,b,c;z)$ denotes the Gaussian hypergeometric function \cite[
15.2.1]{nist_2018} and $(\cdot)_n$ denotes the Pochhammer symbol or rising factorial \cite[5.2.5]{nist_2018}. Being classical orthogonal polynomials, the Jacobi polynomials also satisfy the following three-term recurrence relationship \cite[18.9.2]{nist_2018}:
\begin{equation}\label{eq:jacobirecclassical}
P_{n+1}^{(a,b)}(x)=(A_{n}x+B_{n})P^{(a,b)}_{n}(x)-C_{n}P^{(a,b)}_{n-1}(x),
\end{equation}
with constants
\begin{align*}
A_{n}&=\tfrac{(2n+a+b+1)(2n+a+b+2)}{2(n+1)(n+a+b+1)},\\
B_{n}&=\tfrac{(a^{2}-b^{2})(2n+a+b+1)}{2(n+1)(n+a+b+%
1)(2n+a+b)},\\
C_{n}&=\tfrac{(n+a)(n+b)(2n+a+b+2)}{(n+1)(n+a+b+1)(2%
n+a+b)}.
\end{align*}
This recurrence may be used to define a tridiagonal multiplication operator $\mathrm{X}$ acting on a function expanded in the Jacobi polynomials
\begin{equation*}
    \mathbf{P}(x)^\mathsf{T} \mathrm{X} \mathbf{f} = x f(x).
\end{equation*}
Together with element-wise addition and subtraction, as well as the raising and lowering operators defined below, these properties make up the basic toolbox for function approximation and arithmetic used in sparse spectral methods. The range of recurrences and symmetry properties of the Jacobi polynomials is extensive and we will not attempt to reproduce it fully here. A large collection of these may be found in \cite[18.3--18.18]{nist_2018} and a number of additional ladder operators are listed in \cite{olver_recurrence_2018}. Instead we selectively state some results which we will be using in the development of our method, beginning with the symmetry relation \cite[18.6.1]{nist_2018}:
\begin{equation}\label{eq:jacobisymm}
P_n^{(a,b)}(-x) = (-1)^n P_n^{(b,a)}(x).
\end{equation}
The following two relationships between parameters and degree allow one to take care of additional weight terms \cite[18.9.6]{nist_2018}:
\begin{align}
(1+x)P^{(a,b+1)}_{n}\left(x%
\right)=\tfrac{(n+1)}{(n+\frac{a}{2}+\frac{b}{2}+1)}P^{(a,b)}_{n+1}\left(x\right)+\tfrac{(n+b+1)}{(n+\frac{a}{2}+\frac{b}{2}+1)}P^{(a,%
b)}_{n}\left(x\right),\label{eq:jacobiraisingB}\\
(1-x)P^{(a+1,b)}_{n}\left(x%
\right)=\tfrac{-(n+1)}{(n+\frac{a}{2}+\frac{b}{2}+1)}P^{(a,b)}_{n+1}\left(x\right)+\tfrac{(n+a+1)}{(n+\frac{a}{2}+\frac{b}{2}+1)}P^{(a,%
b)}_{n}\left(x\right).\label{eq:jacobiraisingA}
\end{align}
Basis conversions between different parameter Jacobi polynomials may also be accomplished with recurrence relations such as \cite[18.9.5]{nist_2018}:
\begin{align}\label{eq:basisconversionjacobi1}
P^{(a,b)}_{n}\left(x\right)=\tfrac{(n+a+b+1)}{(2n+a+b+1)}P^{(%
a+1,b)}_{n}\left(x\right)-\tfrac{(n+b)}{(2n+a+b+1)}P^{(a+1,b)}_{n-1}\left(x%
\right),\\
P^{(a,b)}_{n}\left(x\right)=\tfrac{(n+a+b+1)}{(2n+a+b+1)}P^{(%
a,b+1)}_{n}\left(x\right)+\tfrac{(n+a)}{(2n+a+b+1)}P^{(a,b+1)}_{n-1}\left(x%
\right).\label{eq:basisconversionjacobi2}
\end{align}
Boundary evaluations for the Jacobi polynomials can be made using \cite[Table 18.6.1]{nist_2018}:
\begin{align}\label{eq:jacobieval1}
&P_n^{(a,b)}(1)=\tfrac{(a+1)_n}{n!},\\
&P_n^{(a,b)}(-1)=(-1)^n \tfrac{(b+1)_n}{n!}.\label{eq:jacobieval2}
\end{align}
Finally, we have derivatives \cite[18.9.15]{nist_2018}:
\begin{align}\label{eq:derivativesjac1}
&\frac{\D P_{n}^{(a,b)}}{\dx}=\tfrac{(a+b+n+1)}{2} P_{n-1}^{(a+1,b+1)}(x),\\\label{eq:derivativesjac2}
&\frac{\D^2 P_{n}^{(a,b)}}{\dx^2} = \tfrac{(a+b+n+1) (a+b+n+2)}{4}  P_{n-2}^{(a+2,b+2)}(x).
\end{align}
Sometimes applications may instead want to make use of shifted Jacobi polynomials, which via straightforward variable substitutions move the domain of orthogonality to e.g. $[0,1]$ instead. As we will see in the next section, a certain type of shifted Jacobi polynomial, namely $P_n^{(a,b)}(2r^2-1)$, $r\in(0,1)$, will feature prominently in the development of our method for equilibrium measures. We thus also reproduce variants of the above-listed properties for the shifted radial basis $P_n^{(a,b)}(2r^2-1)$ in Appendix \ref{app:radialjacobi} as a reference.

\subsection{Sparse spectral methods on unit balls of arbitrary dimension}\label{sec:spectralondisk}
Substantial progress has been made over the past years in expanding the tools described in the previous section to various higher dimensional domains such as triangles \cite{olver_recurrence_2018,olver_sparse_2019}, disks \cite{vasil_tensor_2016}, wedges \cite{olver_orthogonal_2019}, disk slices and trapeziums \cite{snowball_sparse_2020}, quadratic curves \cite{olver_orthogonal_nodate} as well as surfaces of revolution \cite{olver_orthogonal_2020} using multivariate orthogonal polynomials \cite{dunkl_orthogonal_2014}.\\
The most ubiquitous set of orthogonal polynomials on the disk are the so-called Zernike polynomials \cite{zernike1935hyperspharische,zernike_beugungstheorie_1934}, a set of bivariate orthogonal polynomials typically given in terms of polar coordinates $r$ and $\theta$. The radial part of the Zernike polynomials can be written in terms of shifted Jacobi polynomials and as a result they inherit many of their useful properties. This turns out to be the case for higher dimensional ball domains as well. Due to the Zernike polynomials' numerous applications \cite{wallis_modelling_2002} especially in the study of optical abberations \cite{born_principles_2019,rocha_effects_2007,mahajan_zernike_1994,tango_circle_1977,1992aooe,thibos_standards_2000} and atmospheric wavefronts \cite{noll_zernike_1976,roddier_atmospheric_1990}, one finds that their treatment in the scientific literature is dispersed across different fields in multiple disciplines, with normalization and ordering of the polynomials often inconsistent between different publications, see e.g. \cite{greengard_zernike_2018,noll_zernike_1976,1992aooe}.\\
The spectral method we propose in this paper uses a generalized set of polynomials on the $d$ dimensional unit ball. In two dimensions, this basis is usually referred to as the generalized Zernike polynomials \cite{dunkl_orthogonal_2014,wunsche_generalized_2005,torre_generalized_2008,aharmim_generalized_2015,area_recursive_2017}, or sometimes simply Jacobi polynomials on the unit disk $\mathbb{D} = \{x \in \mathbb{R}^2, |x| \leq 1  \}$. They are orthogonal on $\mathbb{D}$ with respect to the weight function
\begin{equation*}
w^k(x) = (1-|x|^2)^k.
\end{equation*}
The radial part \cite{vasil_tensor_2016} of the generalized Zernike polynomials reads
\begin{align*}
Q^{(k,m)}_{n}(|x|) = |x|^m P_n^{(k,m)}(2|x|^2-1),
\end{align*}
where $P_n^{(k,m)}(2|x|^2-1)$ are shifted one-dimensional Jacobi polynomials as introduced in the previous section. Typically, an additional normalization constant is also present. The angle component of the generalized Zernike polynomials corresponds to Fourier modes, see e.g. \cite{vasil_tensor_2016,dunkl_orthogonal_2014}. For radially symmetric functions only the $m=0$ modes of the basis are relevant, meaning that an appropriate basis of orthogonal polynomials on the $d$-dimensional unit ball is simply given by the polynomials
\begin{align*}
P_n^{(k,\frac{d-2}{2})}(2|x|^2-1),
\end{align*}
which are orthogonal with respect to $(1-|x|^2)^k$. For $k=0$ the above definitions are equivalent to the standard Zernike polynomials, making `generalized Zernike polynomials' an appropriate description. We note that the ordinary Zernike polynomials thus have a similar relationship to the generalized Zernike polynomials as the Legendre polynomials do to the ultraspherical or Jacobi polynomials, cf. \cite{bhatia_circle_1954,wunsche_generalized_2005,vasil_tensor_2016}.\\ As a direct consequence of the generalized Zernike polynomials' defining relationship with shifted Jacobi polynomials, they satisfy a recurrence relationship and various other properties in the radial parameter, see Appendix \ref{app:radialjacobi}.\\

\subsection{The fractional Laplacian of radial functions}\label{sec:laplace}
The Laplace operator and its non-local generalization in fractional calculus are important tools for the development of our method, so we give a brief overview of the relevant results. First, a standard definition:
\begin{definition}[Fractional Laplace operator]\label{def:fraclaplace} We define the negative fractional Laplace operator $(-\Delta)^\frac{\gamma}{2}$ for $\gamma\in(0,2)$ via the following singular integral
\begin{align*}
(-\Delta)^\frac{\gamma}{2}f(x) = \frac{2^\gamma |\Gamma(\frac{d+\gamma}{2})|}{\pi^{\frac{d}{2}}\Gamma({-\frac{\gamma}{2}})} \lim_{\epsilon\rightarrow 0^+} \frac{}{} \int\limits_{\mathbb{R}^d \backslash B_\epsilon}{\frac{f(x)-f(y)}{|x-y|^{d+\gamma}}\,dy},
\end{align*}
where $B_\epsilon = B(0,\epsilon)$ denotes a ball of radius $\epsilon$ around the origin. Equivalently \cite{kwasnicki_ten_2017} with range of validity $\gamma \in (0,d)$ we can write the fractional Laplacian as the inverse of the Riesz potential, thus denoted $(-\Delta)^{-\frac{\gamma}{2}}$:
\begin{align*}
(-\Delta)^{-\frac{\gamma}{2}}f(x)=\frac{\Gamma(\frac{d-\gamma}{2})}{\pi^{\frac{d}{2}}2^\gamma\Gamma(\frac{\gamma}{2})}\int_{\mathbb{R}^d} \frac{f(x-y)}{|y|^{d-\gamma}}\mathrm{d}y=\frac{\Gamma(\frac{d-\gamma}{2})}{\pi^{\frac{d}{2}}2^\gamma\Gamma(\frac{\gamma}{2})}\int_{\mathbb{R}^d} \frac{f(y)}{|x-y|^{d-\gamma}}\mathrm{d}y.
\end{align*}
\end{definition}
Kwaśnicki recently published a mostly self-contained and comprehensive survey of many equivalent definitions of the fractional Laplacian in \cite{kwasnicki_ten_2017}. The Riesz potential will be introduced and discussed in more detail in section \ref{sec:riesztheory}. Without stating the construction explicitly, we note that the definition of the fractional Laplace operator as a singular integral can be extended for higher parameters $\gamma \in (0,k)$ where $k$ is an even integer using certain centered differences of $f$, see \cite{stein_singular_1970,dyda_fractional_2017}.

As we will primarily be interested in functions $f(x)=f(r)$ with $r=|x|$ with $\mathrm{supp}(f)=B_R \subset \mathbb{R}^d$, we now shift focus from the whole space to ball domains and rotationally symmetric functions. For the ordinary Laplace operator $\Delta$ acting on a radially symmetric function one has
\begin{align}\label{eq:radiallaplace}
\Delta f(r) &= \frac{1}{r^{d-1}}\frac{\D}{\D r}\left(r^{d-1}f'(r)\right) \nonumber \\&=\frac{1}{r^{d-1}} \left( (d-1)r^{d-2} f'(r) + r^{d-1} f''(r) \right)  \\&=\frac{d-1}{r}f'(r) + f''(r).\nonumber
\end{align}
A closely related formula for the negative fractional Laplacian is known but due to non-locality involves several non-trivial integrals, see \cite[Lemma 7.1]{garofalo_fractional_2018}. Dyda, Kuznetsov and Kwaśnicki recently made substantial progress in the theory of fractional Laplacians in \cite{dyda_fractional_2017}, in particular also for ball domains. They proved the following result valid for $\gamma > -d$ and $x \in \mathbb{R}^d$:
\begin{align}
\nonumber &\text{If} \quad f(x):=V(x) G^{mn}_{pq} ( 
\,|x|^2; \mathbf{a}; \mathbf{b}), \quad \text{then}\\
&(-\Delta)^{\frac{\gamma}{2}} f(x)=2^{\gamma} V(x) G^{m+1,n+1}_{p+2,q+2}\Big( 
\,|x|^2;\begin{matrix}
1-\tfrac{d + 2l + \gamma}{2}, & \mathbf{a}-\tfrac{\gamma}{2}, & -\tfrac{\gamma}{2} \\ 0, & \mathbf{b}-\tfrac{\gamma}{2}, &1-\tfrac{d + 2l}{2}
\end{matrix}\,
\Big ),
\end{align}
where $G$ is the so-called Meijer G-function \cite[16.17.1]{nist_2018} defined as a significant further generalization of the generalized hypergeometric functions ${}_p F_q$ and $V(x)$ is a solid harmonic polynomial of order $l$, i.e. a solution to Laplace's equation which is polynomial in $|x|$. In particular, if we set $V(x)=1$ then $l=0$. This very general formula can be used to derive several more specialized results. Of interest to us is the following statement which holds true for the unit ball $x \in B_1$ for certain parameter ranges for $\gamma$ \cite[Theorem 3]{dyda_fractional_2017}:
\begin{align}
\nonumber &\text{If} \quad f(x) = (1 - |x|^2)^{\frac{\gamma}{2}} V(x) P^{(\frac{\gamma}{2},\frac{d}{2}+l-1)}_n(2 |x|^2 - 1), \quad \text{then}\\
 &(-\Delta )^{\frac{\gamma}{2}} f(x) = \frac{2^\gamma \Gamma (1 + \tfrac{\gamma }{2} + n) \Gamma (\tfrac{d + 2 l + \gamma }{2} + n)}{n! \, \Gamma (\tfrac{d + 2l}{2} + n)} \, V(x) P^{(\frac{\gamma}{2},\frac{d}{2}+l-1)}_n(2 |x|^2 - 1) \label{eq:eigenjacobi}, 
\end{align}
with $V(x)$ as above and where $P_n^{(a,b)}(x)$ are the Jacobi polynomials. While the authors in \cite{dyda_fractional_2017} state the range of validity of this derived formula as $\gamma >0$, their proof of it can be modified to also work when $\gamma \in (-2,0)$. As we will be needing this additional range we present a proof of this fact on the basis of the general results in \cite{dyda_fractional_2017} in Appendix \ref{app:dydaformula}. Note that \eqref{eq:eigenjacobi} implies that the Jacobi polynomials with the stated parameter choices form a complete orthogonal basis of eigenfunctions for the fractional Laplacian $(-\Delta )^{\frac{\gamma}{2}}$ weighted by $(1 - |x|^2)^{\frac{\gamma}{2}}$.
\subsection{Riesz and power law potentials on balls}\label{sec:riesztheory}
Riesz potentials pervade the wider mathematical context of equilibrium measures, with many results in Riesz potential theory going back to the original study of these constructions by the eponymous M. Riesz \cite{riesz1938integrales,arectification}. Some standard surveys which include modern and historical aspects are \cite{rafeiro_fractional_2016,landkof_foundations_1972,samko_fractional_1993,adams_function_1999}.\\
In the context of mathematical potential theory, a potential $U_K^\rho(x)$ is generally an integral operator
\begin{align*}
U_K^\rho(x) = \int_{\mathbb{R}^d}K(x,y)\mathrm{d}\rho(y) = \int_{\mathrm{supp}(\rho)}K(x,y)\mathrm{d}\rho(y)
\end{align*}
over the support of a measure $\rho$, with a given kernel $K(x,y)$ and where $\mathrm{supp}(\rho) \subseteq \mathbb{R}^d$. When $\mathrm{d}\rho(y) = \rho(y)\mathrm{d}y$ we will instead write it as acting on the density, i.e.
\begin{align*}
U_K[\rho(y)](x) = \int_{\mathbb{R}^d}K(x,y)\rho(y)\mathrm{d}y = \int_{\mathrm{supp}(\rho)}K(x,y)\rho(y)\mathrm{d}y.
\end{align*}
In many applications what one is primarily interested in are so-called convolution kernels $K(x,y) = K(x-y)$ or even more specifically kernels which depend only on the Euclidean distance $K(x,y)=K(|x-y|)$. The Riesz potential is of this kind and includes the well-known logarithmic and classical Newtonian kernels as limiting and special cases respectively, cf. \cite{landkof_foundations_1972}.
\begin{definition}[Riesz kernel]\label{def:rieszkernel}
The Riesz kernel $k_\gamma(|x-y|)$, with $x,y \in \mathbb{R}^d$ is
\begin{align*}
k_\gamma(|x-y|) = \frac{\Gamma(\frac{d-\gamma}{2})}{\pi^{\frac{d}{2}}2^\gamma\Gamma(\frac{\gamma}{2})}\frac{1}{|x-y|^{d-\gamma}}.
\end{align*}
\end{definition}
\begin{definition}[Riesz potential] \label{def:rieszpotential} The Riesz potential of a function $\rho(y)$ with $y \in \mathrm{supp}(\rho) \subseteq \mathbb{R}^d$ is defined as the integral
\begin{align*}
U_{k_\gamma}\rho(x) = (-\Delta )^{-\frac{\gamma}{2}}\rho(x) =  \frac{\Gamma(\frac{d-\gamma}{2})}{\pi^{\frac{d}{2}}2^\gamma\Gamma(\frac{\gamma}{2})}\int_{\mathrm{supp}(\rho)}\frac{1}{|x-y|^{d-\gamma}}\rho(y)\mathrm{d}y.
\end{align*}
\end{definition}
We have already seen in section \ref{sec:laplace} that the Riesz potential is the inverse of the negative fractional Laplacian given that $\gamma \in (0,d)$ \cite{adams_function_1999,kwasnicki_ten_2017}.
The normalization constant in the Riesz kernel is chosen such that the following semi-group property holds for Riesz potentials given the condition $(\gamma+\beta),\gamma, \beta \in (0,d)$, cf. \cite{landkof_foundations_1972}:
\begin{align}\label{eq:semigroup}
U_{k_{\gamma+\beta}}\rho = U_{k_\gamma}U_{k_\beta}\rho,
\end{align}
or alternatively in the notation of the fractional Laplace operator:
\begin{align*}
(-\Delta )^{-\frac{\gamma+\beta}{2}}\rho(x) = (-\Delta )^{-\frac{\gamma}{2}} (-\Delta )^{-\frac{\beta}{2}}\rho(x) = (-\Delta )^{-\frac{\beta}{2}} (-\Delta )^{-\frac{\gamma}{2}} \rho(x).
\end{align*}
Most of the standard results for Riesz potentials, such as those found in \cite{landkof_foundations_1972,samko_fractional_1993}, are only valid for parameters $\gamma \in (0,d)$. As with the fractional Laplacian we have started with definitions and properties valid for the whole space $\mathbb{R}^d$. In the context of equilibrium measures we are interested in Riesz pontentials on balls of radius $R$, i.e. $B_R = \{ x\in \mathbb{R}^d, |x|\leq R \}$ or equivalently with measures where $\mathrm{supp}(\rho)=B_R$. Such potentials have several unique properties compared to those on the whole space and were extensively covered in the seminal works of Rubin and Samko, see e.g. \cite{rubin_fractional_1996,rubin_one-dimensional_1983,samko_fractional_1993,rogosin_role_2013}. Of particular note is \cite{rubin_one-dimensional_1983}, in which Rubin showed that for $\gamma \in (0,d)$ and radially symmetric functions the Riesz potential on a ball of arbitrary dimension can be reduced to two nested one-dimensional fractional Riemann--Liouville integrals and \cite{rubin_fractional_1994,rubin_fractional_1996} in which he introduced two new types of fractional integrals on balls, c.f. \cite{edmunds_ball_2002,samko_fractional_1993}, the composition of which yields the Riesz potential.\\ The power law kernels we are interested in for the purposes of equilibrium measures are Riesz kernels on balls but in the equilibrium measure literature, c.f. \cite{carrillo_explicit_2017,gutleb_computing_2020}, the following notation with different normalization is sometimes preferred and will also be used in this paper henceforth:
\begin{definition}[Power law kernel] \label{def:powerlawkernel}
We define the power law kernel $K^\alpha(|x-y|)$, with $x,y \in \mathbb{R}^d$ by
\begin{align*}
K^\alpha(|x-y|) = \frac{|x-y|^{\alpha}}{\alpha}.
\end{align*}
\end{definition}
\begin{definition}[Power law potential] \label{def:powerlawpotential} The power law potential of a function $\rho(y)$ with $y \in \mathrm{supp}(\rho) \subseteq \mathbb{R}^d$ is defined as the integral
\begin{align*}
U^{\alpha}[\rho(y)](x) =  \frac{1}{\alpha}\int_{\mathrm{supp}(\rho)}|x-y|^{\alpha}\rho(y)\mathrm{d}y.
\end{align*}
\end{definition}
\begin{remark}
The power law kernel as defined in Definition \ref{def:powerlawkernel} corresponds to the Riesz potential in Definition \ref{def:rieszkernel} when $\alpha=\gamma-d$ with different normalization. Note that in contrast to the conventional Riesz potential the normalization in Definition \ref{def:powerlawkernel} is chosen such that a kernel $K^\alpha(|x-y|)$ will always result in an attractive potential while a negative sign $-K^\alpha(|x-y|)$ will result in a repulsive potential, even when the sign of $\alpha$ changes.
\end{remark}
Simple variable substitutions and multiplications with the appropriate constants allow for the conversion of all valid statements for Riesz potentials when $\gamma-d = \alpha \in (-d,0)$. As many interesting equilibrium measure phenomena occur exlusively in high parameter regimes, compare the high parameter results in \cite{carrillo_explicit_2017,gutleb_computing_2020}, we want to loosen the restriction of $\alpha \in (-d,0)$ or equivalently $\gamma \in (0,d)$ but whether this can be done depends fundamentally on the function the potential acts on. We will see important special cases in which this can be done in the next section.

\section{Results for power law potential of radial functions}\label{sec:resultsandproofs}
In \cite{gutleb_computing_2020} the authors derived recurrence relations for left and right-handed Riemann--Liouville fractional integrals in weighted ultraspherical polynomial bases and used them to prove a recurrence relationship for power law kernel integrals in the one-dimensional case. A natural higher dimensional analogue of Riemann--Liouville fractional integrals is found in the Riesz potential. Attempting a straightforward higher dimensional version of the approach in \cite{gutleb_computing_2020} fails because the Riemann--Liouville integrals satisfy properties, e.g. the semi-group property in \eqref{eq:semigroup}, which for Riesz potentials are only valid in very limited parameter ranges. In this section we thus take a different approach, which nevertheless leads to a true generalization of the results in \cite{gutleb_computing_2020} to arbitrary dimension $d$. Throughout this section, we write $$B(x,y) = \frac{\Gamma(x)\Gamma(y)}{\Gamma(x+y)}$$ to denote the Beta function and $B_R = \{ x \in \mathbb{R}^d, |x|\leq R \}$ for a $d$-dimensional ball of radius $R$ centered at the origin. A straightforward rescaling allows us to set the radius of the ball domain to $R = 1$ without loss of generality, which significantly simplifies our notation when dealing with Jacobi polynomials.
\subsection{Auxiliary results}
The purpose of this section is to build and gather the mathematical tools needed to prove our main theorems. The following finite sum expression for hypergeometric ${}_pF_q$ functions can be found in the extensive works of Prudnikov, Brychkov and Marichev \cite[5.3.6.3]{prudnikov2003integrals}:
\begin{align*}
\sum_{k=0}^n (-1)^k &\binom{n}{k} \frac{(1-\mathfrak{b})_k}{(1-\mathfrak{a})_k} {}_{p+m}F_q\left(\begin{matrix}
\mathbf{a}_p,  \quad \Delta(m,\mathfrak{a}-k)\\
\mathbf{b}_q
\end{matrix} ; x \right)  \\
&=\frac{(\mathfrak{b}-\mathfrak{a})_n}{(1-\mathfrak{a})_n} {}_{p+2m}F_{q+m}\left(\begin{matrix}
\mathbf{a}_p, \quad \Delta(m,\mathfrak{b}-\mathfrak{a}+n), \quad \Delta(m,\mathfrak{a}-n)\\
\mathbf{b}_q,  \quad \Delta(m,\mathfrak{a}-\mathfrak{b}-n+1)
\end{matrix}; x\right)
\end{align*}
where $\Delta(k, \mathfrak{a}) := \frac{\mathfrak{a}}{k}, \frac{\mathfrak{a}+1}{k}, ..., \frac{\mathfrak{a}+k-1}{k}$, cf \cite[p.798]{prudnikov2003integrals}. The proof of Theorem \ref{theorem:hypergeomgeneral} will be using the case $p=q=m=1$, which simplifies to:
\begin{align}\label{eq:prudnikovsum}
\sum_{k=0}^n (-1)^k \binom{n}{k} &\frac{(1-\mathfrak{b})_k}{(1-\mathfrak{a})_k} {}_{2}F_1\left( \begin{matrix} a_1,\quad \mathfrak{a}-k\\b_1 \end{matrix};x\right) \\ &= \frac{(\mathfrak{b}-\mathfrak{a})_n}{(1-\mathfrak{a})_n} {}_{3}F_{2}\left(\begin{matrix}a_1,\quad \mathfrak{b}-\mathfrak{a}+n, \quad\mathfrak{a}-n\\ b_1,\quad \mathfrak{a}-\mathfrak{b}-n+1\end{matrix}; x\right).\nonumber
\end{align} 
Next, we discuss some important lemmas which follow from results in potential theory. The following general power law integral with $x=0$ will be useful in determining certain integration constants and previously been used in the derivation of certain analytical power law equilibrium measure results, cf. \cite[Appendix A]{carrillo_explicit_2017}:
\begin{lemma}\label{lemma:ballmonomialx0}
On $B_R$ the power law potential of $(R^2-|y|^2)^{m} |y|^k$ at $x=0$ evaluates:
\begin{align*}
\int_{B_R} (R^2-|y|^2)^{m} |y|^k \mathrm{d}y = R^{k+2m+d}\frac{\pi^{\frac{d}{2}}}{\Gamma(\frac{d}{2})}B\left(\frac{k+d}{2},m+1\right).
\end{align*}
\end{lemma}
\begin{proof}
The $d=1$ variant of this formula may be obtained after splitting the expression into two integrals from $0$ to $R$ and $-R$ to $0$ respectively, then using the general formula \cite[3.3.2.4]{prudnikov1986integrals}:
\begin{align*}
\int_0^R (R^\mu-x^\mu)^{\beta-1} x^{\alpha-1}\dx = R^{\mu(\beta-1)+\alpha}\mu^{-1} B\left(\frac{\alpha}{\mu},\beta\right).
\end{align*}
The higher dimensional variant then follows via a reduction formula to the one-dimensional case \cite[3.3.2.4]{prudnikov1986integrals}.
\end{proof}
We prove a generalization of Lemma \ref{lemma:ballmonomialx0} for Jacobi polynomials on the unit ball.
\begin{lemma}\label{lemma:balljacobix0}
On the $d$-dimensional unit ball $B_1$ the power law potential of the weighted Jacobi polynomial $(1-|y|^2)^{m} P_n^{(a,b)}(2|y|^2-1)$ at $x=0$ evaluates:
\begin{align*}
\int_{B_1} (1-|y|^2)^{m} |y|^p  &P_n^{(a,b)}(2|y|^2-1) \mathrm{d}y \\ &=\sum_{k=0}^{n}(-1)^{n+k}\tfrac{{\left(n+a+%
b+1\right)_{k}}{\left(b+k+1\right)_{n-k}}}{k!\;(n-k)!} \frac{\pi^{\frac{d}{2}}}{\Gamma(\frac{d}{2})}B\left(\frac{2k+p+d}{2},m+1\right)\\
&= \tfrac{(-1)^n\pi ^{\frac{d}{2}+1} (\alpha +d) \Gamma \left(\frac{d}{2}+n\right)}{2 \Gamma \left(\frac{d}{2}\right)^2 \Gamma (n+1) \sin \left(\frac{\pi(\alpha +d)}{2}   \right)}{}_3F_2\left(\begin{matrix}-n,\quad n-\frac{\alpha }{2},\quad 1+\frac{\alpha+d}{2}\\2,\quad \frac{d}{2}\end{matrix};1\right).
\end{align*}
\end{lemma}
\begin{proof}
Via the explicit representation of Jacobi polynomials in \eqref{eq:radialseriesrep} we can expand this integral into a form from which the result follows via Lemma \ref{lemma:ballmonomialx0}.
\begin{align*}
\int_{B_1} (1-|y|^2)^{m} |y|^p &P_n^{(a,b)}(2|y|^2-1) \mathrm{d}y \\&=\sum_{k=0}^{n}(-1)^{n+k}\tfrac{{\left(n+a+%
b+1\right)_{k}}{\left(b+k+1\right)_{n-k}}}{k!\;(n-k)!} \int_{B_1} (1-|y|^2)^{m} |y|^{2k+p} \mathrm{d}y \\ &= \sum_{k=0}^{n}(-1)^{n+k}\tfrac{{\left(n+a+%
b+1\right)_{k}}{\left(b+k+1\right)_{n-k}}}{k!\;(n-k)!} \frac{\pi^{\frac{d}{2}}}{\Gamma(\frac{d}{2})}B\left(\frac{2k+p+d}{2},m+1\right).
\end{align*}
The hypergeometric function variant then follows by definition of the beta and $_3 F_2$ hypergeometric function, see e.g. the properties in \cite[5.12, 16.2]{nist_2018}.
\end{proof}
The following lemma is fundamental for the method introduced in this paper.
\begin{lemma}\label{lemma1}
On $B_R$ the power law potential of the function $(R^2-|y|^2)^{-\frac{\alpha+d}{2}}$, with power $\alpha \in (-d,2-d)$, evaluates to an explicitly known constant: 
\begin{align*}
\int_{B_R} \left\lvert x-y \right\rvert^\alpha (R^2-|y|^2)^{-\frac{\alpha+d}{2}} \mathrm{d}y =\frac{\pi^{\frac{d}{2}+1}}{\Gamma(\frac{d}{2})\sin\left(\frac{(\alpha+d)\pi}{2}\right)}.
\end{align*}
\end{lemma}
\begin{proof}
A proof of this result based on point inversion (Kelvin transforms) can be found in \cite[Appendix]{landkof_foundations_1972}.
\end{proof}
The next lemma is a direct generalization of Lemma \ref{lemma1}. While the general case proof of this generalization requires some additional work and we thus only provide a reference to its proof, it is nevertheless worthwhile for our later results to sketch how incremental generalizations of Lemma \ref{lemma1} can be obtained as was discussed in \cite{carrillo_explicit_2017}. Keeping $\alpha \in (-d,2-d)$, we note the action of the ordinary Laplace operator $\Delta$ on a power law integral with increased power $\alpha+2$:
\begin{align*}
\Delta \int_{B_R} \left\lvert x-y \right\rvert^{\alpha+2} (R^2-|y|^2)^{-\frac{\alpha+d}{2}} \mathrm{d}y &= (\alpha+d)(\alpha+2) \int_{B_R} \left\lvert x-y \right\rvert^{\alpha} (R^2-|y|^2)^{-\frac{\alpha+d}{2}} \mathrm{d}y
\\&= \frac{(\alpha+d)(\alpha+2)\pi^{\frac{d}{2}+1}}{\Gamma(\frac{d}{2})\sin\left(\frac{(\alpha+d)\pi}{2}\right)}.
\end{align*}
Inverting the Laplace operator here is straightforward and leads to
\begin{align*}
 \int_{B_R} \left\lvert x-y \right\rvert^{\alpha+2} (R^2-|y|^2)^{-\frac{\alpha+d}{2}} \mathrm{d}y = \frac{(\alpha+d)\pi^{\frac{d}{2}+1}}{\Gamma(\frac{d}{2})\sin\left(\frac{(\alpha+d)\pi}{2}\right)}\left(\frac{(\alpha+2)}{2d}|x|^2+ \frac{R^2}{2} \right),
\end{align*}
where the integration constant was fixed by evaluation at $x=0$ using Lemma \ref{lemma:ballmonomialx0} above. Finally replacing $\alpha+2$ with $\alpha$ to retain the same kernel structure leads us to the following result which is now valid for $\alpha \in (2-d,4-d)$:
\begin{align*}
\int_{B_R} \left\lvert x-y \right\rvert^{\alpha} (R^2-|y|^2)^{1-\frac{\alpha+d}{2}} \mathrm{d}y &= -\frac{(\alpha+d-2)\pi^{\frac{d}{2}+1}}{\Gamma(\frac{d}{2})\sin\left(\frac{(\alpha+d)\pi}{2}\right)}\left(\frac{\alpha}{2d}|x|^2+ \frac{R^2}{2} \right).
\end{align*}
Repeating this process we can successively find solutions for higher powers. One finds that these are all special cases of the following general lemma:
\begin{lemma}\label{lemma2}
On $B_R$ the power law potential of the function $(R^2-|y|^2)^{\ell-\frac{\alpha+d}{2}}$, with power $\alpha \in (-d,2+2\ell-d)$, evaluates as follows:
\begin{align*}
\int_{B_R} \left\lvert x-y \right\rvert^\alpha (R^2-|y|^2)^{\ell-\frac{\alpha+d}{2}} \mathrm{d}y &= \frac{\pi^{\frac{d}{2}}R^{2\ell}}{\Gamma\left(\tfrac{d}{2}\right)}B\left(\tfrac{\alpha+d}{2},\tfrac{2\ell+2-\alpha-d}{2}\right){}_2F_1\left(\begin{matrix}-\tfrac{\alpha}{2},\quad -\ell,\\ \tfrac{d}{2}\end{matrix};\tfrac{|x|^2}{R^2}\right).
\end{align*}
If furthermore $\ell \in \mathbb{N}_0$, this hypergeometric function reduces to a polynomial. Explicitly:
\begin{align*}
\int_{B_R} \left\lvert x-y \right\rvert^\alpha (R^2-|y|^2)^{\ell-\frac{\alpha+d}{2}} \mathrm{d}y &= \tfrac{(-1)^\ell \pi\Gamma \left(1+\frac{\alpha }{2}\right) \Gamma \left(\ell-\frac{\alpha }{2}\right)}{\left(\frac{d}{2}\right)_\ell} P_\ell^{(\ell-\frac{\alpha+d}{2},\frac{d-2}{2})}\left(2\frac{|x|^2}{R^2}-1\right).
\end{align*}
\end{lemma}
\begin{proof}
This generalization of Lemma \ref{lemma1} was derived in \cite{huang_explicit_2014,biler_nonlocal_2015,biler_barenblatt_2011} on the basis of the above-mentioned connection to the theory of fractional Laplace operators. The Jacobi polynomial variant with order dependence in the parameter follows from their explicit hypergeometric function expression in \eqref{eq:jacobihyper} combined with \eqref{eq:jacobisymm}.
\end{proof}
The fact that Lemma \ref{lemma2} does not require $\ell \in \mathbb{Z}$ means that we can separate the power of the weight from that of the kernel, a fact we make explicit in the following corollary:
\begin{corollary}\label{corr:colemma2}
On $B_R$ the power law potential of the function $(R^2-|y|^2)^{m-\frac{\alpha+d}{2}}$ with power $\alpha \in(-d,2+2m-d)$, $m\in\mathbb{N}_0$ and $\beta>-d$, evaluates as follows:
\begin{align*}
\int_{B_R} \left\lvert x-y \right\rvert^\beta &(R^2-|y|^2)^{m-\frac{\alpha+d}{2}} \mathrm{d}y =\\ &\frac{\pi^{\frac{d}{2}}R^{2m+\beta-\alpha}}{\Gamma\left(\tfrac{d}{2}\right)}B\left(\tfrac{\beta+d}{2},\tfrac{2m+2-\alpha-d}{2}\right){}_2F_1\left(\begin{matrix}-\tfrac{\beta}{2},\quad -m-\tfrac{\beta-\alpha}{2} \\ \tfrac{d}{2}\end{matrix}; \tfrac{|x|^2}{R^2}\right).
\end{align*}
\begin{proof}
Starting with $\int_{B_R} \left\lvert x-y \right\rvert^\beta (R^2-|y|^2)^{\ell-\frac{\beta+d}{2}} \mathrm{d}y$ and setting $\ell=m+\frac{\beta-\alpha}{2}$, the result follows directly from Lemma \ref{lemma2}.
\end{proof}
\end{corollary}
\subsection{Laplacians of shifted Jacobi polynomials}
This section collects further useful lemmas related to recurrences to compute the ordinary Laplace operator of various radial hypergeometric functions, including shifted Jacobi polynomials.
\begin{lemma}\label{lemma:laplacejacobi}
The Laplace operator $\Delta$ acting on the shifted Jacobi polynomials $P_n^{(a,b)}(2|y|^2-1)$ on the $d$-dimensional unit ball, $y \in B_R = \{ x\in \mathbb{R}^d, |x|\leq R \}$, evaluates:
\begin{align*}
\Delta P_n^{(a,b)}(2|y|^2-1) &= 2d (a+b+n+1) P_{n-1}^{(a+1,b+1)}\left(2 |y|^2-1\right)\\
&+4 (a+b+n+1)(a+b+n+2) |y|^2 P_{n-2}^{(a+2,b+2)}\left(2 |y|^2-1\right).
\end{align*}
\end{lemma}
\begin{proof}
This is easily obtained via the representation for the Laplace operator on radial functions in \eqref{eq:radiallaplace} in combination with the derivative relations for Jacobi polynomials in (\ref{eq:derivativesjac1}--\ref{eq:derivativesjac2}).
\end{proof}
Note that the terms on the right-hand side of Lemma \ref{lemma:laplacejacobi} can be transformed into a basis with consistent parameters by using shift operators such as those listed in (\ref{eq:basisconversionjacobi1}--\ref{eq:basisconversionjacobi2}). Unfortunately, when doing so the expressions do not generally simplify further and end up involving several Jacobi polynomials of different polynomial orders. That being said, we can obtain significant further simplification for certain choices of $a$ and $b$. Specifically, this is the case for the radially symmetric Jacobi polynonimals with basis parameters as seen in section \ref{sec:spectralondisk}.
\begin{lemma}\label{lemma:diagjacobilaplace}
The Laplace operator $\Delta$ acting on the shifted Jacobi polynomials $P_n^{(a,\frac{d-2}{2})}(2|y|^2-1)$ on the $d$-dimensional unit ball $B_1 = \{ x\in \mathbb{R}^d, |x|\leq 1 \}$ evaluates:
\begin{align*}
\Delta P_{n}^{\left(a,\frac{d-2}{2}\right)}\left(2 |y|^2-1\right) &= (2n+d-2) (2n+2a+d) P_{n-1}^{(a+2,\frac{d-2}{2})}(2|y|^2-1).
\end{align*}
\end{lemma}
\begin{proof}
This can be obtained from Lemma \ref{lemma:laplacejacobi} via (\ref{eq:basisconversionjacobi1}--\ref{eq:basisconversionjacobi2}).
\end{proof}
For the case $a=0$ and $d=2$, i.e. for the case of ordinary Zernike polynomials on the unit disk, similar results to Lemma \ref{lemma:diagjacobilaplace} for the Laplacian and inverse Laplacian were discussed in \cite{janssen_zernike_2014} while also including their angular components.\\

\subsection{Power law potentials of Jacobi polynomials on the unit ball}\label{sec:subsecpowerbanded}
With all of the above results, we are now ready to investigate the action of the ball Riesz potential operator on radially symmetric functions expanded in orthogonal polynomials on arbitrary dimensional balls, where the radial part is expanded in terms of one-dimensional Jacobi polynomials $P_n^{(a,b)}(2r^2-1)$ with $r^2=|y|^2$ as discussed in sections \ref{sec:jacobi} and \ref{sec:spectralondisk}. First we prove the following result in low parameter ranges:
\begin{theorem}\label{theorem:diagonalformula}
On the $d$-dimensional unit ball $B_1 = \{ x \in \mathbb{R}^d, |x|\leq 1 \}$ the power law potential of the $n$-th radial Jacobi polynomial $P_n^{\left(-\frac{\alpha+d}{2},\frac{d-2}{2}\right)}(2|y|^2-1)$ weighted with $(1-|y|^2)^{-\frac{\alpha+d}{2}}$ evaluates explicitly as follows in the parameter range $\alpha \in (-d,2-d)$:
\begin{align*}
\int_{B_1}|x-y|^\alpha &(1-|y|^2)^{-\frac{\alpha+d}{2}} P_n^{(-\frac{\alpha+d}{2},\frac{d-2}{2})}(2|y|^2-1) \mathrm{d}y\\ 
&= \tfrac{ \pi ^{\frac{d}{2}} \Gamma \left(\frac{\alpha+d}{2}\right) \Gamma \left(n-\frac{\alpha }{2}\right) \Gamma \left(1 - \frac{\alpha+d }{2} + n\right)}{\Gamma \left(-\frac{\alpha }{2}\right) \Gamma \left(\frac{d}{2}+n\right) n!}P_n^{(-\frac{\alpha+d}{2},\frac{d-2}{2})}(2|x|^2-1)\\ 
&=-\tfrac{\alpha  \pi ^{\frac{d}{2}+1} \left(1-\frac{\alpha }{2}\right)_{n-1} \left(1-\frac{\alpha+d}{2}\right)_n}{2 \sin \left(\frac{\pi  (\alpha +d)}{2} \right) \Gamma \left(\frac{d}{2}+n\right) n!}P_n^{(-\frac{\alpha+d}{2},\frac{d-2}{2})}(2|x|^2-1).
\end{align*}
\end{theorem}
\begin{proof}
There are a number of ways to derive this result from what we have so far discussed, the most straightforward of which is to use results based on those of Dyda, Kuznetsov and Kwaśnicki \cite[Theorem 3]{dyda_fractional_2017} which we reproduced in \eqref{eq:eigenjacobi} and extended in Appendix \ref{app:dydaformula}. When $\gamma \in (0,d)$ the Riesz potential is the inverse of the fractional Laplace operator, motivating the use of \eqref{eq:eigenjacobi} to find a related result for the Riesz potential. With $V(x)=1$ and thus $l=0$ the expression reads
\begin{align*}
(-\Delta )^{\frac{\gamma}{2}} (1 - |x|^2)^{\frac{\gamma}{2}} &P^{(\frac{\gamma}{2},\frac{d-2}{2})}_n(2 |x|^2 - 1) \\&= \tfrac{2^\gamma \Gamma \left(1 + \frac{\gamma }{2} + n\right) \Gamma \left(\frac{d + \gamma }{2} + n\right)}{\, \Gamma \left(\frac{d}{2} + n\right)n!} \, P^{(\frac{\gamma}{2},\frac{d-2}{2})}_n(2 |x|^2 - 1).
\end{align*}
Setting $\gamma = -(\alpha+d)$ then leads to
\begin{align*}
(-\Delta )^{-\frac{\alpha+d}{2}} (1 - |x|^2)^{-\frac{\alpha+d}{2}} &P^{(-\frac{\alpha+d}{2},\frac{d-2}{2})}_n(2 |x|^2 - 1) \\&= \tfrac{ \Gamma \left(1 - \frac{\alpha+d }{2} + n\right) \Gamma \left(n-\frac{\alpha }{2}\right)}{ 2^{\alpha+d} \, \Gamma \left(\frac{d}{2} + n\right)n!} \, P^{(-\frac{\alpha+d}{2},\frac{d-2}{2})}_n(2 |x|^2 - 1).
\end{align*}
The inherited validity of this expression is $\alpha \in (-d,2-d)$, see Appendix \ref{app:dydaformula}. The above result holds for Riesz potentials, which as we saw in section \ref{sec:riesztheory} has different normalization constants to the equivalent power law potentials we are working with here. Taking a glance at Definitions \ref{def:fraclaplace} and \ref{def:rieszpotential}, this difference can be accounted for by simply multiplying both sides in the above equation by $\tfrac{  \pi^{\frac{d}{2}} 2^{\alpha+d}\Gamma(\frac{\alpha+d}{2})}{\Gamma(-\frac{\alpha}{2})}$. Carrying out the appropriate cancellations we are then left with the first stated result.\\
To obtain the second variation from the previous result, note that
\begin{align*}
\tfrac{ \pi ^{\frac{d}{2}} \Gamma \left(\frac{\alpha+d}{2}\right) \Gamma \left(n-\frac{\alpha }{2}\right) \Gamma \left(1 - \frac{\alpha+d }{2} + n\right)}{\Gamma \left(-\frac{\alpha }{2}\right) \Gamma \left(\frac{d}{2}+n\right) n!}=-\tfrac{\alpha  \pi ^{\frac{d}{2}} B\left(\frac{\alpha+d}{2},1-\frac{\alpha+d}{2}\right) \left(1-\frac{\alpha }{2}\right)_{n-1} \left(1-\frac{\alpha+d}{2}\right)_n}{2 \Gamma \left(\frac{d}{2}+n\right) n!},
\end{align*}
which can easily be verified by expanding the beta function and Pochhammer symbols into their representations in terms of Gamma functions, see \cite[5.2.5,5.12.1]{nist_2018}. The second variation then immediately follows from the fact \cite[5.5.3]{nist_2018} that
\begin{align*}
B\left(\tfrac{\alpha+d}{2},1-\tfrac{\alpha+d}{2}\right) = \tfrac{\pi}{ \sin \left(\frac{\pi  (\alpha +d)}{2} \right)}.
\end{align*}
Note that this is the same factor that already appeared in Lemma \ref{lemma1}.
\end{proof}
\begin{corollary}\label{corr:diagonality}
On the $d$-dimensional unit ball $B_1 = \{ x \in \mathbb{R}^d, |x|\leq 1 \}$ the power law potential of $(1-|y|^2)^{-\frac{\alpha+d}{2}}P_{n+1}^{\left(-\frac{\alpha+d}{2},\frac{d-2}{2}\right)}(2|y|^2-1)$ can be computed using a three term recurrence relationship:
\begin{align*}
\int_{B_1}|x-y|^\alpha &(1-|y|^2)^{-\frac{\alpha+d}{2}} P_{n+1}^{(-\frac{\alpha+d}{2},\frac{d-2}{2})}(2|y|^2-1) \mathrm{d}y =\\
((2|x|^2-1)\kappa_A + \kappa_B)&\int_{B_1}|x-y|^\alpha (1-|y|^2)^{-\frac{\alpha+d}{2}} P_n^{(-\frac{\alpha+d}{2},\frac{d-2}{2})}(2|y|^2-1) \mathrm{d}y \\-\kappa_C &\int_{B_1}|x-y|^\alpha (1-|y|^2)^{-\frac{\alpha+d}{2}} P_{n-1}^{(-\frac{\alpha+d}{2},\frac{d-2}{2})}(2|y|^2-1) \mathrm{d}y,
\end{align*}
where
\begin{align*}
&\kappa_A =-\tfrac{(4 n-\alpha ) (-\alpha +4 n+2) (\alpha +d-2 n-2)}{8 (n+1)^2 (d+2 n)},\\
&\kappa_B =-\tfrac{(\alpha +2) (\alpha +2 d-2) (4 n-\alpha ) (\alpha +d-2 n-2)}{8 (n+1)^2 (d+2 n) (-\alpha +4 n-2)},\\
&\kappa_C =-\tfrac{(-\alpha +2 n-2) (-\alpha +4 n+2) (\alpha +d-2 n-2) (\alpha +d-2 n)^2}{8 n (n+1)^2 (d+2 n) (-\alpha +4 n-2)}.
\end{align*}
\end{corollary}
\begin{proof}
This immediately follows from expanding the Jacobi polynomial on the right-hand side of the formula given in Theorem \ref{theorem:diagonalformula} using the classical three-term recurrence for Jacobi polynomials in \eqref{eq:jacobirecclassical}.
\end{proof}
\begin{remark}
We can easily convince ourselves that the diagonality result in Corollary \ref{corr:diagonality} cannot in general be true for higher parameters $\alpha>2-d$, i.e. that diagonal operators only exist for  $\alpha \in (-d,2-d)$. First, note that in order for solutions to exist when $\alpha>2-d$ we need the weight parameter $\ell$ to satisfy $\ell>0$ in Lemma \ref{lemma2}. Obviously one necessary condition for the operator to be diagonal is that the $0$-th order polynomial is mapped to a constant but the right-hand side of Lemma \ref{lemma2} tells us that this only occurs when $\ell=0$, as otherwise the hypergeometric function (or equivalently the Jacobi polynomial) is not constant with respect to $|x|$. Corollary \ref{corr:diagonality} covers precisely the case $\ell =0$.
\end{remark}
While the operators cannot be diagonal in higher parameter regimes, we can nevertheless obtain \emph{banded} operators with explicit elements which is similarly efficient for computing purposes. The following theorem shows how to obtain higher bandwidth operators for higher parameter ranges and is the fundamental result for our sparse spectral method:

\begin{theorem}\label{theorem:recurrencen1case}
On the $d$-dimensional unit ball $B_1 = \{ x \in \mathbb{R}^d, |x|\leq 1 \}$ the power law potential of the $n$-th radial Jacobi polynomial $P_n^{\left(1-\frac{\alpha+d}{2},\frac{d-2}{2}\right)}(2|y|^2-1)$ weighted with $(1-|y|^2)^{1-\frac{\alpha+d}{2}}$ evaluates as follows in the parameter range $\alpha \in (2-d,4-d)$:
\begin{align*}
\int_{B_1}|x-y|^\alpha (1-|y|^2)^{1-\frac{\alpha+d}{2}} P_n^{(1-\frac{\alpha+d}{2},\frac{d-2}{2})}(2|y|^2-1) \mathrm{d}y &=\kappa_a P_{n-1}^{\left(1-\frac{\alpha+d}{2},\frac{d-2}{2}\right)}\left(2 |x|^2-1\right)
\\&+\kappa_b P_{n}^{\left(1-\frac{\alpha+d}{2},\frac{d-2}{2}\right)}\left(2 |x|^2-1\right)
\\ &+\kappa_c P_{n+1}^{\left(1-\frac{\alpha+d}{2},\frac{d-2}{2}\right)}\left(2 |x|^2-1\right),
\end{align*}
where the constants are given by
\begin{align*}
&\kappa_a = -\tfrac{4 \pi ^{d/2} \Gamma \left(\frac{\alpha+d}{2}\right) \Gamma \left(n-\frac{\alpha}{2}\right) \Gamma \left(n-\frac{\alpha+d}{2}+2\right)}{(\alpha-4 n-2) (\alpha-4 n) \Gamma \left(-\frac{\alpha}{2}\right) \Gamma (n+1) \Gamma \left(\frac{d}{2}+n-1\right)},\\
&\kappa_b = \tfrac{8 \pi ^{d/2} \Gamma \left(\frac{\alpha+d}{2}\right) \Gamma \left(n-\frac{\alpha}{2}+1\right) \Gamma \left(n-\frac{\alpha+d}{2}+2\right)}{ (\alpha-4 n) (\alpha-4 (n+1)) \Gamma \left(-\frac{\alpha}{2}\right) \Gamma (n+1) \Gamma \left(\frac{d}{2}+n\right)},\\
&\kappa_c =-\tfrac{4 \pi ^{d/2} \Gamma \left(\frac{\alpha+d}{2}\right) \Gamma \left(n-\frac{\alpha}{2}+2\right) \Gamma \left(n-\frac{\alpha+d}{2}+2\right)}{(\alpha-4 n-2) (\alpha-4 (n+1)) \Gamma \left(-\frac{\alpha}{2}\right) \Gamma (n+1) \Gamma \left(\frac{d}{2}+n+1\right)}.
\end{align*}
\end{theorem}
\begin{proof}
First we note how the $d$-dimensional ordinary Laplace operator $\Delta_x$ in $x$ acts on power law integrals with increased power $\alpha+2$, where $\alpha\in(-d,2-d)$ as before:
\begin{align*}
\Delta_x \int_{B_1}&|x-y|^{\alpha+2} (1-|y|^2)^{-\frac{\alpha+d}{2}} P_{n}^{(-\frac{\alpha+d}{2},\frac{d-2}{2})}(2|y|^2-1) \mathrm{d}y = \\&(\alpha+d)(\alpha+2)\int_{B_1}|x-y|^{\alpha} (1-|y|^2)^{-\frac{\alpha+d}{2}} P_{n}^{(-\frac{\alpha+d}{2},\frac{d-2}{2})}(2|y|^2-1) \mathrm{d}y =\\
&-\tfrac{(\alpha+d)(\alpha+2)\alpha  \pi ^{\frac{d}{2}} B\left(\frac{\alpha+d}{2},1-\frac{\alpha+d}{2}\right) \left(1-\frac{\alpha }{2}\right)_{n-1} \left(1-\frac{\alpha+d}{2}\right)_n}{2 \Gamma \left(\frac{d}{2}+n\right) n!}P_n^{(-\frac{\alpha+d}{2},\frac{d-2}{2})}(2|x|^2-1).
\end{align*}
The final step makes use of Theorem \ref{theorem:diagonalformula}. Using Lemma \ref{lemma:diagjacobilaplace} to invert the Laplacian in this form is ill-advised, as it will reduce the first parameter of the Jacobi polynomial on the right-hand side by $2$ with no guarantee that the first Jacobi parameter will remain greater than $-1$. To avoid this happening, we first raise the basis parameters using the shift operators in (\ref{eq:basisconversionjacobi1}--\ref{eq:basisconversionjacobi2}). This first yields
\begin{align*}
P_n^{(-\frac{\alpha+d}{2},\frac{d-2}{2})}(2|x|^2-1) &= \tfrac{\alpha -2 n}{\alpha -4 n} P_n^{\left(1-\frac{\alpha+d}{2},\frac{d-2}{2}\right)}\left(2 |x|^2-1\right) \\&+\tfrac{d+2 n-2}{\alpha -4n}P_{n-1}^{\left(1-\frac{\alpha+d}{2},\frac{d-2}{2}\right)}\left(2 |x|^2-1\right),
\end{align*}
and in a second application
\begin{align*}
P_n^{(-\frac{\alpha+d}{2},\frac{d-2}{2})}(2|x|^2-1) = &\tfrac{(d+2 n-4) (d+2 n-2)}{(4 n-\alpha )(-\alpha +4 n-2)}P_{n-2}^{\left(2-\frac{\alpha+d}{2},\frac{d-2}{2}\right)}\left(2 |x|^2-1\right)
\\+&\tfrac{2 (d+2 n-2) (\alpha-2 n)}{(\alpha -4 n)^2-4}P_{n-1}^{\left(2-\frac{\alpha+d}{2},\frac{d-2}{2}\right)}\left(2 |x|^2-1\right)\\
+&\tfrac{(\alpha -2 n) (\alpha -2 n-2)}{(4 n-\alpha ) (-\alpha +4 n+2)}P_n^{\left(2-\frac{\alpha+d}{2},\frac{d-2}{2}\right)}\left(2 |x|^2-1\right).
\end{align*}
With this conversion and because of the radial symmetry of the integrals we can invert the Laplacian obtaining
\begin{align*}
\int_{B_1}|x-y|^{\alpha+2} (1-|y|^2)^{-\frac{\alpha+d}{2}} &P_{n}^{\left(-\frac{\alpha+d}{2},\frac{d-2}{2}\right)}(2|y|^2-1) \mathrm{d}y =\\ c_a &P_{n-1}^{\left(-\frac{\alpha+d}{2},\frac{d-2}{2}\right)}\left(2 |x|^2-1\right)
+c_b P_{n}^{\left(-\frac{\alpha+d}{2},\frac{d-2}{2}\right)}\left(2 |x|^2-1\right)
\\ +c_c &P_{n+1}^{\left(-\frac{\alpha+d}{2},\frac{d-2}{2}\right)}\left(2 |x|^2-1\right)
+ c_d,
\end{align*}
with constants
\begin{align*}
&c_a = -\tfrac{4 \pi ^{d/2} \Gamma \left(\frac{1}{2} (d+\alpha +2)\right) \Gamma \left(n-\frac{\alpha }{2}-1\right) \Gamma \left(n-\frac{\alpha+d}{2}+1\right)}{\Gamma \left(-\frac{\alpha }{2}-1\right) (-\alpha +4 n-2) (4 n-\alpha ) \Gamma (n+1) \Gamma \left(\frac{d}{2}+n-1\right)},\\
&c_b = \tfrac{8 \pi ^{d/2} \Gamma \left(\frac{1}{2} (d+\alpha +2)\right) \Gamma \left(n-\frac{\alpha }{2}\right) \Gamma \left(n-\frac{\alpha+d}{2}+1\right)}{ \Gamma \left(-\frac{\alpha }{2}-1\right) \left((\alpha -4 n)^2-4\right) \Gamma(n+1) \Gamma \left(\frac{d}{2}+n\right)},\\
&c_c = -\tfrac{4 \pi ^{d/2} \Gamma \left(\frac{1}{2} (d+\alpha +2)\right) \Gamma \left(n-\frac{\alpha }{2}+1\right) \Gamma \left(n-\frac{\alpha+d}{2}+1\right)}{\Gamma \left(-\frac{\alpha }{2}-1\right)  (-\alpha +4 n+2) (4 n-\alpha ) \Gamma (n+1) \Gamma \left(\frac{d}{2}+n+1\right)}.
\end{align*}
Finally, the constant term $c_d$ may be fixed by evaluating the integral expression at $x=0$, which can be done using Lemma \ref{lemma:balljacobix0}. Using the evaluation operations in (\ref{eq:jacobieval1}--\ref{eq:jacobieval2}) then leaves us with the following equation:
\begin{align*}
(-1)^n &\left(-\frac{(\frac{d}{2})_{n-1}}{(n-1)!}c_a+\frac{(\frac{d}{2})_{n}}{n!}c_b-\frac{(\frac{d}{2})_{n+1}}{(n+1)!}c_c\right)+c_d =\\& \sum_{k=0}^{n}(-1)^{n+k}\tfrac{{\left(n-\frac{\alpha+d}{2}+%
\frac{d-2}{2}+1\right)_{k}}{\left(\frac{d-2}{2}+k+1\right)_{n-k}}}{k!\;(n-k)!} \frac{\pi^{\frac{d}{2}}}{\Gamma(\frac{d}{2})}B\left(\tfrac{2k+\alpha+2+d}{2},1-\tfrac{\alpha+d}{2}\right),
\end{align*}
which can be solved for $c_d$ for all $n$. The result is as nice as one could hope for:
\begin{align*}
c_d = 0.
\end{align*}
Now replacing $\alpha+2$ with $\alpha$ to retain the same form for the kernel as before yields the desired result, with inherited range of validity $\alpha \in (2-d,4-d)$.
\end{proof}
\begin{remark}
The method used to prove Theorem \ref{theorem:recurrencen1case} can be used in a straightforward way to find banded operators for all $\ell \in \mathbb{N}_0$ for the power law potential
\begin{align*}
\int_{B_1}|x-y|^\alpha (1-|y|^2)^{\ell-\frac{\alpha+d}{2}} P_n^{(\ell-\frac{\alpha+d}{2},\frac{d-2}{2})}(2|y|^2-1) \mathrm{d}y,
\end{align*}
where the bandwidth increases with $\ell$, starting from the diagonal case in Theorem \ref{theorem:diagonalformula} with $\ell=0$. The inherited range of validity for the successively obtained expressions is $(2\ell-d,2+2\ell-d)$ and the respective operators each have exactly $2\ell+1$ bands in our basis. We omit these further explicit derivations, as the involved expressions become lengthy and do not require any new ideas.
\end{remark}

\subsection{Decoupling the weight and kernel power}
When dealing with both an attractive and a repulsive power law integral at the same time, we need additional tools to treat both simultaneously in one consistent basis. We begin by proving a result for power law potentials of weighted monomials on balls:
\begin{lemma}\label{lemma3}
On $B_R$ the power law potential of $|y|^{2k} (1-|y|^2)^{\ell-\frac{\alpha+d}{2}}$ with $k>0$ and $l\in\mathbb{N}_0$ satisfies the following recurrence relation:
\begin{align*}
\int_{B_R} \left\lvert x-y  \right\rvert^\alpha |y|^{2k} &(R^2-|y|^2)^{\ell-\frac{\alpha+d}{2}} \mathrm{d}y =\\ R^2 &\int_{B_R} \left\lvert x-y \right\rvert^\alpha |y|^{2(k-1)} (R^2-|y|^2)^{\ell-\frac{\alpha+d}{2}} \mathrm{d}y \\ - &\int_{B_R} \left\lvert x-y \right\rvert^\alpha |y|^{2(k-1)} (R^2-|y|^2)^{\ell+1-\frac{\alpha+d}{2}} \mathrm{d}y.
\end{align*}
\begin{proof}
Expanding one of the powers of the weight in the right-most term with the highest weight leads directly to
\begin{align*}
\int_{B_R} \left\lvert x-y \right\rvert^\alpha |y|^{2(k-1)} &(R^2-|y|^2)^{\ell+1-\frac{\alpha+d}{2}} \mathrm{d}y =\\ R^2 &\int_{B_R} \left\lvert x-y \right\rvert^\alpha |y|^{2(k-1)} (R^2-|y|^2)^{\ell-\frac{\alpha+d}{2}} \mathrm{d}y \\ - &\int_{B_R} \left\lvert x-y  \right\rvert^\alpha |y|^{2k} (R^2-|y|^2)^{\ell-\frac{\alpha+d}{2}} \mathrm{d}y.
\end{align*}
\end{proof}
\end{lemma}
\begin{remark}
Lemma \ref{lemma3} tells us that on $B_R$ the Riesz potential of finite sums of terms of the form $|y|^{2k} (R^2-|y|^2)^{\ell-\frac{\alpha+d}{2}}$ for varying $k$ can be reduced to a sum depending exclusively on solutions of the $k=0$ case in Lemma \ref{lemma1} with increasing weight parameters. For example, the case $k=1$ may be evaluated as follows:
\begin{align*}
\int_{B_R} \left\lvert x-y \right\rvert^\alpha |y|^2 &(R^2-|y|^2)^{\ell-\frac{\alpha+d}{2}} \mathrm{d}y =\\ &\frac{\pi^{\frac{d}{2}}R^{2l+2}}{\Gamma\left(\tfrac{d}{2}\right)}B\left(\tfrac{\alpha+d}{2},\tfrac{2\ell+2-\alpha-d}{2}\right){}_2F_1\left(\begin{matrix}-\tfrac{\alpha}{2}, \quad -\ell, \\ \tfrac{d}{2}\end{matrix};\tfrac{|x|^2}{R^2}\right).\\
- &\frac{\pi^{\frac{d}{2}}R^{2l+2}}{\Gamma\left(\tfrac{d}{2}\right)}B\left(\tfrac{\alpha+d}{2},\tfrac{2\ell+4-\alpha-d}{2}\right){}_2F_1\left(\begin{matrix}-\tfrac{\alpha}{2}, \quad -\ell-1, \\ \tfrac{d}{2}\end{matrix};\tfrac{|x|^2}{R^2}\right).
\end{align*}
Note that even in the general case, we can evaluate all the terms on the right-hand side given $\alpha \in (-d,2+2\ell-d)$ and that if $\ell$ is chosen such that the $k=0$ term is a polynomial, i.e. $\ell\in\mathbb{N}$, then all higher order terms are also polynomials.
\end{remark}
The above results could in principle be used to design a dense spectral method for equilibrium measures in a basis of weighted monomials but such methods would be computationally expensive as well as significantly less robust. The above is easily modified to instead give a generic Jacobi polynomial recurrence:
\begin{lemma}\label{lemma:genericcasejacobi}
On the unit ball $B_1$, the Jacobi polynomials $P_n^{\left(\ell-\frac{\alpha+d}{2},\frac{d-2}{2}\right)}(2|y|^2-1)$ with weight $(1-|y|^2)^{\ell-\frac{\alpha+d}{2}}$ satisfy the following two-term recurrence relationship:
\begin{align*}
\int_{B_1}|x&-y|^\beta (1-|y|^2)^{\ell-\frac{\alpha+d}{2}} P_{n+1}^{(\ell-\frac{\alpha+d}{2},\frac{d-2}{2})}(2|y|^2-1) \mathrm{d}y =\\ & \tfrac{2n+2\ell-\alpha-d+2}{2n+2} \int_{B_1}|x-y|^\beta (1-|y|^2)^{\ell-\frac{\alpha+d}{2}} P_{n}^{(\ell-\frac{\alpha+d}{2},\frac{d-2}{2})}(2|y|^2-1) \mathrm{d}y\\
&-\tfrac{4 n +2 \ell+2-\alpha}{2 n+2} \int_{B_1}|x-y|^\beta (1-|y|^2)^{\ell+1-\frac{\alpha+d}{2}} P_{n}^{(\ell+1-\frac{\alpha+d}{2},\frac{d-2}{2})}(2|y|^2-1) \mathrm{d}y.
\end{align*}
\end{lemma}
\begin{proof}
This is a consequence of the well known weight reduction recurrence relationship of the Jacobi polynomials in (\ref{eq:jacobiraisingB}--\ref{eq:jacobiraisingA}). In particular, with $x\rightarrow \left(2|y|^2-1\right)$ the recurrence in \eqref{eq:jacobiraisingA} takes the following form:
\begin{align*}
(1-|y|^2)P_n^{(a+1,b)}(2|y|^2-1)&=\tfrac{-(n+1)}{2(n+\frac{a}{2}+\frac{b}{2}+1)}P_{n+1}^{(a,b)}(2|y|^2-1)\\&+\tfrac{(n+a+1)}{2(n+\frac{a}{2}+\frac{b}{2}+1)}P_{n}^{(a,b)}(2|y|^2-1),
\end{align*}
which with some straightforward rearranging turns into
\begin{align*}
P_{n+1}^{(a,b)}(2|y|^2-1)&=\tfrac{(n+a+1)}{(n+1)}P_{n}^{(a,b)}(2|y|^2-1)\\&-\tfrac{2(n+\frac{a}{2}+\frac{b}{2}+1)}{(n+1)}(1-|y|^2)P_n^{(a+1,b)}(2|y|^2-1).
\end{align*}
The result then follows from linearity after plugging this into the weighted power law integral with appropriate parameters $(a,b)$.
\end{proof}
As before, advancing to step $n+1$ using this recurrence uses knowledge of the previous step $n$ in the same basis as well as one with higher weight parameter, meaning that the computational cost of this recurrence does not scale linearly with $n$. Normally such a general recurrence relationship which holds true irrespective of the structure of the kernel would be of little use, as the solution for high orders requires the luxury of extensive knowledge about initial results with higher weight terms, in particular for the $0$-th order terms. In the present case of the power law potential, however, exactly this knowledge is given by Lemma \ref{lemma2}. In conjunction with Theorem \ref{theorem:recurrencen1case} we could thus in principle use this to compute equilibrium measures even in the attractive-repulsive case where two potentials with different powers are present, while remaining in a consistent polynomial basis. In practice however this recurrence would also be a bad choice as it scales poorly with $n$ and suffers from numerical instability as the polynomial orders increase. To circumvent this problem, we first take a detour and look for a direct generic solution to the power law integral of Jacobi polynomials:
\begin{theorem}\label{theorem:hypergeomgeneral}
On the $d$-dimensional unit ball $B_1$ the power law potential, with power $\alpha \in(-d,2+2m-d)$, $m\in\mathbb{N}_0$ and $\beta>-d$, of the $n$-th weighted radial Jacobi polynomial $$(1-|y|^2)^{m-\frac{\alpha+d}{2}}P_n^{(m-\frac{\alpha+d}{2},\frac{d-2}{2})}(2|y|^2-1)$$ reduces to a Gaussian hypergeometric function as follows:
\begin{align*}
\int_{B_1}&|x-y|^\beta (1-|y|^2)^{m-\frac{\alpha+d}{2}} P_n^{(m-\frac{\alpha+d}{2},\frac{d-2}{2})}(2|y|^2-1) \mathrm{d}y\\
&= \tfrac{\pi ^{d/2} \Gamma \left(1+\frac{\beta}{2}\right) \Gamma \left(\frac{\beta+d}{2}\right) \Gamma \left(m+n-\frac{\alpha+d}{2}+1\right)}{\Gamma \left(\frac{d}{2}\right) \Gamma (n+1) \Gamma \left(\frac{\beta}{2}-n+1\right) \Gamma \left(\frac{\beta-\alpha}{2}+m+n+1\right)}{}_2F_1\left(\begin{matrix}n-\frac{\beta}{2}, \quad -m-n+\frac{\alpha-\beta}{2} \\\frac{d}{2}\end{matrix};|x|^2\right).
\end{align*}
\begin{proof}
Using the explicit representation found on the second line of \eqref{eq:radialseriesrep}, the left hand side power law integral can be rewritten
\begin{align*}
&\int_{B_1}|x-y|^\beta (1-|y|^2)^{m-\frac{\alpha+d}{2}} P_n^{(m-\frac{\alpha+d}{2},\frac{d-2}{2})}(2|y|^2-1) \mathrm{d}y \\
&= \tfrac{\Gamma \left(m+n-\frac{\alpha+d}{2}+1\right)}{n! \Gamma \left(m+n-\frac{\alpha}{2}\right)} \sum_{k=0}^n (-1)^k \binom{n}{k}\tfrac{\Gamma \left(m+n+k-\frac{\alpha}{2}\right)}{\Gamma \left(m+k-\frac{\alpha+d}{2}+1\right)} \int_{B_1}|x-y|^\beta (1-|y|^2)^{k+m-\frac{\alpha+d}{2}} \mathrm{d}y.
\end{align*}
Applying Corollary \ref{corr:colemma2}, we can resolve the integral despite the power of the weight and kernel being independent to obtain:
\begin{align*}
\tfrac{\pi ^{\frac{d}{2}}\Gamma \left(m+n-\frac{\alpha+d}{2}+1\right)}{n! \Gamma(\frac{d}{2}) \Gamma \left(m+n-\frac{\alpha}{2}\right)} \sum_{k=0}^n (-1)^k \binom{n}{k}\tfrac{\Gamma \left(m+n+k-\frac{\alpha}{2}\right)\Gamma \left(\frac{d+\beta}{2}\right)}{\Gamma \left(k+m+\frac{\beta-\alpha}{2}+1\right)} {}_2F_1\left(\begin{matrix}-\tfrac{\beta}{2}, -m-k-\tfrac{\beta-\alpha}{2} \\ \tfrac{d}{2}\end{matrix} ;|x|^2\right)
\end{align*}
In the above we expanded the Beta function into its Gamma function representation and performed the then obvious cancellation of terms. With some careful algebraic manipulations this can then be cast into the following form
\begin{align*}
\kappa_{m,n}^{\alpha,\beta,d} \sum_{k=0}^n (-1)^k \binom{n}{k}\tfrac{\left(1-\left(-n-m+\frac{\alpha}{2}+1\right)\right)_k}{\left(1-\left(-m-\frac{\beta-\alpha}{2}\right)\right)_k} {}_2F_1\left(\begin{matrix}-\tfrac{\beta}{2},\quad -m-k-\tfrac{\beta-\alpha}{2}\\\tfrac{d}{2}\end{matrix};|x|^2\right),
\end{align*}
with constant
\begin{align*}
\kappa_{m,n}^{\alpha,\beta,d} = \tfrac{\pi ^{\frac{d}{2}}\Gamma \left(m+n-\frac{\alpha+d}{2}+1\right)\Gamma \left(\frac{d+\beta}{2}\right)\Gamma \left(-\frac{\alpha}{2}+m+n\right)}{n! \Gamma(\frac{d}{2}) \Gamma \left(m+n-\frac{a}{2}\right)\Gamma \left(\frac{\beta-\alpha}{2}+m+1\right)}.
\end{align*}
In this form it is clear that the appearing sum is of the form of Equation \eqref{eq:prudnikovsum} with $\mathfrak{a}=-m-\tfrac{\beta-\alpha}{2}$ and $\mathfrak{b}=-n-m+\frac{\alpha}{2}+1$. Evaluation via Equation \eqref{eq:prudnikovsum} then yields
\begin{align}\label{eq:prudnikovfinalstep}
\nonumber \kappa_{m,n}^{\alpha,\beta,d} &\frac{(\mathfrak{b}-\mathfrak{a})_n}{(1-\mathfrak{a})_n} {}_3F_2\left(\begin{matrix}-\frac{\beta}{2},\quad \mathfrak{a}-\mathfrak{b}+1,\quad \mathfrak{a}-n \\ \frac{d}{2},\quad \mathfrak{a}-\mathfrak{b}-n+1\end{matrix};|x|^2\right)\\
&= \kappa_{m,n}^{\alpha,\beta,d} \frac{(1 + \frac{\beta}{2} - n)_n}{(\frac{\beta-\alpha}{2}+m+1)_n}{}_3F_2\left(\begin{matrix}-\frac{\beta}{2},\quad n-\frac{\beta}{2},\quad -m-n+\frac{\alpha-\beta}{2}\\ \frac{d}{2},\quad -\frac{\beta}{2}\end{matrix};|x|^2\right).
\end{align}
In general, any hypergeometric $_p F_q$ function with an equal upper and lower parameter reduces to a $_{p-1} F_{q-1}$ function, a fact that is straightforwardly seen from their series representation. Explicitly for the above appearing $_3F_2$ we have:
\begin{align*}
{}_3F_2\left(\begin{matrix}a_1, \quad a_2, \quad a_3\\c_1,\quad a_3\end{matrix};z\right) = {}_2F_1\left(\begin{matrix}a_1,\quad a_2 \\c_1\end{matrix};z\right),
\end{align*}
where we remind ourselves that $_p F_q$ functions are symmetric when exchanging upper with other upper parameters and likewise for the lower parameters.
Applying this to \eqref{eq:prudnikovfinalstep} and simplifying the constants concludes the proof of the theorem.
\end{proof}
\end{theorem}
\begin{remark}\label{rem:genericweightparam}
The specific form of Theorem \ref{theorem:hypergeomgeneral} was chosen to suit our needs but it is easy to see that a simple substitution yields the generic weight parameter form
\begin{align*}
\int_{B_1}&|x-y|^\beta (1-|y|^2)^{\lambda} P_n^{(\lambda,\frac{d-2}{2})}(2|y|^2-1) \mathrm{d}y\\
&= \tfrac{\pi ^{d/2} \Gamma \left(1+\frac{\beta}{2}\right) \Gamma \left(\frac{\beta+d}{2}\right) \Gamma \left(\lambda+n+1\right)}{\Gamma \left(\frac{d}{2}\right) \Gamma (n+1) \Gamma \left(\frac{\beta}{2}-n+1\right) \Gamma \left(\frac{\beta+d}{2}+\lambda+n+1\right)}{}_2F_1\left(\begin{matrix} n-\frac{\beta}{2},\quad -\lambda-n-\frac{\beta+d}{2} \\ \frac{d}{2}\end{matrix};|x|^2\right).
\end{align*}
without any change to the proof.
\end{remark}
The importance of Theorem \ref{theorem:hypergeomgeneral} for our method is that it can be used in conjunction with highly efficient and accurate implementations for the computation of the Gaussian hypergeometric functions, as implemented in HypergeometricFunctions.jl \cite{juliamathhypergeometricfunctions_2020}, to compute the operator entries even when $\alpha \neq \beta$. The specific hypergeometric function algorithm we relied on for the numerical experiments in this paper is the one discussed in \cite{pearson_numerical_2017}. This can be done either by directly expanding the $n$-th variant of the formula or preferably by using the following recurrence relationship where the basis and kernel powers are now decoupled:
\begin{corollary}\label{cor:generalrecurrence}
On the unit ball $B_1$, the power law integral of the Jacobi polynomials $P_n^{\left(m-\frac{\alpha+d}{2},\frac{d-2}{2}\right)}(2|y|^2-1)$ with weight $(1-|y|^2)^{m-\frac{\alpha+d}{2}}$, $\alpha \in(-d,2+2m-d)$ and $\beta >-d$ satisfies the following three term recurrence relationship:
\begin{align*}
\int_{B_1} |x-y|^\beta &(1-|y|^2)^{m-\frac{\alpha+d}{2}} P_{n+1}^{(m-\frac{\alpha+d}{2},\frac{d-2}{2})}(2|y|^2-1) \mathrm{d}y\\ &= (\mathfrak{c}_a |x|^2 + \mathfrak{c}_b) \int_{B_1}|x-y|^\beta (1-|y|^2)^{m-\frac{\alpha+d}{2}} P_n^{(m-\frac{\alpha+d}{2},\frac{d-2}{2})}(2|y|^2-1) \mathrm{d}y \\ &+ \mathfrak{c}_c \int_{B_1}|x-y|^\beta (1-|y|^2)^{m-\frac{\alpha+d}{2}} P_{n-1}^{(m-\frac{\alpha+d}{2},\frac{d-2}{2})}(2|y|^2-1) \mathrm{d}y,
\end{align*}
where
\begin{align*}
\mathfrak{c}_a &= -\tfrac{(-\alpha +2 m+4 n) (-\alpha +2 m+4 n+2) (\alpha +d-2 (m+n+1))}{2 (n+1) (-\alpha +\beta +2 m+2 n+2) (-\alpha +\beta +d+2 m+2 n)},\\
\mathfrak{c}_b &= -\tfrac{(-\alpha +2 m+4 n) (\alpha +d-2 (m+n+1)) (d (-\alpha +2 \beta +2 m+2)-2 (2 n-\beta ) (-\alpha +\beta +2 m+2 n))}{2 (n+1) (-\alpha +2 m+4 n-2) (-\alpha +\beta +2 m+2 n+2) (-\alpha +\beta +d+2 m+2 n)}, \\
\mathfrak{c}_c &= \tfrac{(-\beta +2 n-2) (\beta +d-2 n) (-\alpha +2 m+4 n+2) (\alpha +d-2 (m+n)) (\alpha +d-2 (m+n+1))}{4 n (n+1) (-\alpha +2 m+4 n-2) (-\alpha +\beta +2 m+2 n+2) (-\alpha +\beta +d+2 m+2 n)}.
\end{align*}
\end{corollary}
\begin{proof}
This is obtained by plugging the appropriate parameters obtained in Theorem \ref{theorem:hypergeomgeneral} into the following recurrence relationship for Gaussian hypergeometric functions:
\begin{align*}
\, _2F_1\left(\begin{matrix}a+1,\quad b-1 \\ c\end{matrix};z\right) &= \tfrac{(a-b) \left(z \left((a-b)^2-1\right)+2 a b-a c-b c+c\right)}{a (a-b-1) (b-c)}\, _2F_1\left(\begin{matrix}a,\quad b\\c\end{matrix};z\right)\\
 &-\tfrac{b (a-b+1) (a-c)}{a (a-b-1) (b-c)}\, _2F_1\left(\begin{matrix}a-1,\quad b+1 \\ c\end{matrix};z\right).
\end{align*}
The fact that there exist recurrence relationships for Gaussian hypergeometric functions of form $_2F_1\left( \begin{matrix}a+\epsilon_1 n, \quad b+\epsilon_2 n \\ c+\epsilon_3 n\end{matrix};z\right)$ with $\epsilon_i \in \{-1,0,1\}$ which can be derived from contiguous relations \cite[15.5.11--15.5.18]{nist_2018} is widely known, cf. \cite{gil2006abc,gil2007numerically,pearson2017numerical} and the references therein, and Mathematica code for the symbolic generation of such recurrences is available \cite{ibrahim2008contiguous}.
\end{proof}
\begin{remark}\label{rem:remarkbanded}
The special case where the kernel and weight terms are related via $\alpha = \beta$ results in a polynomial right hand side, meaning by Theorem \ref{theorem:hypergeomgeneral} the operator must have finitely many non-zero subdiagonal bands and by the three term recurrence in Corollary \ref{cor:generalrecurrence} it then also has finitely many superdiagonal non-zero bands. The bandwidth of the operator in the polynomial case for any specific parameter $\alpha=\beta$ can be determined by the order of the right hand side polynomial when $n=0$. The minimal bandwidth that is achievable using Jacobi polynomials is determined by the range of validity $\beta \in(-d,2+2m+\beta-\alpha-d)$, i.e. for $\alpha = \beta$ and $\alpha \in (-d,2-d)$ we can get diagonal operators, for $\alpha \in (2-d,4-d)$ we can at best get tridiagonal operators and so on, consistent with the results of the previous sections.
\end{remark}
\subsection{The mass condition in arbitrary dimension}\label{sec:themasscondition}
As mentioned in section \ref{sec:introequilibrium}, equilibrium measure problems of the form discussed in this paper only become well-posed with the addition of a mass condition, i.e. by demanding that the equilibrium measure must satisfy
\begin{align*}
\int_{B_R} \rho(y) \D y = M
\end{align*}
on its domain of support $B_R$, the $d$-dimensional ball of radius $R$. In the computational setting this will require us to be able to evaluate the unit ball integral of sufficiently well-behaved functions expanded in a weighted basis of radial Jacobi polynomials. 
\begin{lemma}
Let $\rho(y)=\rho(|y|^2)$, $y\in B_1$ be a function such that $\int_{B_1} \rho(y) \mathrm{d}y = M_1$, with $M_1$ constant. Furthermore, we assume that its expansion
$$\rho(y) = \sum_{n=0}^\infty \rho_n (1-|y|^2)^a P_n^{(a,\frac{d-2}{2})}(2|y|^2-1)$$
in weighted radial Jacobi polynomials on the $d$-dimensional unit ball $B_1$ exists and satisfies $\int_{B_1} \sum_{n=0}^{\infty} |\rho_n (1-|y|^2)^a P_n^{(a,\frac{d-2}{2})}(2|y|^2-1)| \mathrm{d}y < \infty$. Then the value of $M_1$ is determined entirely by the $0$-th coefficient, that is:
\begin{align*}
M_1 = \int_{B_1} \rho(y) \D y = \frac{\pi^{\frac{d}{2}}\Gamma (a+1)}{\Gamma \left(a+\frac{d}{2}+1\right)} \rho_0.
\end{align*}
\end{lemma}
\begin{proof}
The domain and radial symmetry of this problem suggests the use of hyperspherical coordinates:
\begin{align*}
M_1 = \int_{B_1} \rho(y) dy &= \sum_{n=0}^\infty \rho_{n} \int_{B_1} (1-|y|^2)^a P_n^{(a,\frac{d-2}{2})}(2|y|^2-1) dy\\
&=\sum_{n=0}^\infty \rho_{n} \int_{S^{d-1}} \int_{r=0}^1 (1-r^2)^a P_n^{(a,\frac{d-2}{2})}(2r^2-1) r^{d-1} dr d\sigma(\omega)\\
&= \sigma(S^{d-1}) \sum_{n=0}^\infty \rho_{n} \int_{r=0}^1 (1-r^2)^a P_n^{(a,\frac{d-2}{2})}(2r^2-1) r^{d-1} dr,
\end{align*}
where $\sigma(S^{d-1}) = \frac{2\pi^\frac{d}{2}}{\Gamma(\frac{d}{2})}$ is the surface area of the $d-1$ dimensional sphere. The exchange of integration and infinite sum in the first line is justified by the Fubini-Tonelli theorem. The integral in the resulting expression can be evaluated easily by reversing the quadratic shift of the Jacobi polynomials via transformation $2r^2-1 = t$:
\begin{align*}
\sum_{n=0}^\infty \rho_{n} \int_0^1 (1-r^2)^a &P_n^{(a,\frac{d-2}{2})}(2r^2-1) r^{d-1} dr \\&= \frac{1}{4} \sum_{n=0}^\infty \rho_{n} \int_{-1}^1 \frac{(1-t)^a}{2^a}\frac{(1+t)^{\frac{d-2}{2}}}{2^{\frac{d-2}{2}}} P_n^{(a,\frac{d-2}{2})}(t) dt\\
 &= 2^{\frac{-2-2a-d}{2}} \sum_{n=0}^\infty \rho_{n} \int_{-1}^1 (1-t)^a (1+t)^{\frac{d-2}{2}} P_n^{(a,\frac{d-2}{2})}(t) dt\\
 &= \frac{\Gamma (a+1) \Gamma \left(\frac{d}{2}\right)}{2 \Gamma \left(a+\frac{d}{2}+1\right)} \rho_0,
\end{align*}
where the last equality relies on the classical orthogonality condition of the Jacobi polynomials in Equation \eqref{eq:orthogonalityconditionJacobi} where we remind ourselves that $P_0(t) = 1$. Combining this with the prior expression yields the stated result.
\end{proof}
\section{Description of the numerical method}\label{sec:methoddescription}
In the previous sections we derived recurrence relationships and explicit representations of the power law potential of Jacobi polynomials on $d$ dimensional balls. We now detail how these results can be used to produce banded and approximately banded operators to efficiently compute the power law potential of a given rotationally symmetric function $f(r)$ and then discuss how these tools can be used to solve equilibrium measure problems.
\subsection{Operator sparsity structure for classical Jacobi polynomials}\label{sec:sparsitystructure}
We have seen in section \ref{sec:subsecpowerbanded} that we can achieve banded power law integral operators acting on coefficient vectors of weighted radial Jacobi polynomials by making an appropriate basis choice. In particular we have seen that for a given power $\alpha \in (2\ell-d,2+2\ell-d)$ with $\ell \in \mathbb{N}_0$ in $d$ dimensions, the choice of basis $$(1-|y|^2)^{\ell-\frac{\alpha+d}{2}}P_n^{(\ell-\frac{\alpha+d}{2},\frac{d-2}{2})}(2|y|^2-1)$$ gives banded operators with exactly $2\ell+1$ bands. We present some illustrative examples of this in terms of matrix spy plots in Figure \ref{fig:spy_banded}. Using Theorem \ref{theorem:diagonalformula}, Theorem \ref{theorem:recurrencen1case} and the straightforward generalizations one can generate these operators with very high efficiency.\\
\begin{figure}
     \subfloat[$\alpha = -\pi, d = 5$]
    {{ \centering \includegraphics[width=4cm]{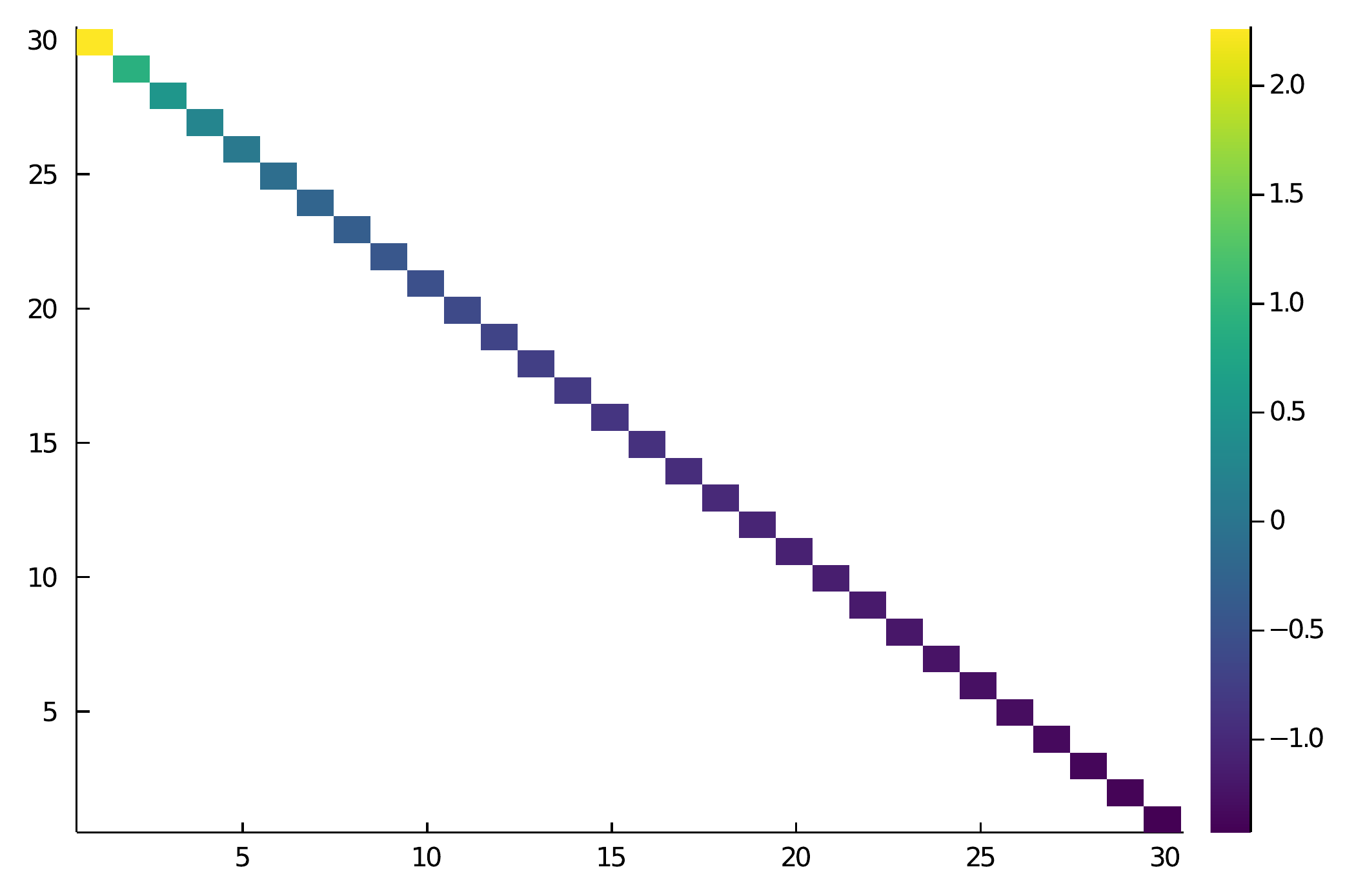} }}
         \subfloat[$\alpha = 0.5, d = 2$]
    {{ \centering \includegraphics[width=4cm]{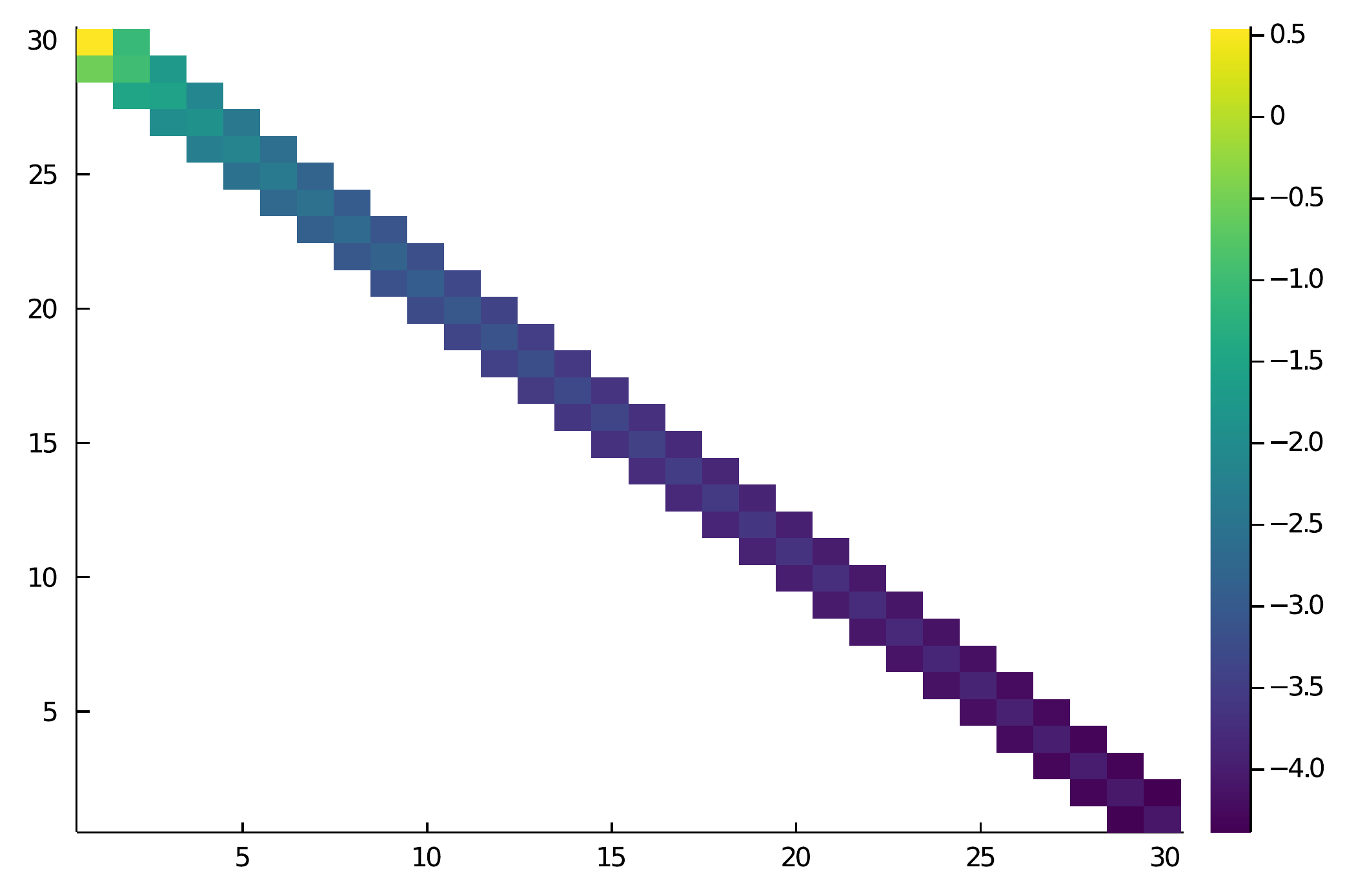} }}
             \subfloat[$\alpha = 3.6, d = 3$]
    {{ \centering \includegraphics[width=4cm]{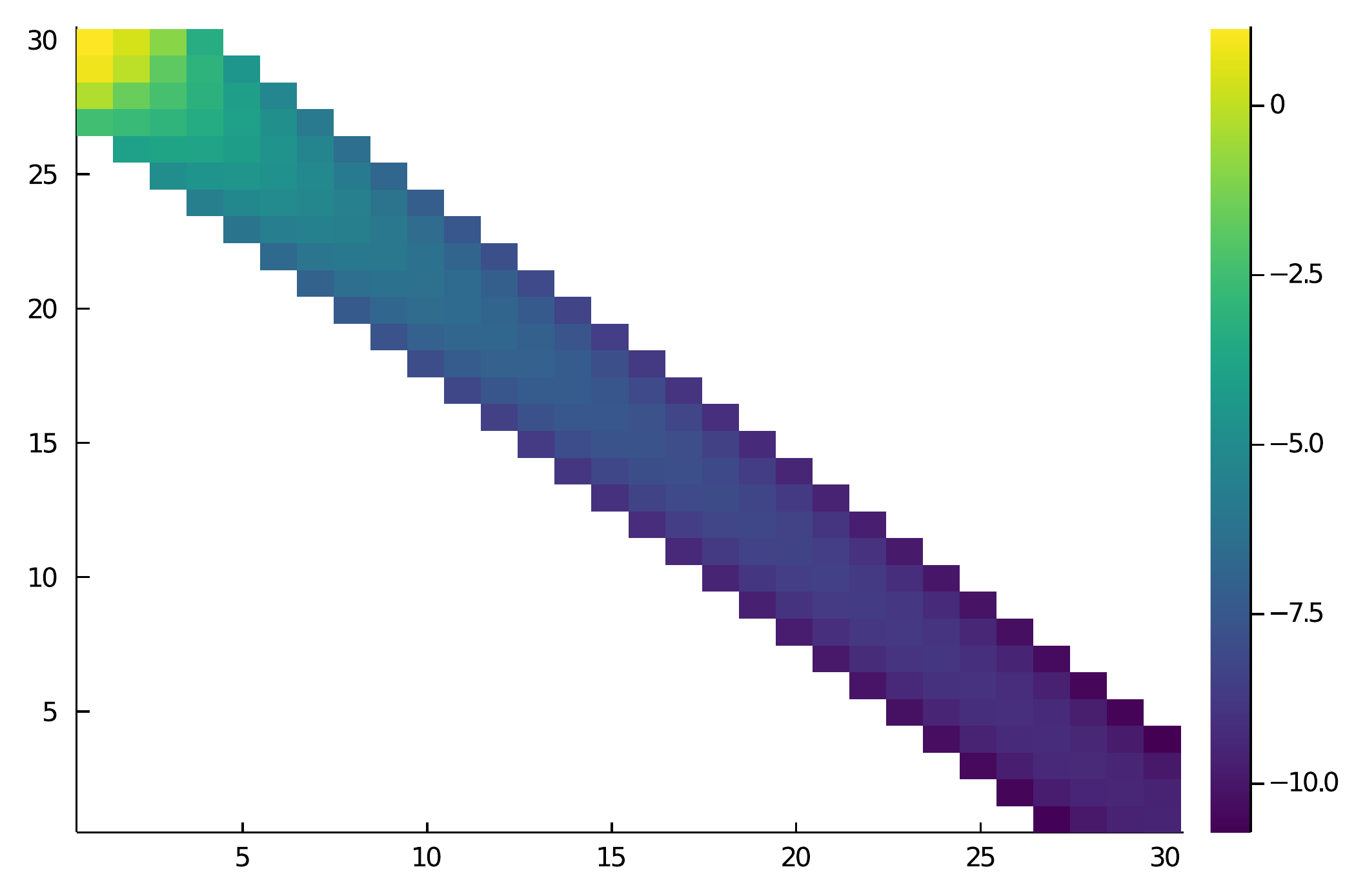} }}
    \caption{Banded operator spy plots for various values of $\alpha$ in different dimensions. The legend is logarithmic and indicates the order of magnitude of the entries.}%
    \label{fig:spy_banded}
    \end{figure}
An important and interesting special case occurs when $\alpha$ is an even integer, as can e.g. be seen from the results in Theorem \ref{theorem:hypergeomgeneral} and Corollary \ref{cor:generalrecurrence}, namely the power law operator only has finitely many non-zero values, the exact number depending only on the value of $\alpha$, see the example matrix spy plots in Figure \ref{fig:spy_evenint}. This is consistent with what was shown in one-dimension for ultraspherical polynomials in \cite{gutleb_computing_2020}. In the attractive-repulsive case this special case result means that we can obtain a simultaneously banded operator for any $\beta$ for which solutions exist.\\
\begin{figure}
     \subfloat[$\alpha = 4, d = 2$]
    {{ \centering \includegraphics[width=4cm]{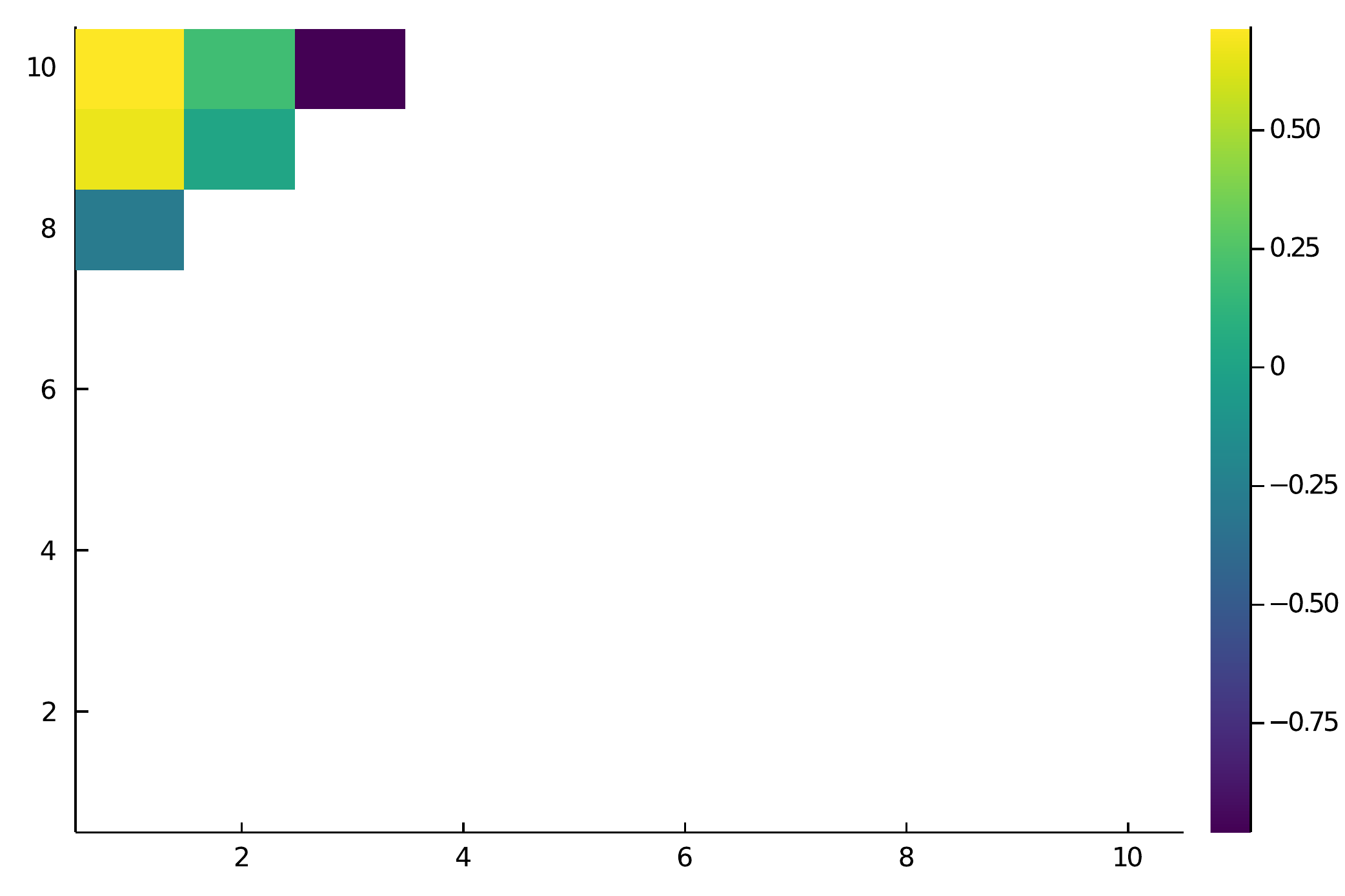} }}
         \subfloat[$\alpha = 6, d = 3$]
    {{ \centering \includegraphics[width=4cm]{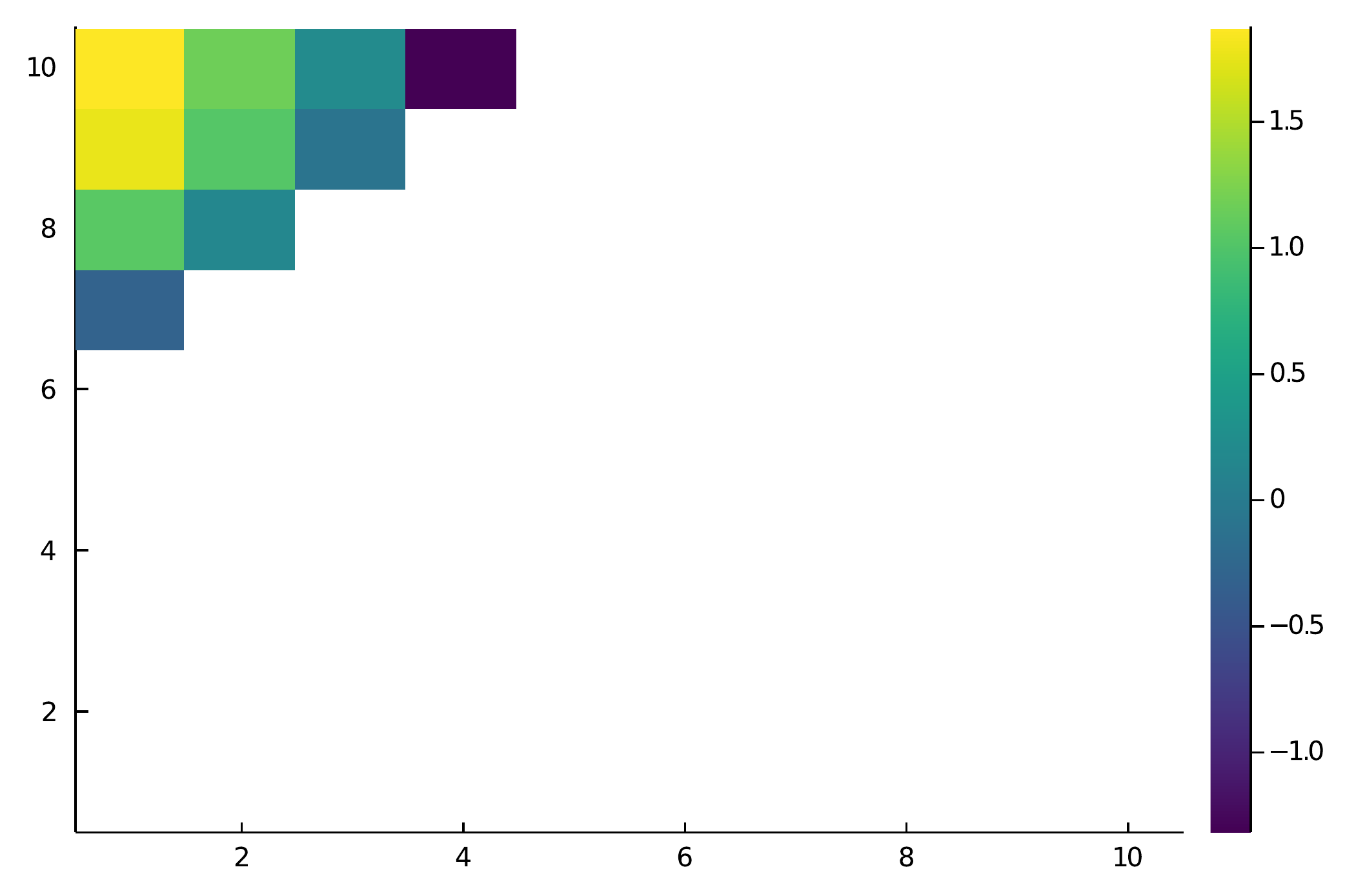} }}
             \subfloat[$\alpha = 4, d = 4$]
    {{ \centering \includegraphics[width=4cm]{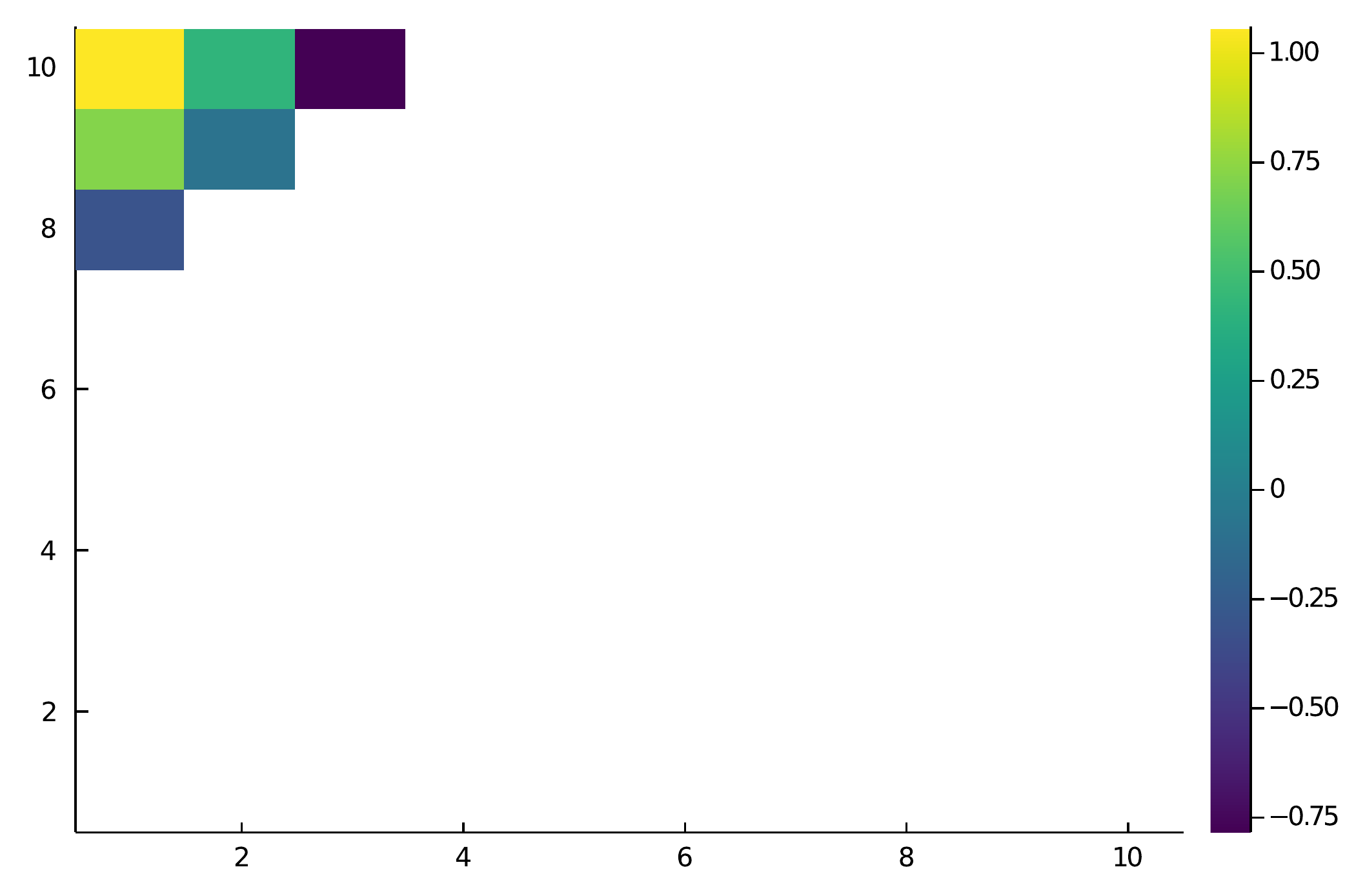} }}
    \caption{Operator spy plots for even integer $\alpha$ in different dimensions. The legend is logarithmic and indicates the order of magnitude of the entries.}%
    \label{fig:spy_evenint}%
\end{figure}
When considering the general attractive-repulsive equilibrium measure problem, one has to consider both power law integrals in a consistent basis in order to apply spectral methods. To do this, we choose the basis such that the attractive $\alpha$-operator is banded as above, then generate the attractive $\beta$-operator in the same basis by using Theorem \ref{theorem:hypergeomgeneral} and Corollary \ref{cor:generalrecurrence}. Again consistent with what was previously observed in one-dimension for ultraspherical polynomials \cite{gutleb_computing_2020}, these operators are approximately banded and decay off the main band, see the example matrix spy plots in Figure \ref{fig:spy_approxbanded}. Improving on the understanding of the one-dimensional case, the general expression obtained in Theorem \ref{theorem:hypergeomgeneral} explains the approximate bandedness as each column of the operator represents the coefficients of an expansion of the polynomial order dependent Gaussian hypergeometric function in a Jacobi polynomial basis. As we have obtained a recurrence relationship for this general case it is possible in practice to compute a banded approximation for a chosen bandwidth without the wasteful step of having to compute the dense operator first. That being said, suitable convergence is generally obtained already for sizes were dense matrix computations are still feasible.
\begin{figure}
     \subfloat[$\alpha = -\pi, \beta =-2.4 , d = 5$]
    {{ \centering \includegraphics[width=4cm]{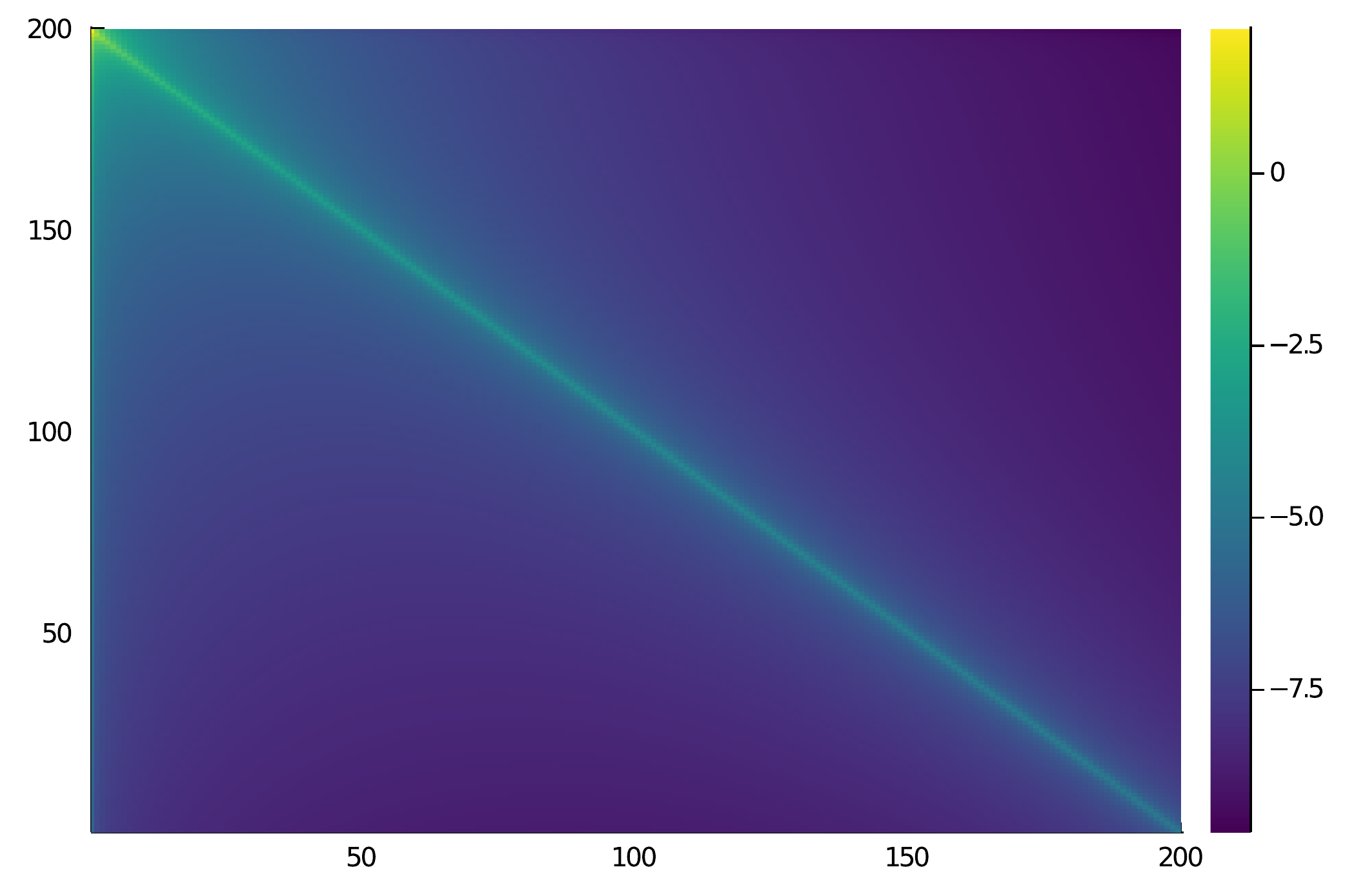} }}
         \subfloat[$\alpha = 1.9, \beta = 0.5, d = 2$]
    {{ \centering \includegraphics[width=4cm]{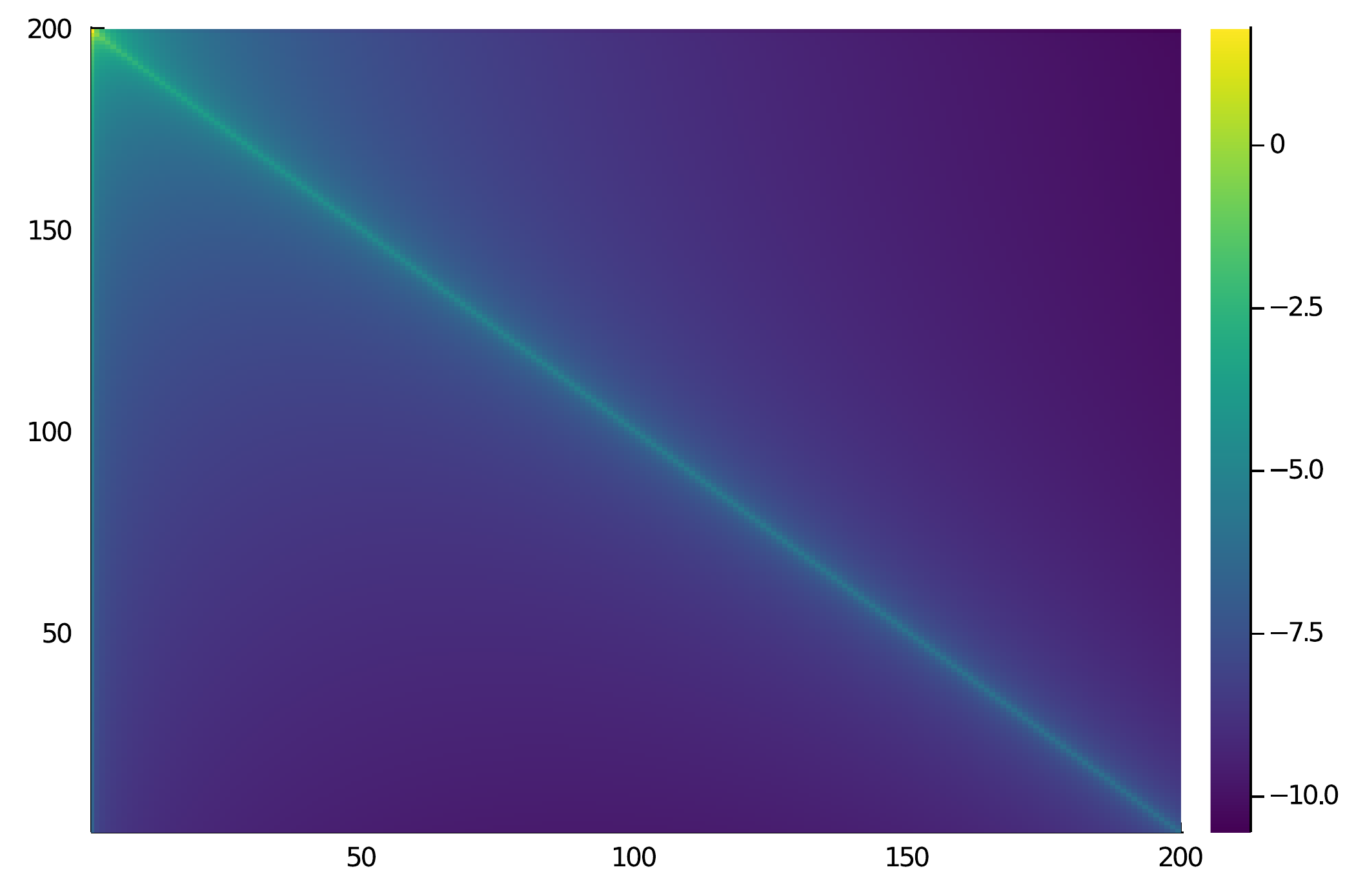} }}
             \subfloat[$\alpha = 3.76, \beta = -\frac{2}{3}, d = 3$]
    {{ \centering \includegraphics[width=4cm]{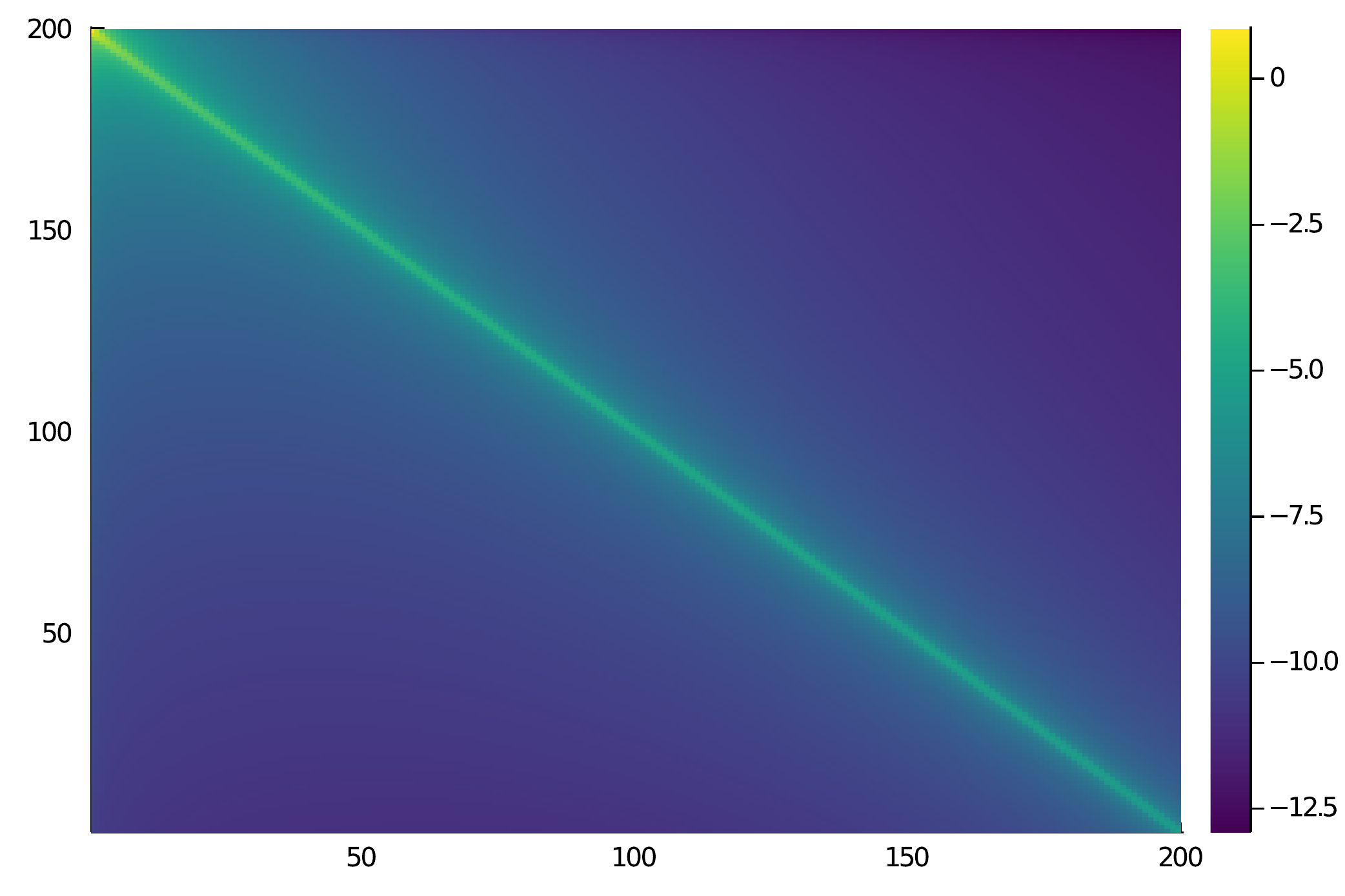} }}
    \caption{Spy plots for $\beta$-operators for various values of $\beta$ in different dimensions with the basis chosen such the $\alpha$-operator is banded. The legend is logarithmic and indicates the order of magnitude of the entries.}%
    \label{fig:spy_approxbanded}
    \end{figure}

\subsection{Application to equilibrium measure problems}
With the established operator structure and recurrences to efficiently generate them, the application of these results to higher dimensional equilibrium measure problems is mostly analogous to that of the one-dimensional ultraspherical method in \cite{gutleb_computing_2020}. As such, we will slightly abbreviate this discussion by omitting the question of solving equilibrium measure problems with only one power and an external potential, which works analogously to what is discussed in \cite[section 3.1]{gutleb_computing_2020}, and instead focus on the significantly more challenging problem of attractive-repulsive systems with vanishing external potential.\\

We remind ourselves of the problem we intend to solve: Find the positive density $\rho(\bar{y})$ which minimizes the scalar energy
\begin{equation*}
\frac{1}{\alpha}\int_{B_R} |\bar{x}-\bar{y}|^\alpha \rho(\bar{y}) d\bar{y} - \frac{1}{\beta}\int_{B_R} |\bar{x}-\bar{y}|^\beta \rho(\bar{y}) d\bar{y} = E,
\end{equation*}
Note both $\rho$ and the radius $R$ of its support $B_R$ are unknown. To be able to use our Jacobi polynomial results we thus first normalize this expression to the unit ball: 
\begin{equation*}
\frac{R^{\alpha+d}}{\alpha}\int_{B_1} |x-y|^\alpha \rho(Ry) dy - \frac{R^{\beta+d}}{\beta}\int_{B_1} |x-y|^\beta \rho(Ry) dy = E.
\end{equation*}
where $\bar{x} = Rx$ and $\bar{y} = Ry$ and $|y|,|x| \leq 1$. With its support normalized to the unit ball $B_1$, the density $\rho(Ry)$ is now assumed to be given in its expansion in weighted Jacobi polynomials
\begin{align*}
\rho	(Ry) = \sum_{n=0}^\infty \rho_n (1-|y|^2)^{\ell-\frac{\alpha+d}{2}} P_n^{(\ell-\frac{\alpha+d}{2},\frac{d-2}{2})}(2|y|^2-1).
\end{align*}
Plugging this into the governing equation we find
\begin{align*}
E &= \frac{R^{\alpha+d}}{\alpha}\sum_{n=0}^\infty \rho_n\int_{B_1} |x-y|^\alpha (1-|y|^2)^{\ell-\frac{\alpha+d}{2}} P_n^{(\ell-\frac{\alpha+d}{2},\frac{d-2}{2})}(2|y|^2-1) dy \\&- \frac{R^{\beta+d}}{\beta}\sum_{n=0}^\infty \rho_n\int_{B_1} |x-y|^\beta (1-|y|^2)^{\ell-\frac{\alpha+d}{2}} P_n^{(\ell-\frac{\alpha+d}{2},\frac{d-2}{2})}(2|y|^2-1) dy.
\end{align*}
Choosing instead to represent this as the action of operators on coefficient vectors as described in section \ref{sec:jacobi}, we obtain the linear system
\begin{align*}
\left(\frac{R^{\alpha+d}}{\alpha} U^{\alpha}_\alpha - \frac{R^{\beta+d}}{\beta} U^{\beta}_\alpha\right) \vec{\rho} = E \vec{e}_1,
\end{align*}
where the upper index in $U^{a}_b$ denotes the power of the kernel and the bottom index the power of the operator which is made banded through the choice of basis. We have seen in the previous section that $U^{\alpha}_\alpha$ is banded and $U^{\beta}_\alpha$ is approximately banded. This system can be uniquely solved for the coefficient vector $\vec{\rho}$ \emph{if} $R$ is known. Since $R$ is the radius of support of the equilibrium measure, it is in general not known. However, this approach nevertheless allows us to compute a unique corresponding measure for any $R$, meaning that we have reduced a complicated minimization problem over the space of positie measures to a minimization in the single scalar value $R$.\\
There is one further complication to address: obtaining the unique $\vec{\rho}$ corresponding to a given $R$ appears to require knowledge of $E(\rho)$, which is certainly not known. This hurdle is overcome by the mass condition. Rearranging the above to
\begin{align*}
\left(\frac{R^{\alpha+d}}{\alpha} U^{\alpha}_\alpha - \frac{R^{\beta+d}}{\beta} U^{\beta}_\alpha\right) \frac{\vec{\rho}}{E} = \vec{e}_1,
\end{align*}
allows us to solve for a coefficient vector which corresponds to $\frac{\rho(Ry)}{E}$. We can get rid of the constant $E$ by normalizing our measure according to the given mass constraint. This can be done with very high efficiency using the approach described in section \ref{sec:themasscondition}, ultimately allowing us to compute both a unique $\rho(y)$ and its corresponding energy $E$ for a given radius $R$.
\begin{remark}
$U^{\alpha}_\alpha$ and $U^{\beta}_\alpha$ are independent of $R$. Thus, when minimizing the energy by varying $R$, the operator does not need to be re-computed for each $R$. Instead, we store $U^{\alpha}_\alpha$ and $U^{\beta}_\alpha$ after the first generation and then simply compute $\left(\frac{R^{\alpha+d}}{\alpha} U^{\alpha}_\alpha - \frac{R^{\beta+d}}{\beta} U^{\beta}_\alpha\right)$.\end{remark}
\begin{remark} The dimension $d$ has no direct impact on the computational cost, meaning that due to the rotational symmetry this is a rare case where calculations on the $d$-hypersphere are just as efficient as those on the interval. This is in stark constrast to the widely used methods of discrete particle swarm simulations to approach these problems which scale catastrophically with $d$, making even two or three dimensional computations almost impossible in interesting parameter ranges, not to mention higher dimensional problems.
\end{remark}

\subsection{Notes on convergence and stability}
In this section we address the concern of numerical instability due to the fact that Fredholm integral equations of first kind posed on Banach spaces, such as the ones appearing in these equilibrium measure problems, are Hilbert-Schmidt and compact and thus cannot simply be inverted. In \cite{gutleb_computing_2020} we used a Tikhonov regularization \cite{tikhonov1963regularization,tikhonov1963solution} approach to overcome this problem and this also works for the $d$-dimensional generalization in this paper. We will only briefly sketch the idea behind Tikhonov regularization and leave the rest to the dedicated literature, see e.g. \cite{neggal_projected_2016,nair_linear_2009}.\\
The idea behind this approach is to instead solve a well-posed second kind Fredholm integral equation which is in a well-defined sense 'adjacent' to the actual problem we intend to find solutions for. Denoting the full operator we wish to invert simply by $\mathcal{F}$, the most straightforward Tikhonov regularization is to instead solve the problem
\begin{align*}
(s \mathcal{I} + \mathcal{F}^*\mathcal{F})\vec{\rho}_s = \mathcal{F}^*E\vec{e}_1,
\end{align*}
where $s$ is small. As mentioned in \cite{gutleb_computing_2020}, we note that the error of the $n$-th order $s$-regularized approximation $\rho_{s,n}$ incurred from such a modification can be estimated via
\begin{align*}
|\rho_{s,n} - \rho| \leq |\rho_{s,n}- \rho_{s}| + |\rho_{s} - \rho|,
\end{align*}
splitting into an error originated only from the Thikonov projection and an error caused by the spectral expansion of the regularized solution. We discuss the convergence properties of our implementation in section \ref{sec:numericalconvergence}, where we also showcase the substantially improved stability gained from this regularization.
\section{Numerical validation, convergence and experiments}\label{sec:numexp}

\subsection{Uniqueness of the obtained equilibrium measures}
Before comparing our method with analytic results and other numerical approaches, we explore whether our method produces a unique result for each set of parameters. It is currently not analytically known whether power law equilibrium measures are unique for general parameter ranges, so this investigation also provides numerical evidence that this is indeed the case.\\
The most straightforward way to explore this question is to plot the energy of the measures obtained as a function of the radius input, which we do in Figure \ref{fig:uniquesolutions} for two generic combinations of dimensions and parameters. Figure \ref{fig:uniquesolutions} also shows the corresponding computed equilibrium measures. The minimal energy positive measure is unique in each tested case, including all of the other problems considered in the later sections. The consistently observed lower energy states for higher radius values than the local minimum are not positive measures and thus not admissible for the problem. Similar structures were observed in \cite{gutleb_computing_2020} for the one-dimensional ultraspherical method. In practice the problem of finding the minimal positive measure is nice enough for a simple constrained optimization to succeed when started on a lower radius than the solution. Automatically probing the energy on a grid before setting up the constraints of the optimization problem can lead to significant performance improvements.
\begin{figure}
\centering
     \subfloat[$(\alpha, \beta, d)  = (1.2, 0.1993, 2)$]
    {{ \centering \includegraphics[width=4.1cm]{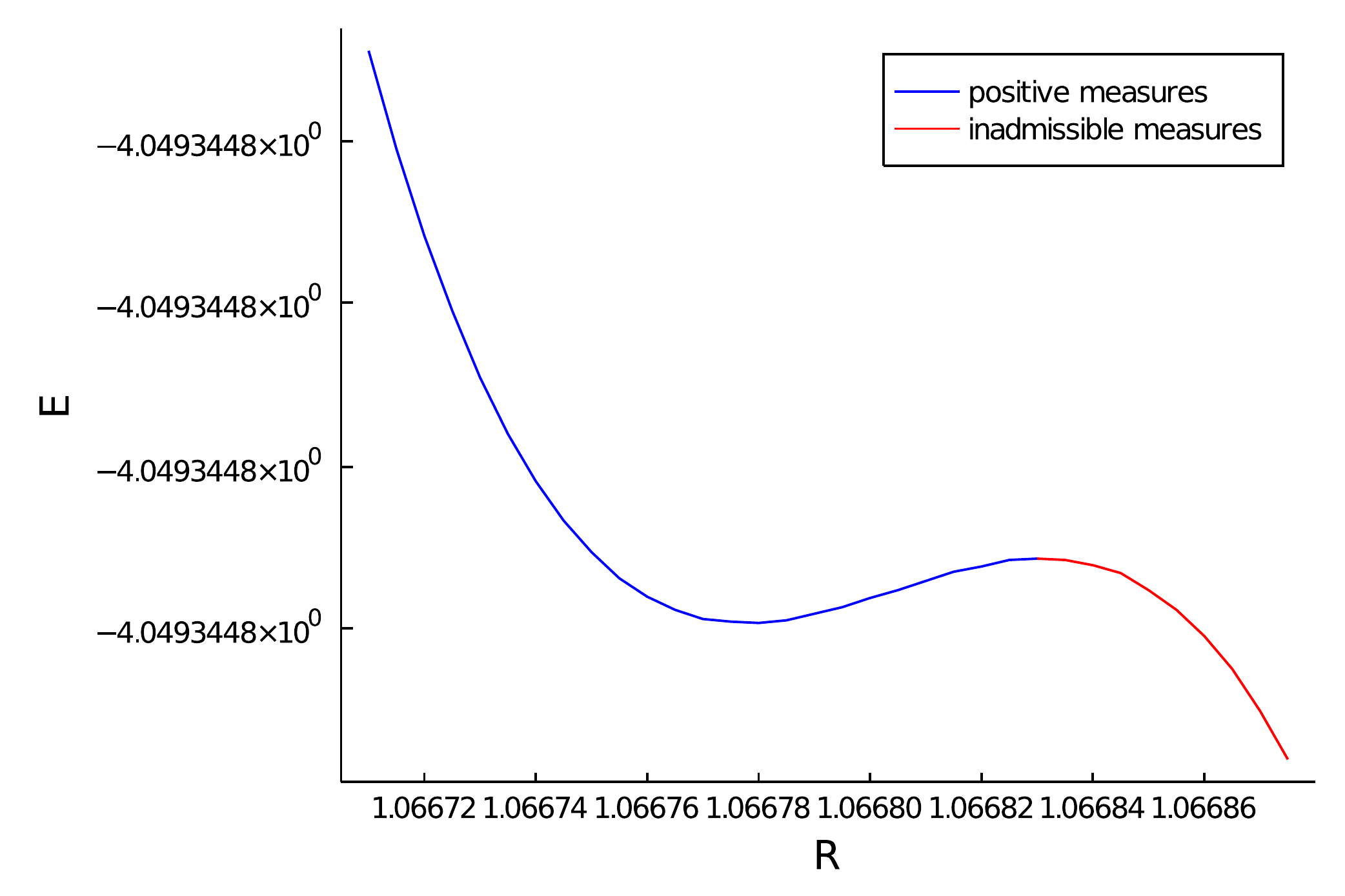} }}
         \subfloat[measure for (a)]
    {{ \centering \includegraphics[width=4.1cm]{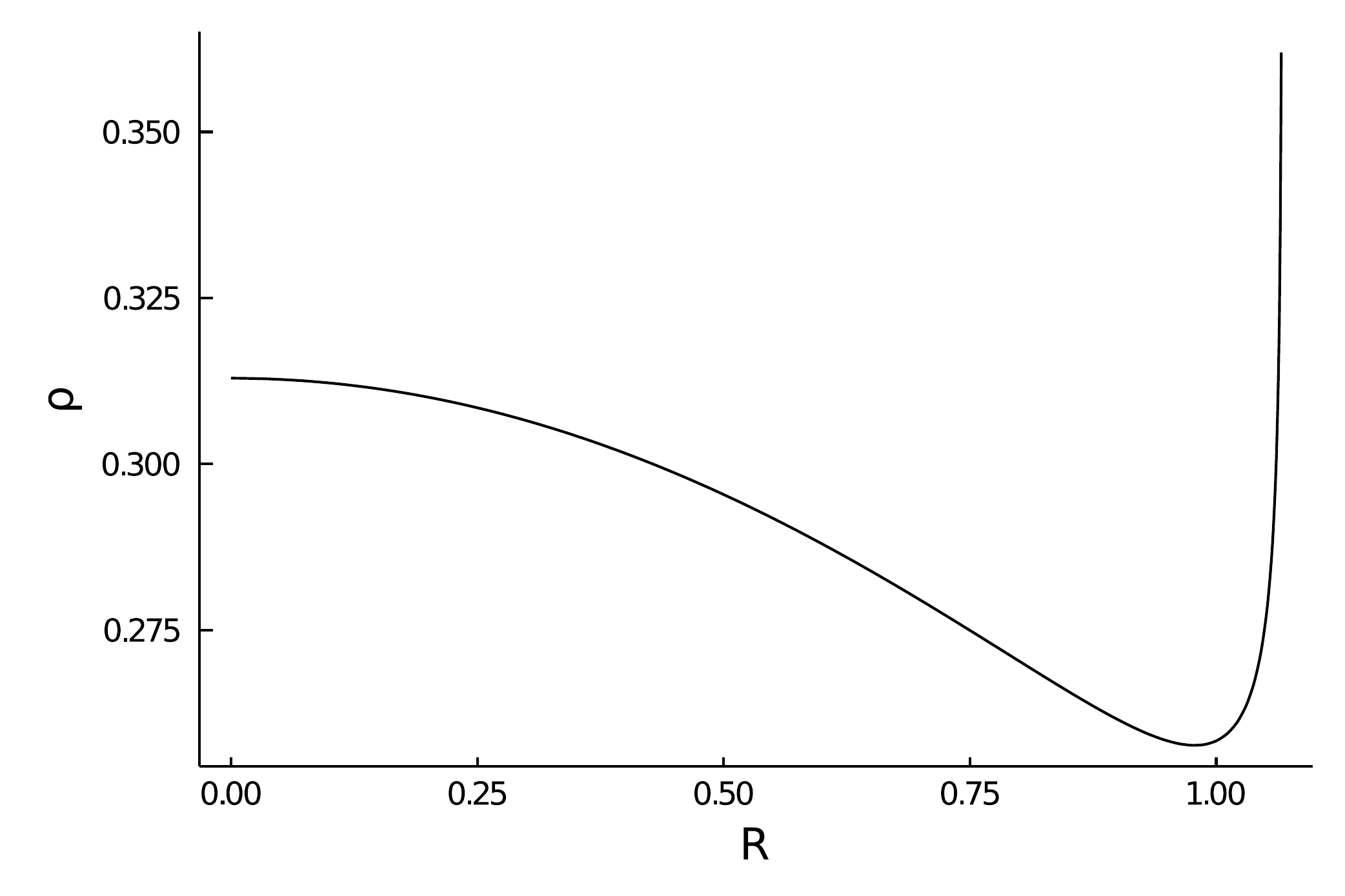} }}
             \subfloat[measure for (a) on disk]
    {{ \centering \includegraphics[width=4.1cm]{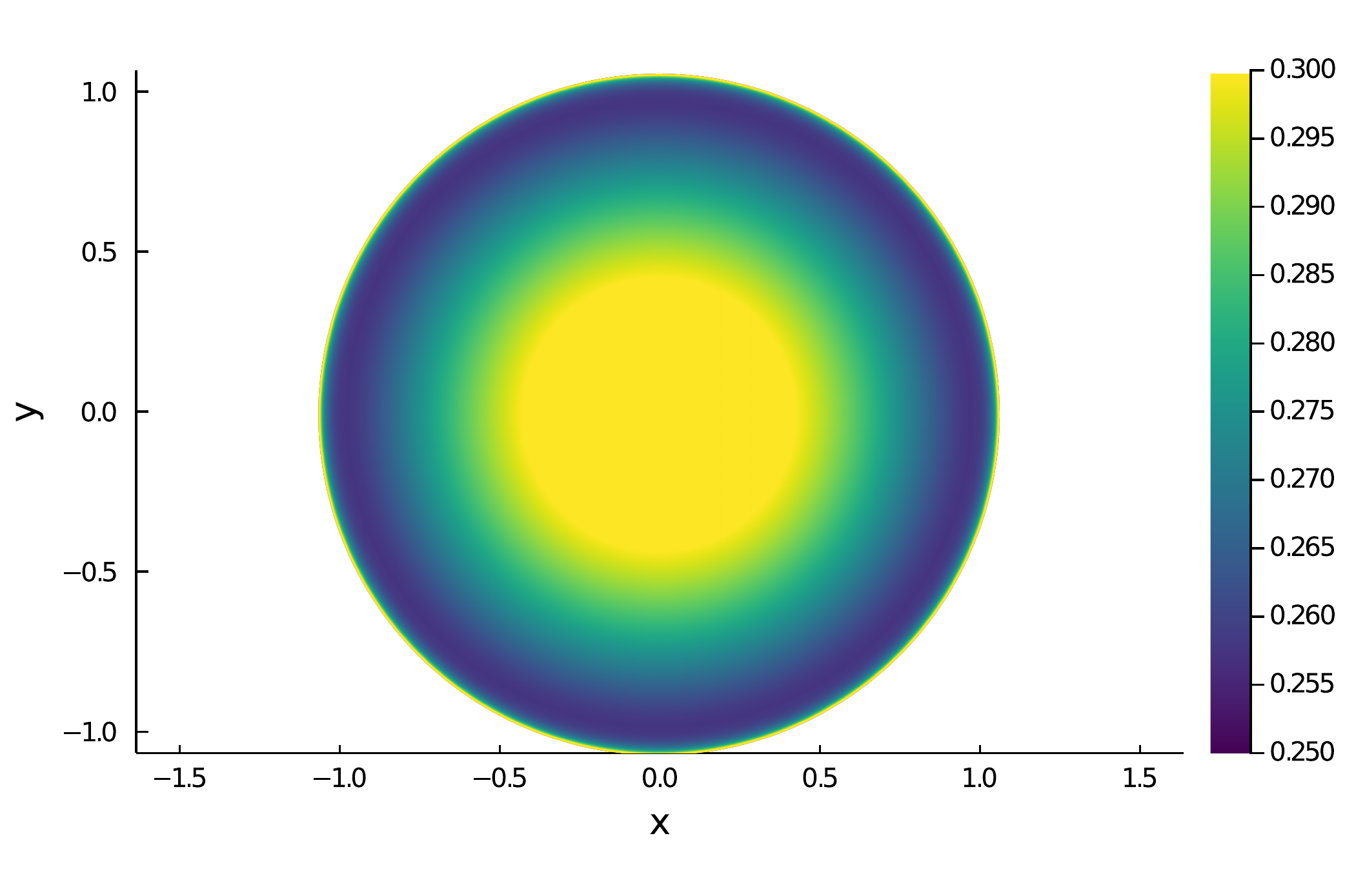} }}\\
         \subfloat[$(\alpha, \beta, d)  = (1.77,0.05,3)$]
    {{ \centering \includegraphics[width=4.1cm]{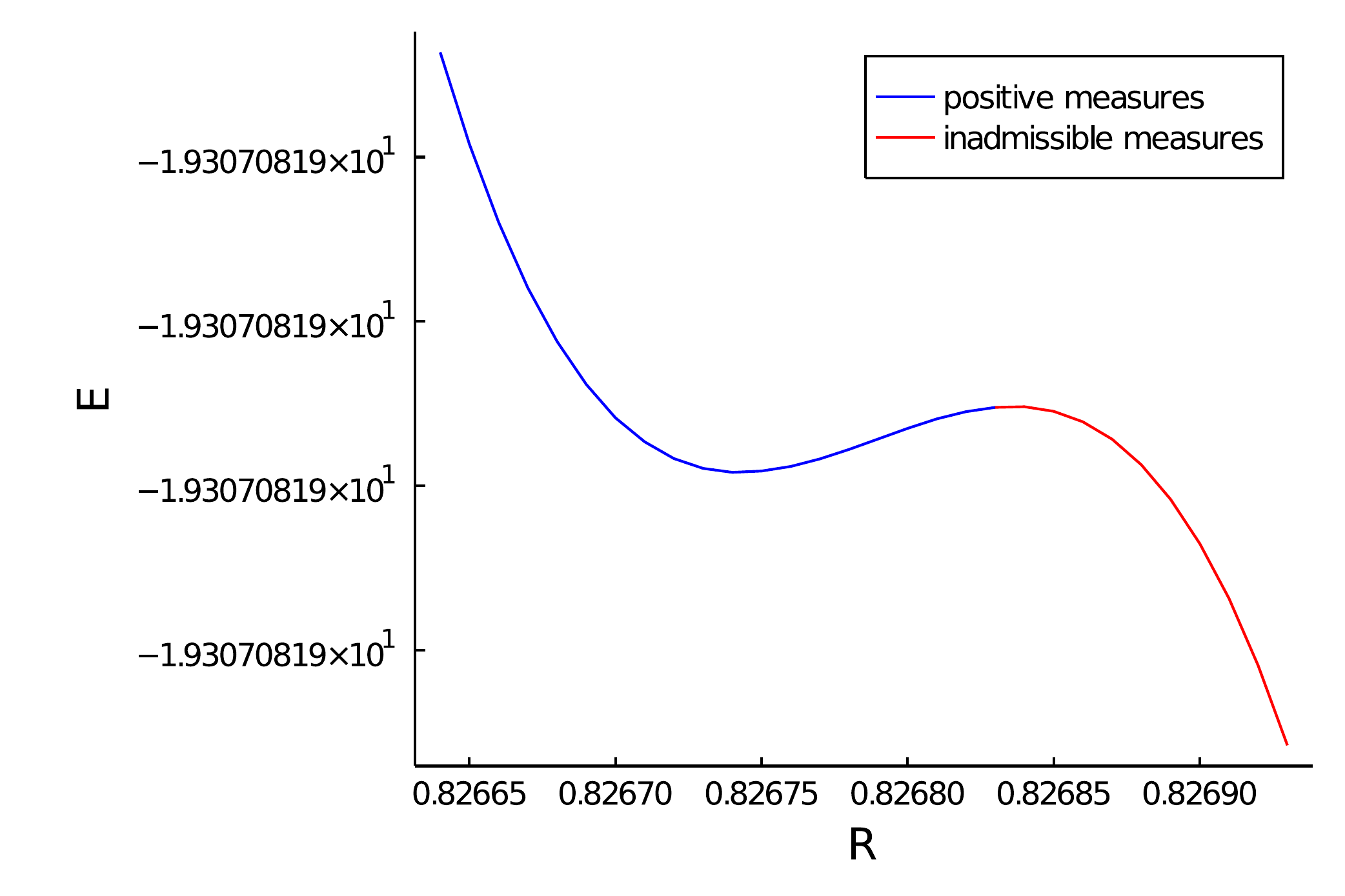} }}
         \subfloat[measure for (d)]
    {{ \centering \includegraphics[width=4.1cm]{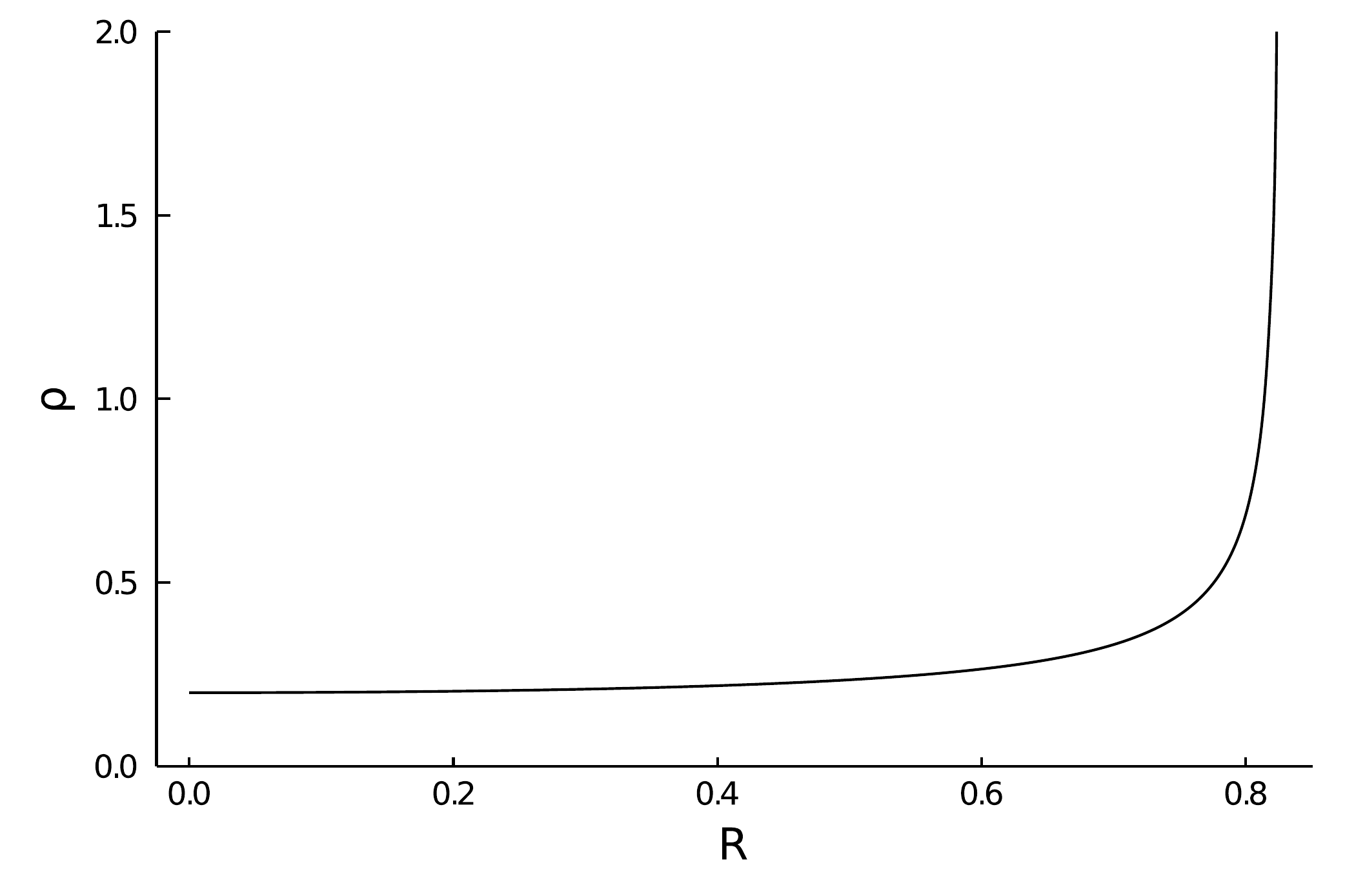} }}
    \caption{Left-most figures (a) and (d) show close-up of local energy minima for the stated example problem parameters. Second column figures (b) and (e) show the corresponding radially symmetric measures plotted as a function of $R=|y|$. Since (a) is a two-dimensional problem we also include a plot of the obtained measure on the full disk domain in (c).}%
    \label{fig:uniquesolutions}%
\end{figure}
\subsection{Comparison with analytic solutions in special cases}
In \cite{carrillo_explicit_2017}, the authors described a number of special case solutions in arbitrary dimension to the power law equilibrium measure problem where one of the powers is an even integer. In this section we use these solutions to verify the accuracy of our method and to provide a practical demonstration of its convergence.\\
We first consider the power law equilibrium measure problem defined by $\alpha = 2$. The solution to this problem was found in \cite{carrillo_explicit_2017} to be given by
\begin{align}
R =\left(\tfrac{\pi \Gamma \left(2-\frac{\beta }{2}\right) }{\sin \left(\frac{\pi  (\beta +d)}{2}\right)\Gamma \left(\frac{d}{2}+1\right) \Gamma \left(-\frac{d}{2}-\frac{\beta }{2}+1\right)}\right)^{\frac{1}{2-\beta }},\label{eq:a2radius}\\
\rho	(x) = -\tfrac{Md\Gamma\left(\frac{d}{2}\right) \sin \left(\frac{\pi  (\beta +d)}{2}\right)}{(\beta+d-2)\pi^{\frac{d+2}{2}}} (R^2-|x|^2)^{1-\frac{\beta+d}{2}}.\label{eq:a2measure}
\end{align}
In Figure \ref{fig:radiusa2} we plot the numerically computed energy with varying $R$, showing that the analytic solutions are exactly the local minima obtained via our method. Additionally, in Figure \ref{fig:analyticdiskcompare} we plot the computed equilibrium measure for a two dimensional example on the disk and include an absolute error heatmap.\\
\begin{figure}
     \subfloat[$(\alpha, \beta, d)  = (2, -0.44, 2)$]
    {{ \centering \includegraphics[width=4.1cm]{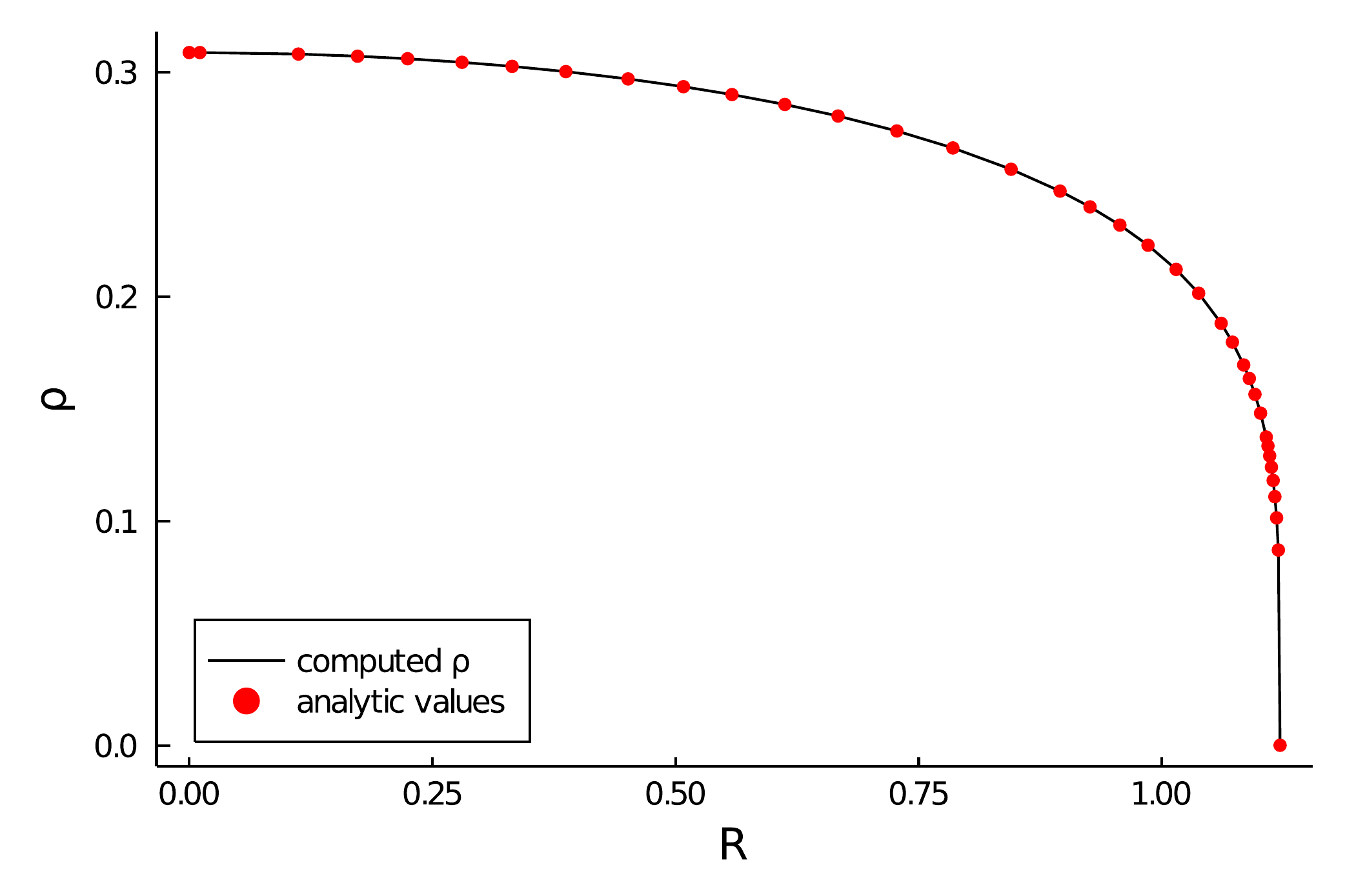} }}
         \subfloat[measure for (a) on disk]
    {{ \centering \includegraphics[width=4.1cm]{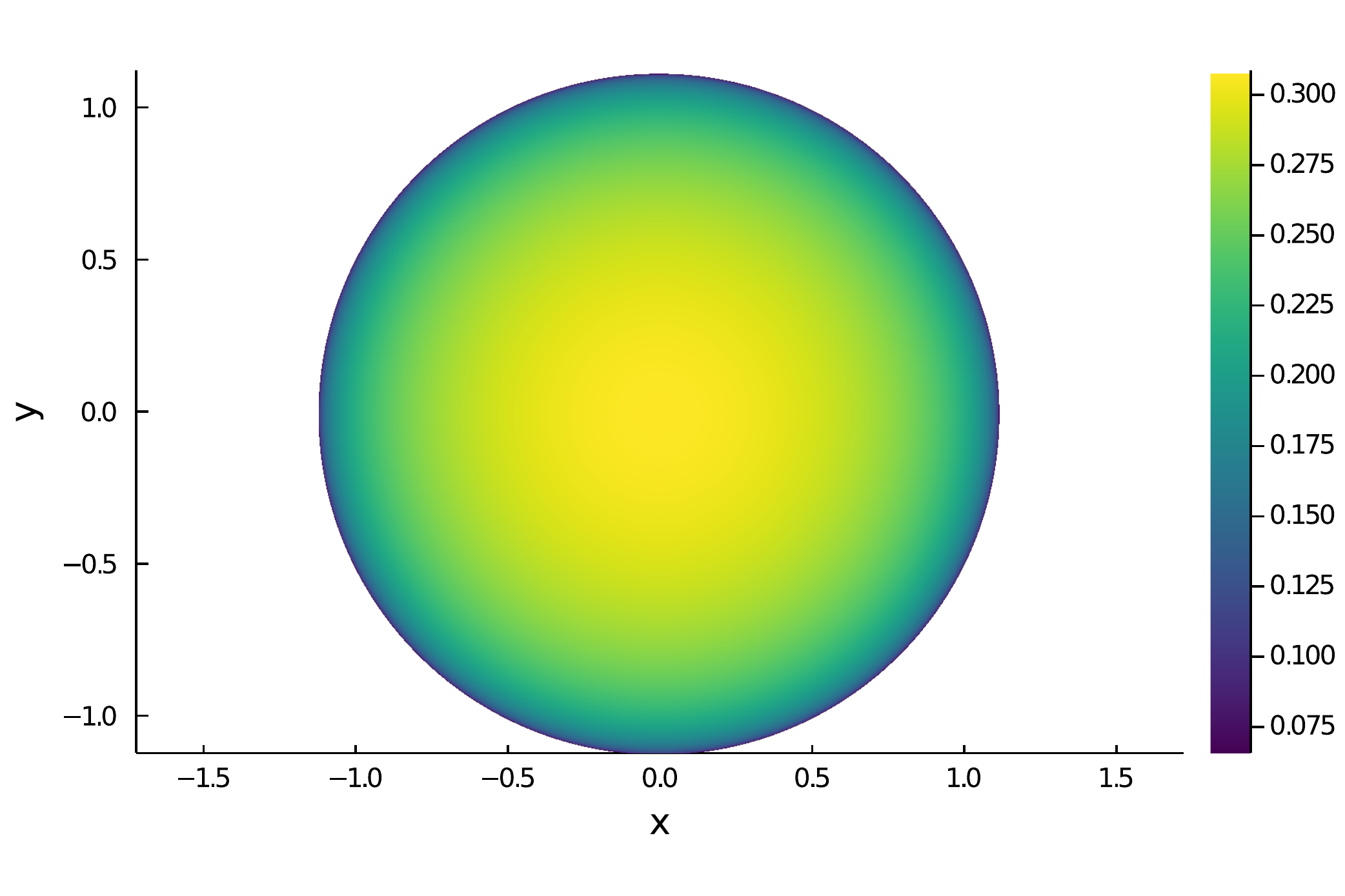} }}
         \subfloat[absolute errors for (a) on disk]
    {{ \centering \includegraphics[width=4.1cm]{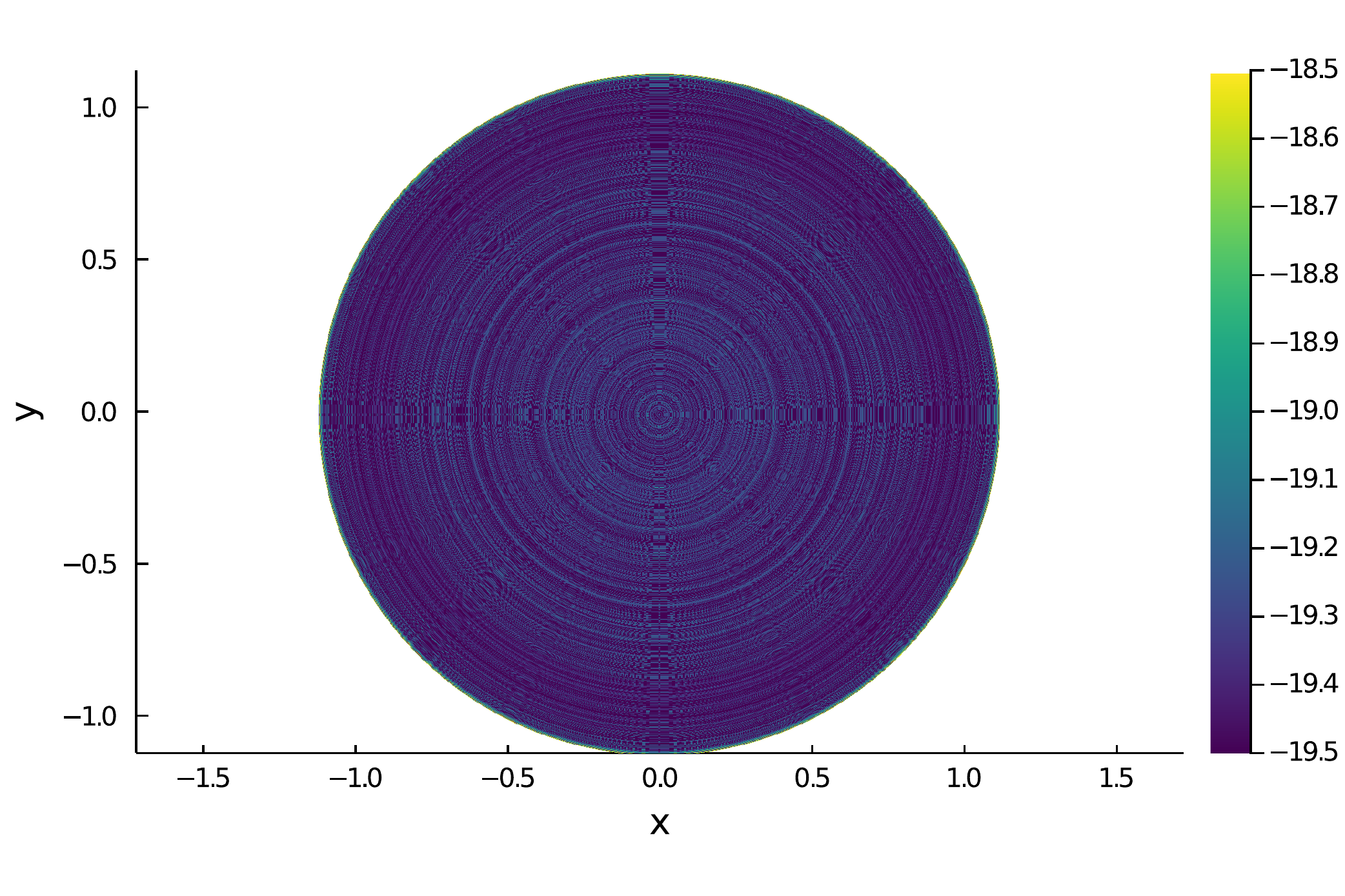} }}
    \caption{(a) shows the equilibrium measure for the stated example problem parameters as a radial plot including a comparison to the analytic solution, (b) shows the same computed equilibrium measure on its full disk domain and (c) shows a heatmap of the absolute errors on the disk domain, where the legend is logarithmic indicating order of magnitude.}%
    \label{fig:analyticdiskcompare}%
\end{figure}
\begin{figure}
     \subfloat[$(\alpha, \beta, d, M)  = (2, 1.2, 1, 1)$]
    {{ \centering \includegraphics[width=4.1cm]{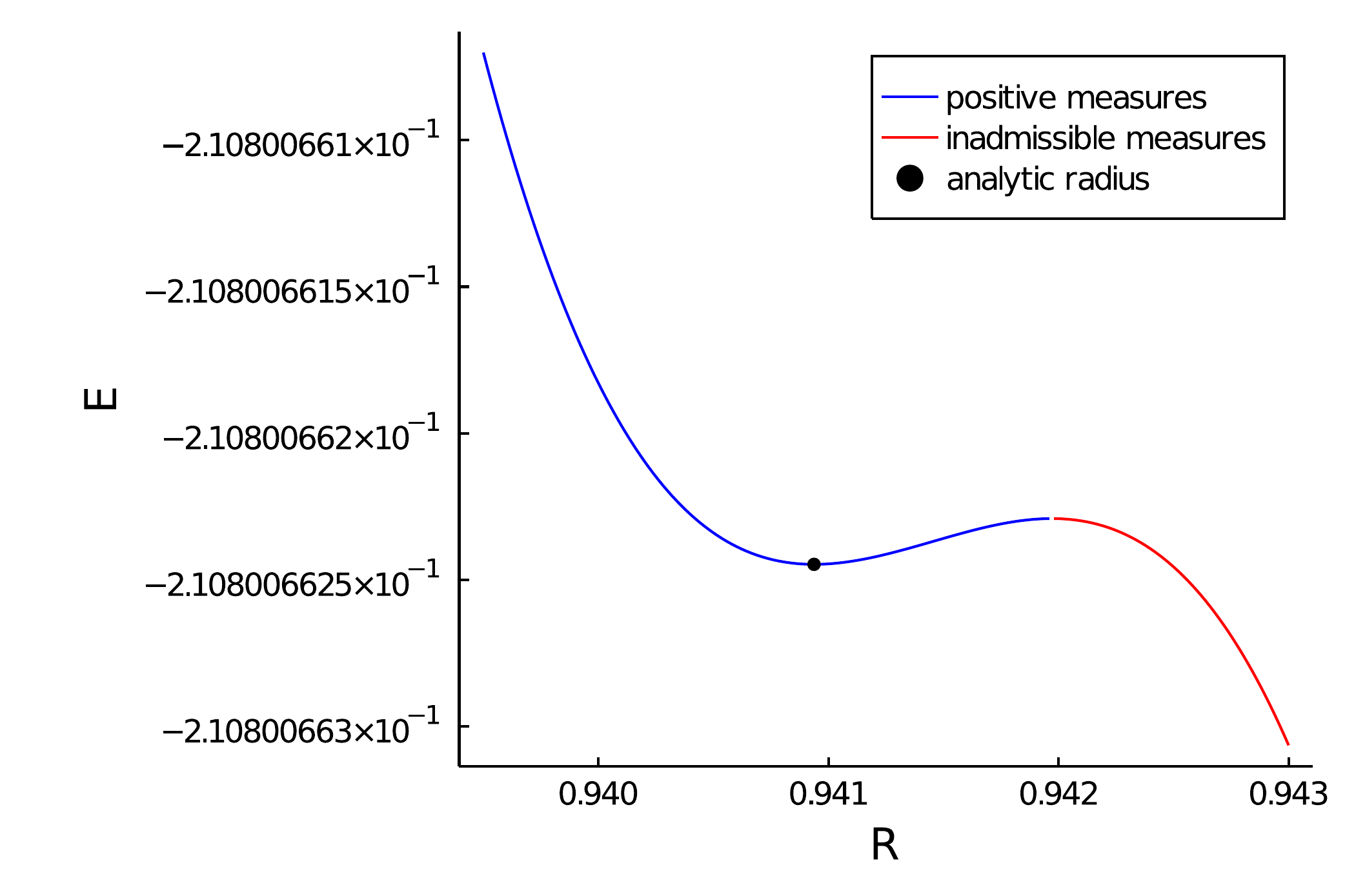} }}
         \subfloat[$(\alpha, \beta, d, M)  = (2, \frac{1}{3}, 2, 2.6)$]
    {{ \centering \includegraphics[width=4.1cm]{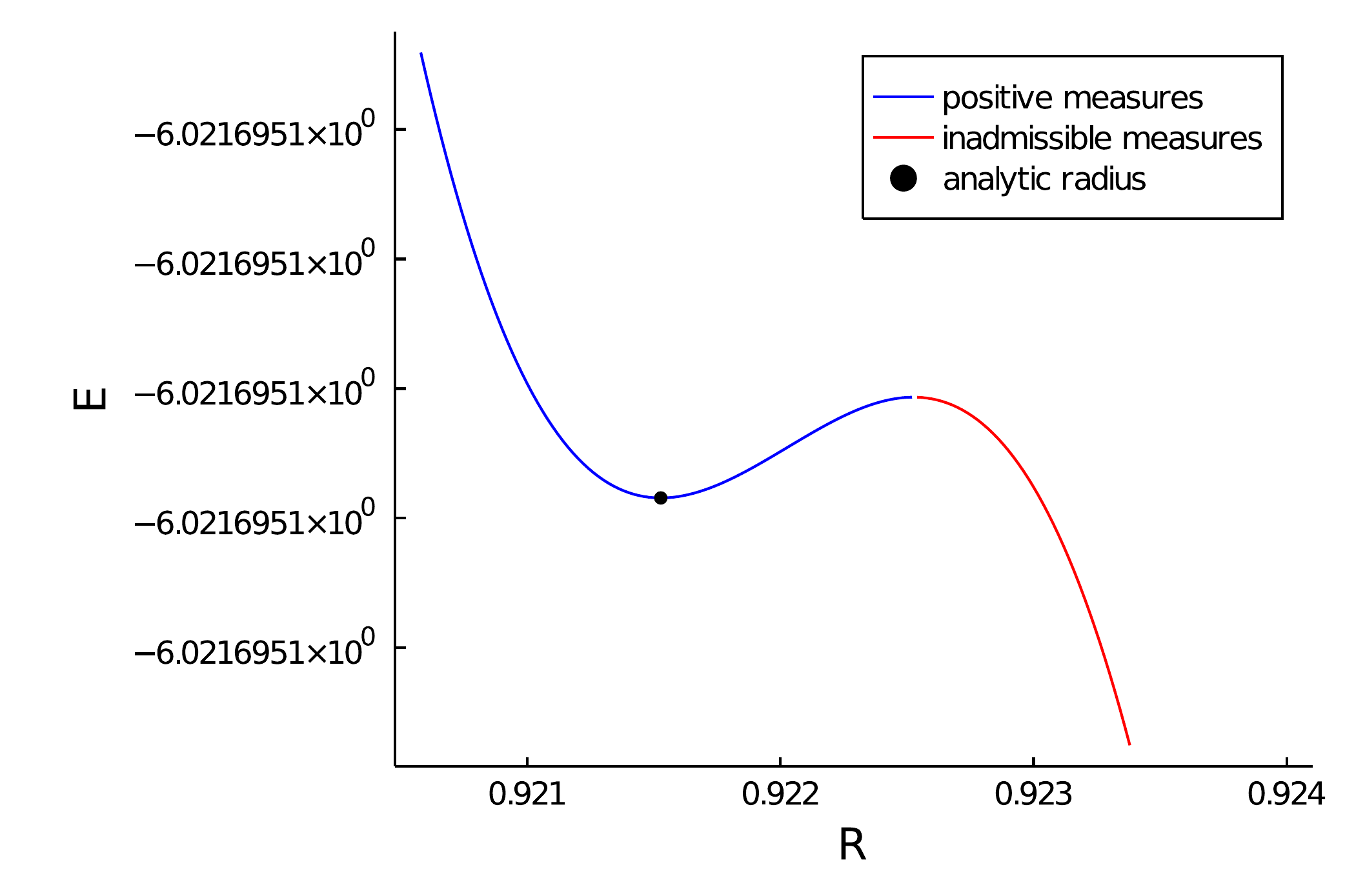} }}
             \subfloat[$(\alpha, \beta, d, M)  = (2, -0.5, 3, 1)$]
    {{ \centering \includegraphics[width=4.1cm]{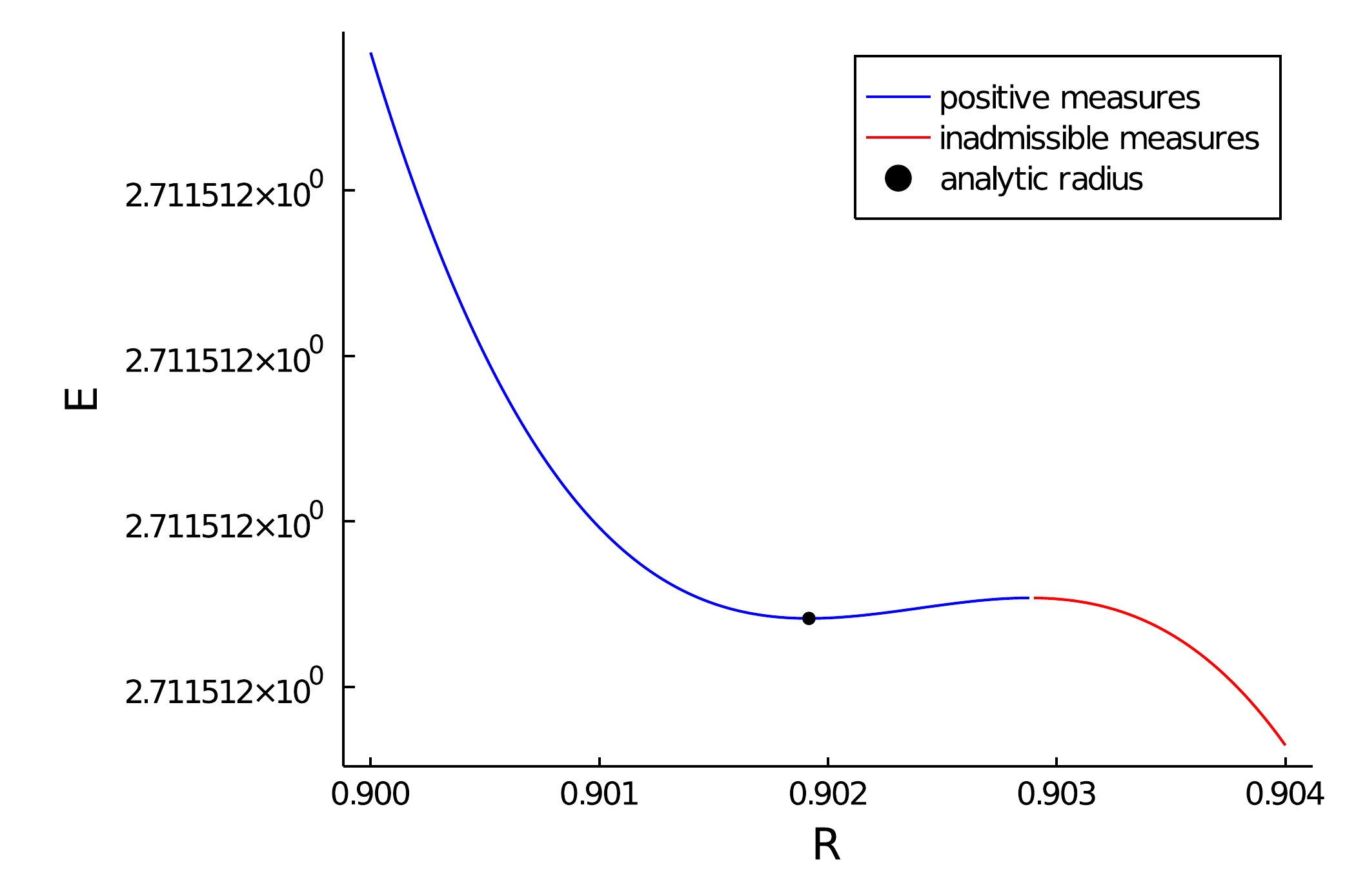} }}\\
         \subfloat[$(\alpha, \beta, d, M)  = (2, -2.5, 4, 0.5)$]
    {{ \centering \includegraphics[width=4.1cm]{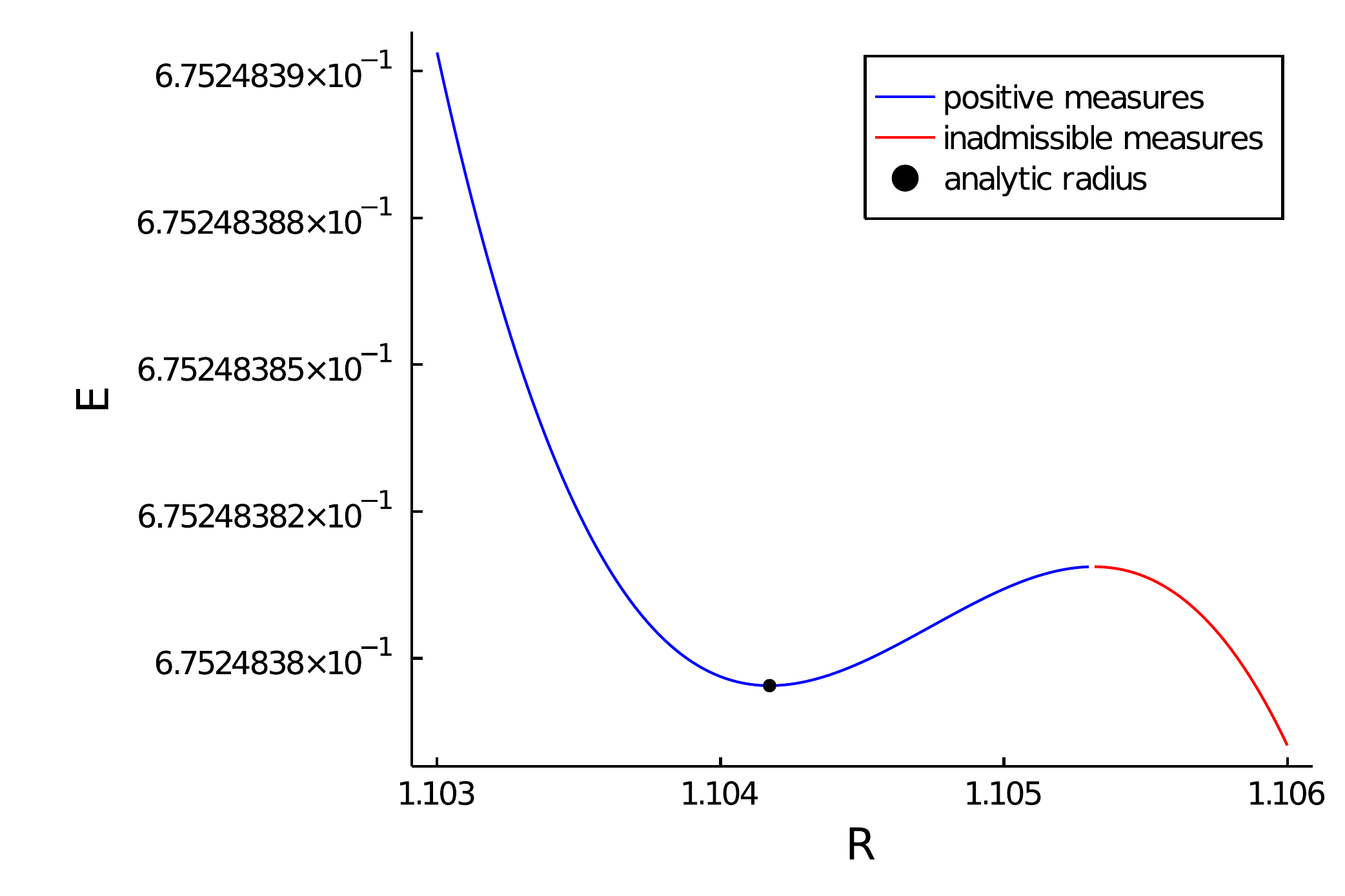} }}
         \subfloat[$(\alpha, \beta, d, M)  = (2, -\frac{4\pi}{5}, 5, 1)$]
    {{ \centering \includegraphics[width=4.1cm]{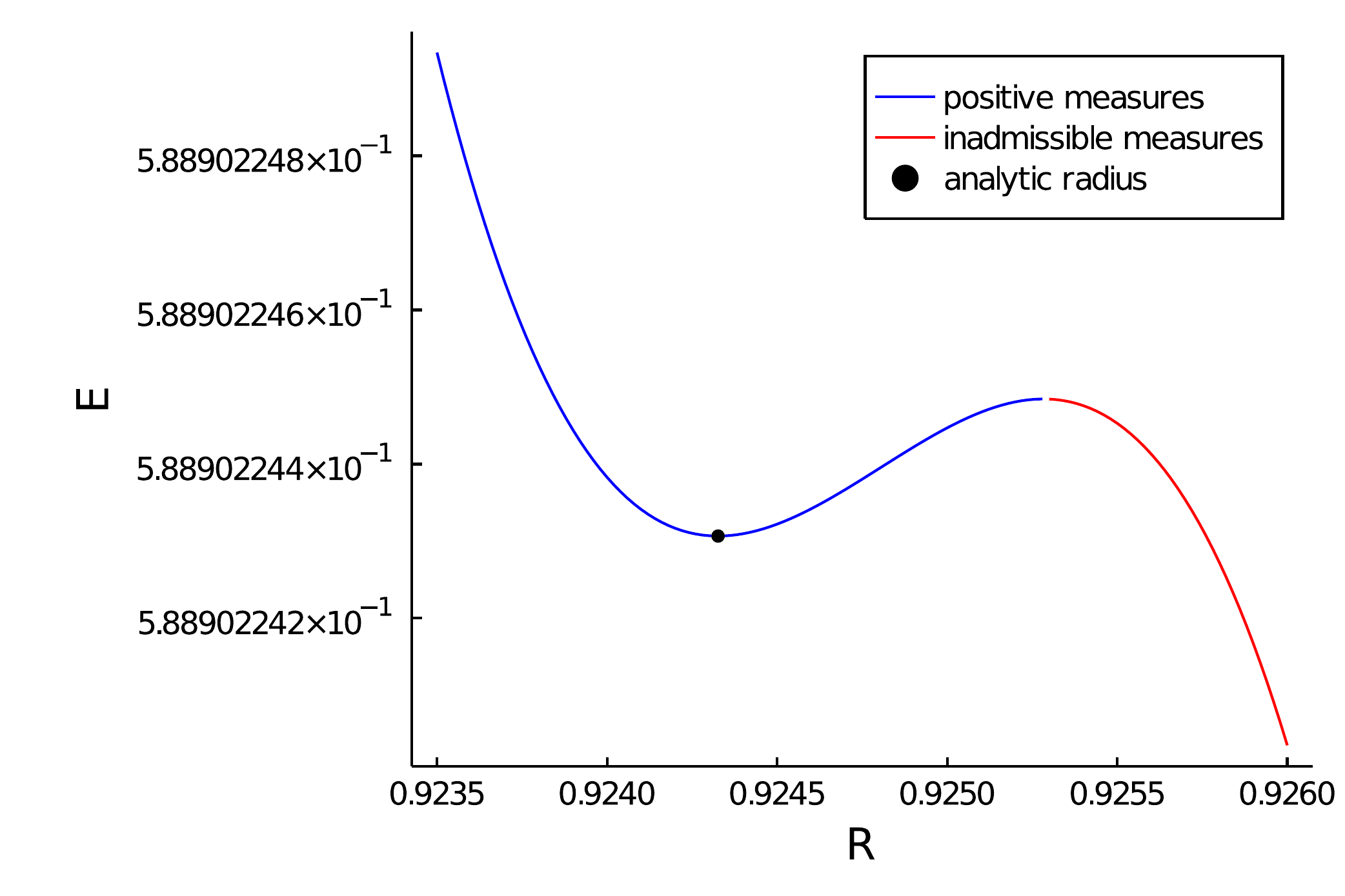} }}
             \subfloat[$(\alpha, \beta, d, M)  = (2, -3.2, 6, 1)$]
    {{ \centering \includegraphics[width=4.1cm]{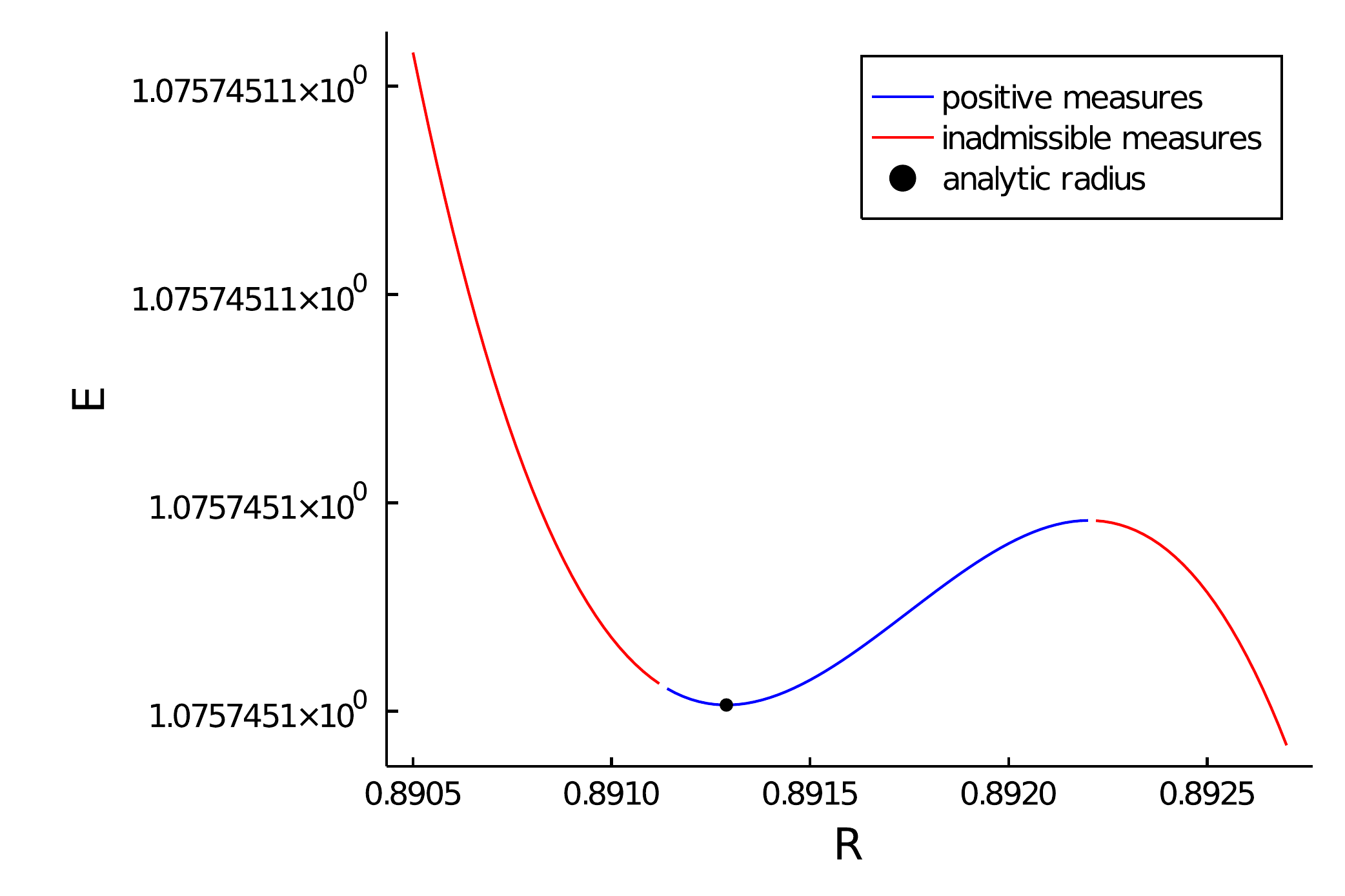} }}
    \caption{Close-up of local energy minima for the stated example problem parameters for $\alpha=2$ in dimensions $1$ to $6$, showing the radius of the analytic solution in Eq. \eqref{eq:a2radius} is consistent with the local minima obtained by our method.}%
    \label{fig:radiusa2}%
\end{figure}
Now for $\alpha = 4$, the analytic solution as found in \cite{carrillo_explicit_2017} may be written as
\begin{align}
R = \left[ \frac{d(d+2)\Gamma(\frac{d}{2})}{2\Gamma(\frac{\beta+d}{2})\Gamma(2-\frac{\beta}{2}))} \left(\frac{1}{4-\beta}+\frac{1}{\sqrt{(2-\beta)(6-\beta)}} \right) \right]^{-\frac{1}{4-\beta}},\label{eq:a4radius}\\
\rho(x) = (R^2-|x|^2)^{1-\frac{\beta+d}{2}}\left( A_1 R^2 + A_2 (R^2-|x|^2)\right)\label{eq:a4measure},
\end{align}
with the constants
\begin{align*}
A_1 &= \frac{\Gamma(\frac{d}{2})}{\pi^{\frac{d}{2}}} \frac{d(d+2)M}{2 B\left(\frac{\beta+d}{2},2-\frac{\beta+d}{2}\right)}\left[\frac{1}{\sqrt{(2-\beta)(6-\beta)}}+\frac{1}{2-\beta} \ \right],\\
A_2 &=  \frac{\Gamma(\frac{d}{2})}{\pi^{\frac{d}{2}}} \frac{d(d+2)M}{B\left(\frac{\beta+d}{2},3-\frac{\beta+d}{2} \right) 4 (\beta-2) M}.
\end{align*}
As before, we plot the numerically computed energy with varying $R$ along with the analytic radius in Figure \ref{fig:radiusa4}.\\
\begin{figure}
     \subfloat[$(\alpha, \beta, d, M)  = (4, 0.5, 2, 1)$]
    {{ \centering \includegraphics[width=4.1cm]{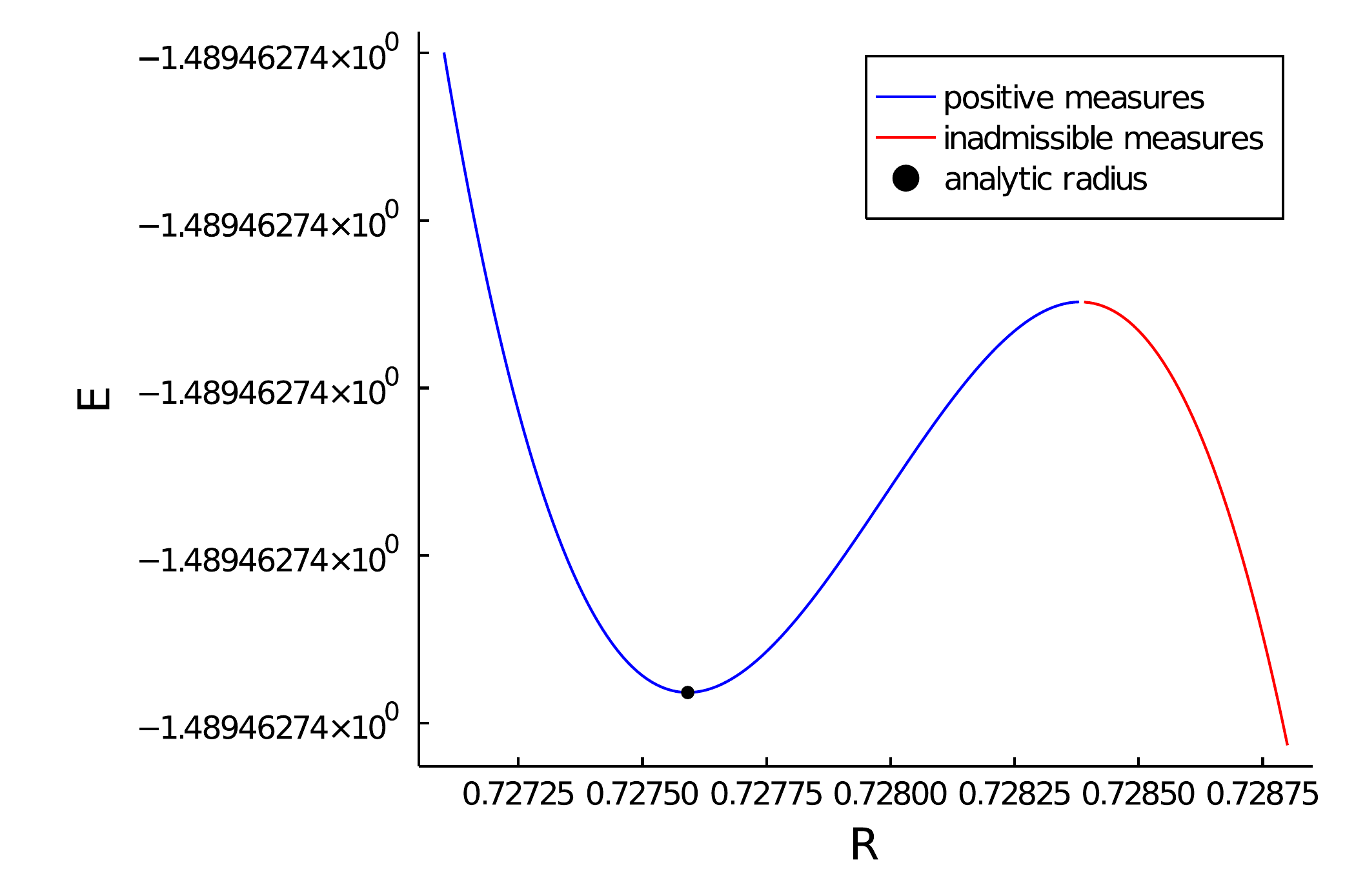} }}
         \subfloat[$(\alpha, \beta, d, M)  = (4, -1.1, 3, 1)$]
    {{ \centering \includegraphics[width=4.1cm]{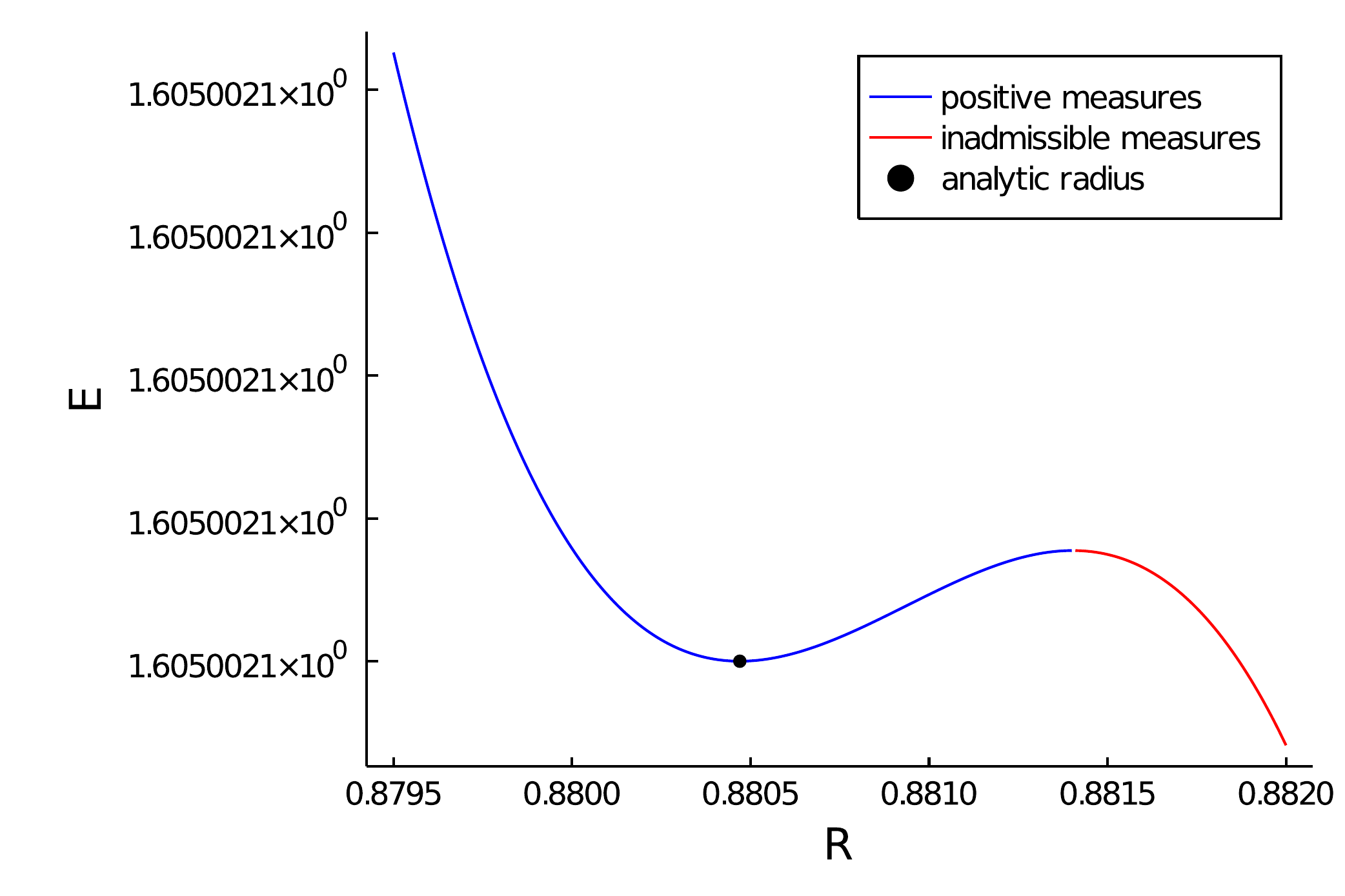} }}
             \subfloat[$(\alpha, \beta, d, M)  = (4, -3.9, 6, 2)$]
    {{ \centering \includegraphics[width=4.1cm]{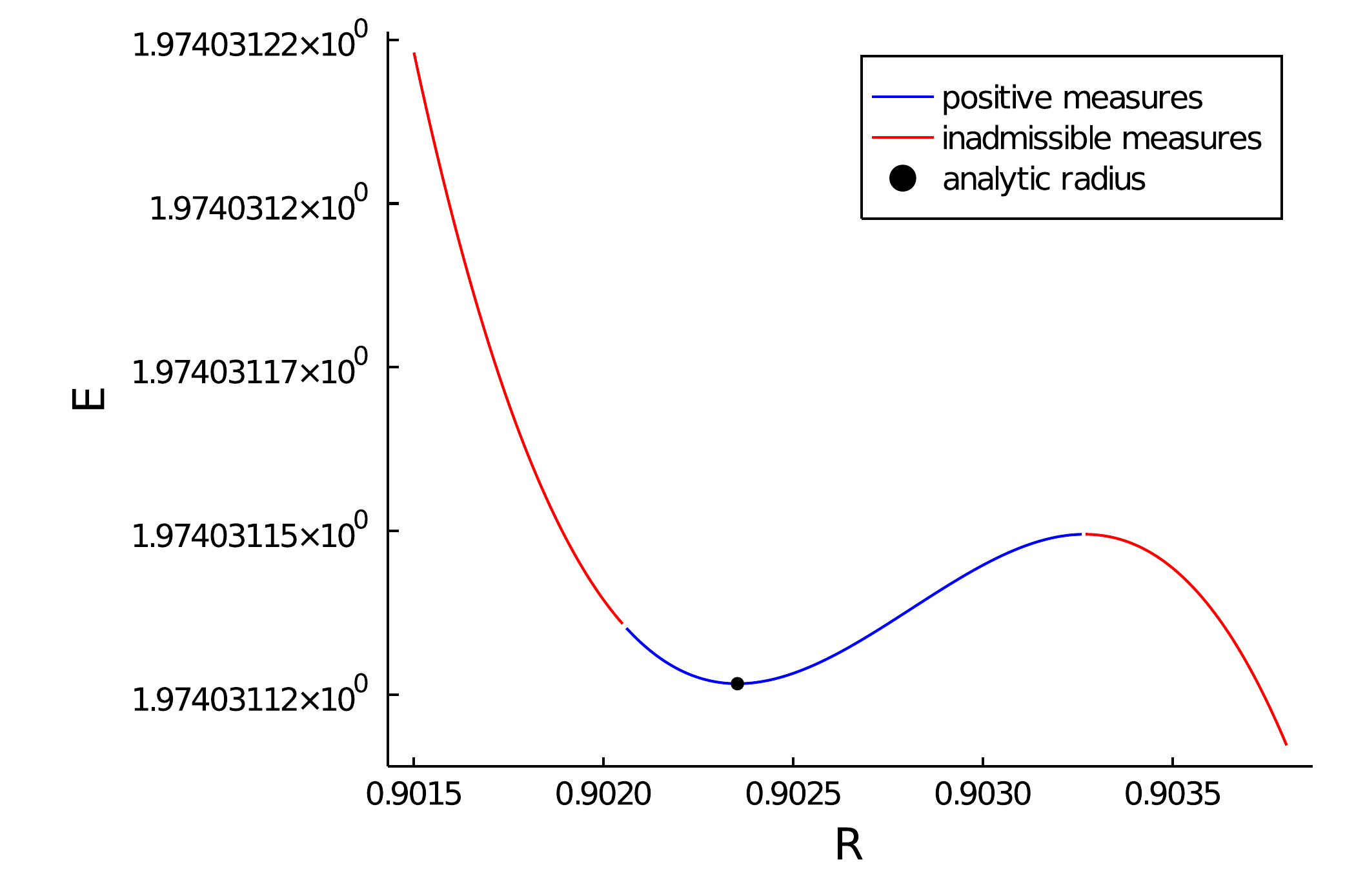} }}
    \caption{Close-up of local energy minima for the stated example problem parameters for $\alpha=4$ in different dimensions, showing the radius of the analytic solution in Eq. \eqref{eq:a4radius} is consistent with the local minima obtained by our method.}%
    \label{fig:radiusa4}%
\end{figure}
The reason these special cases have known analytic solutions is that they are particularly well-behaved -- in fact, using our method both of these solutions may be computed to high precision using arbitrary precision floating point calculations with only one or two polynomial orders of approximation, meaning that these solutions can be computed almost instantaneously. Note how the energy as a function of $R$ in Figure \ref{fig:radiusa2} and \ref{fig:radiusa4} show clear local minima at the analytic radius, further highlighting how these are special well-behaved cases compared to the general case. A glance at the form of the solutions tells us most of the work is already done by making the correct choice of weighted basis. To get a meaningful visualization of the convergence rate of our method, we can thus not rely on examples with presently known analytic solutions as the plot would simply show an almost instant drop to any precision of our choosing. We show an example of a convergence plot in Figure \ref{fig:analyticconvergence} for completeness for an example where $\alpha=4$. The problem of visualizing convergence meaningfully is thus left to the next section.\\
The fact that an analytic radius is known for these special cases does however give us a tool to visualize the error incurred by the precision chosen for the optimization. In Figure \ref{fig:radiusdelta} we thus plot the maximum absolute error incurred for the obtained measure when deviating from the true radius solutions. As was observed in \cite{gutleb_computing_2020} for the one-dimensional case, errors in the computation of the radius propagate linearly to the error in the measure. This can be used as a guiding principle for the convergence conditions of the optimization method to obtain a result with a particular desired accuracy.

\begin{figure}
     \centering \includegraphics[width=8cm]{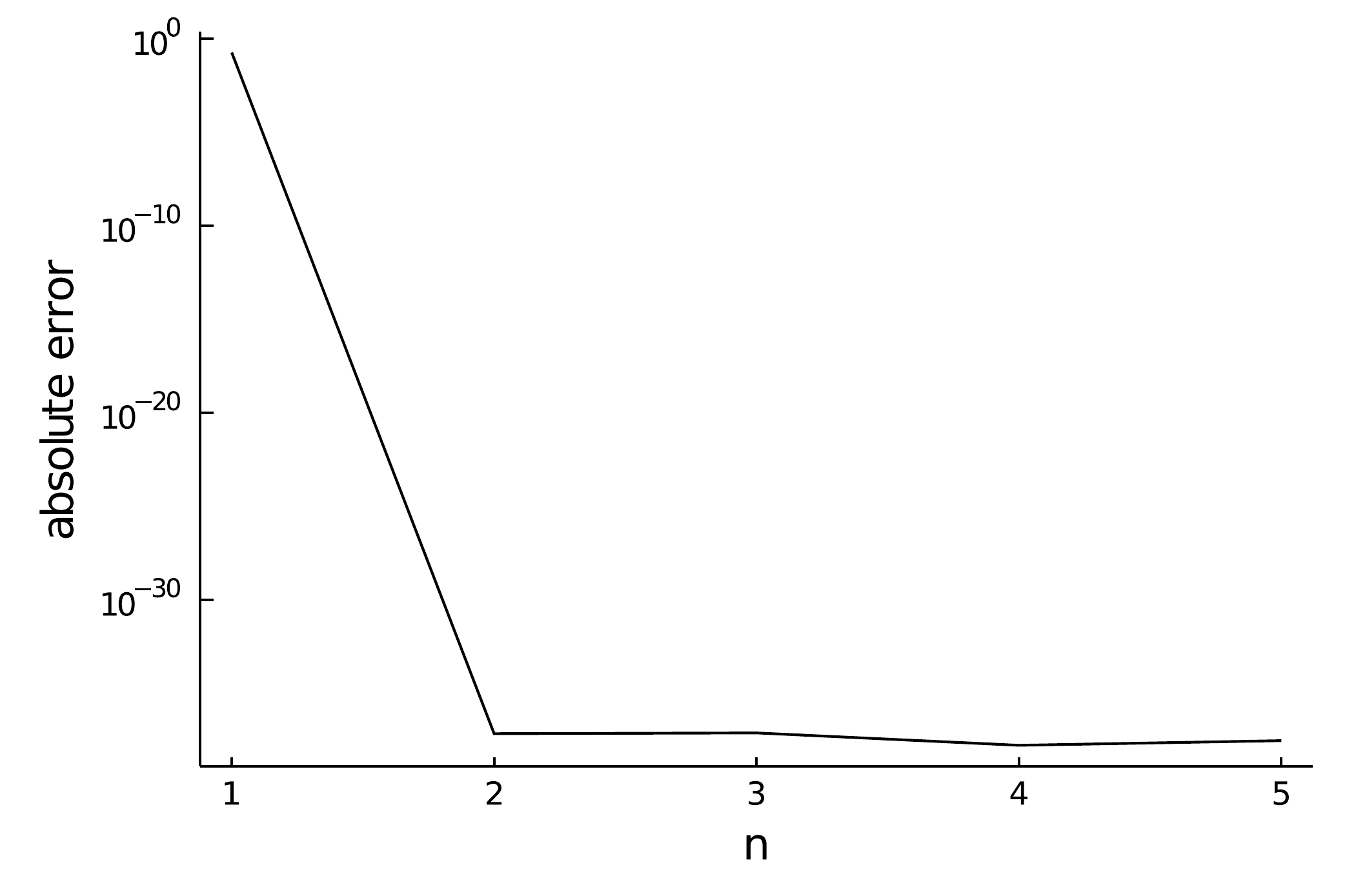}
    \caption{Semi-logarithmic convergence plot of maximum absolute error of the obtained measure for the example problem $(\alpha, \beta, d, M)  = (4, -3.9, 6, 2)$ compared to the analytic solution in Eq. \eqref{eq:a4measure}. Due to the nature of the solution for these special cases, our method almost instantly converges to arbitrary accuracy which may be increased via the chosen precision of arbitrary floating point arithmetic.}%
    \label{fig:analyticconvergence}%
\end{figure}

\begin{figure}
     \subfloat[$(\alpha, \beta, d, M)  = (2, -1.2, 3, 1)$]
    {{ \centering \includegraphics[width=6.1cm]{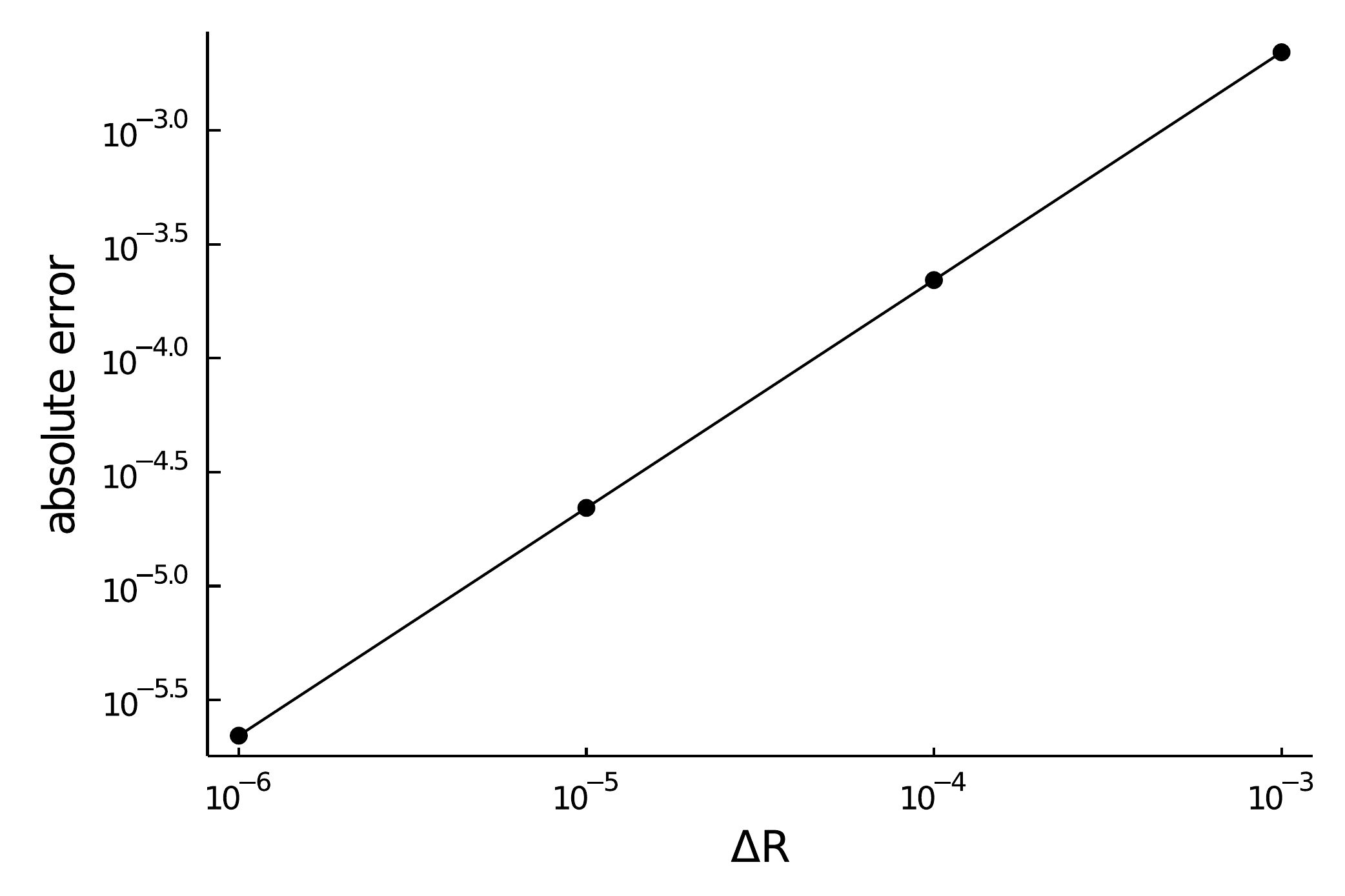} }}
         \subfloat[$(\alpha, \beta, d, M)  = (4, -3.7, 6, 2)$]
    {{ \centering \includegraphics[width=6.1cm]{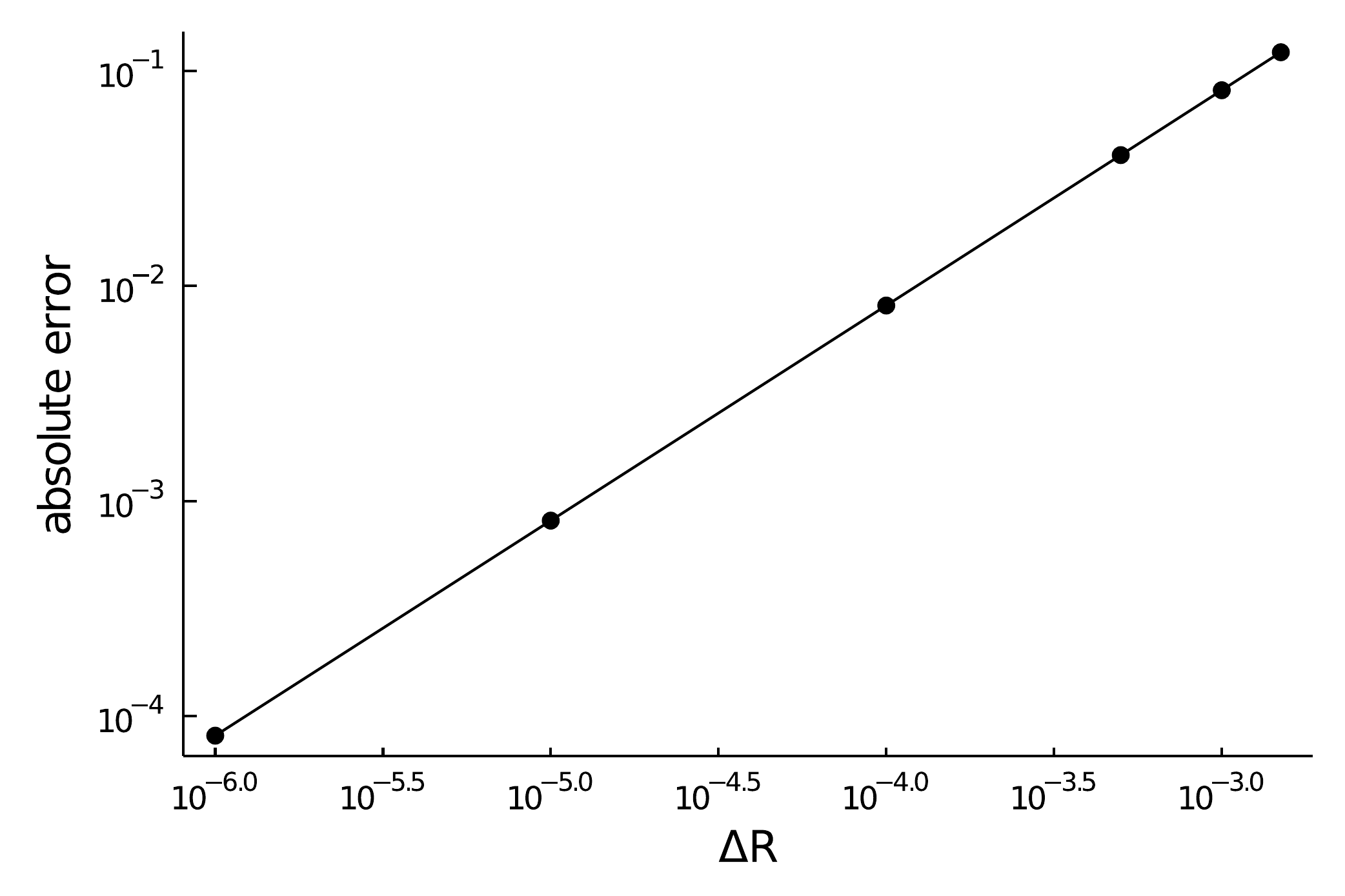} }}
    \caption{Logarithmic plots of the absolute deviation from the analytic solutions when computing the measure for the perturbed radius $R+\Delta R$, with $\Delta R$ on the $x$-axis. To obtain this error, the measures are compared in their normalized form on the unit ball. This shows the linear dependence of our method's accuracy on the chosen convergence cut-off of the optimization method.}%
    \label{fig:radiusdelta}%
\end{figure}

\subsection{Numerical convergence properties and regularization}\label{sec:numericalconvergence}
As seen in the previous section, convergence tests with known analytic solutions result in almost immediate convergence, making it difficult to sensibly visualize the properties of our method. We thus supplement the previous discussion with two example problems with no known analytic solution chosen to not fall into the well-behaved special cases. We furthermore explore the effect of Tikhonov regularization on the stability of the method. \\
A natural way to think about convergence for solutions computed in terms of polynomial expansions is to measure whether and how fast the coefficients of the computed solution decay as the order of approximation increases. For particularly well-behaved functions the decay of such coefficients is exponential. In cases where only a single power law integral with an external potential are present we obtain exponential convergence in $n$ as the right hand side hypergeometric functions are just polynomials when choosing the appropriate basis. Likewise, in the special attractive-repulsive case in which one of the powers is an even integer, the right-hand side of both power law integrals can simultaneously be brought into polynomial form, again leading to exponential convergence.\\
\begin{figure}
     \subfloat[$(\alpha, \beta, d)  = (\pi, -\frac{4\pi}{5}, 5)$]
    {{ \centering \includegraphics[width=4.1cm]{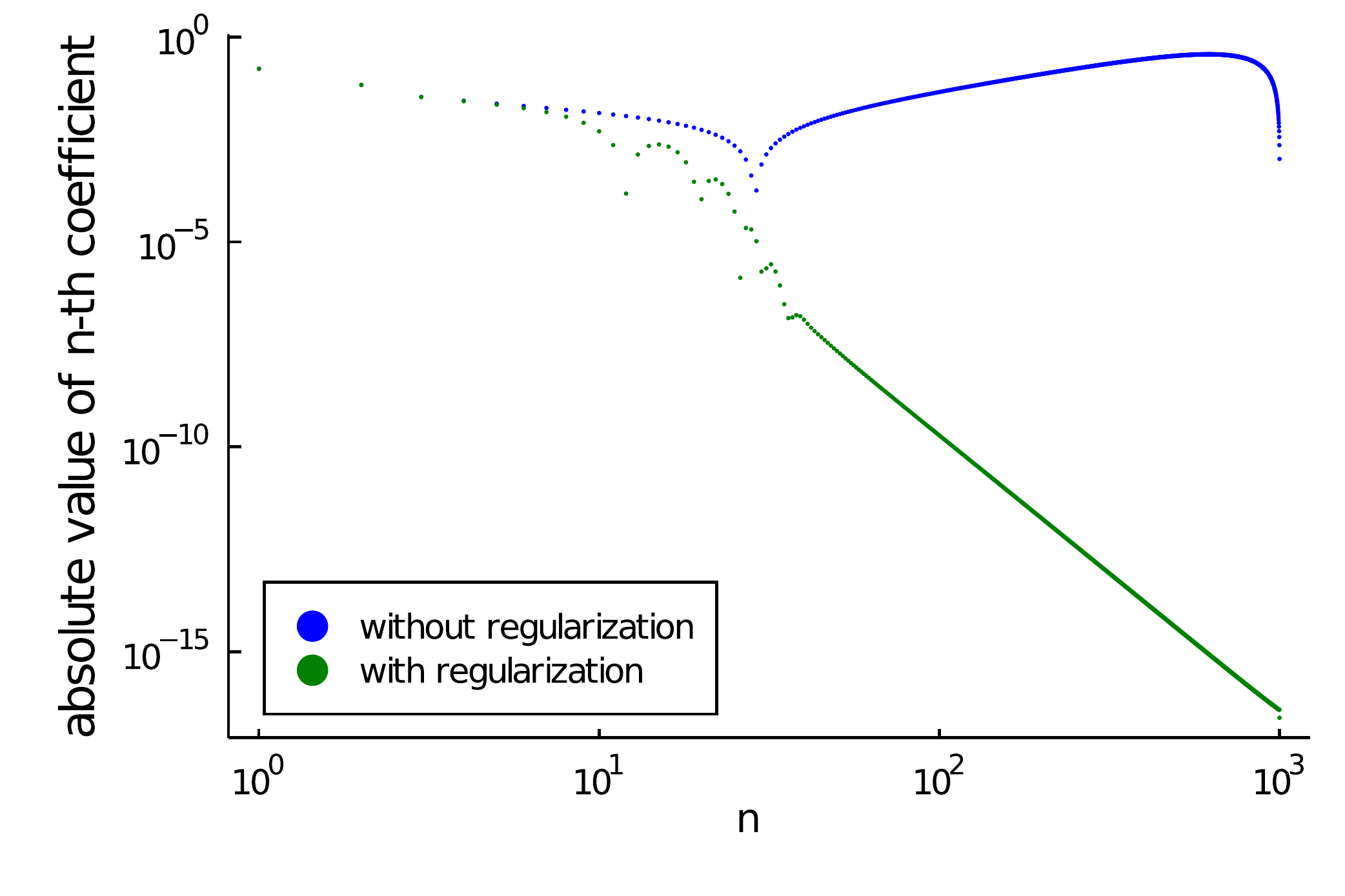} }}
         \subfloat[$(\alpha, \beta, d)  = (1, 0.3, 2)$]
    {{ \centering \includegraphics[width=4.1cm]{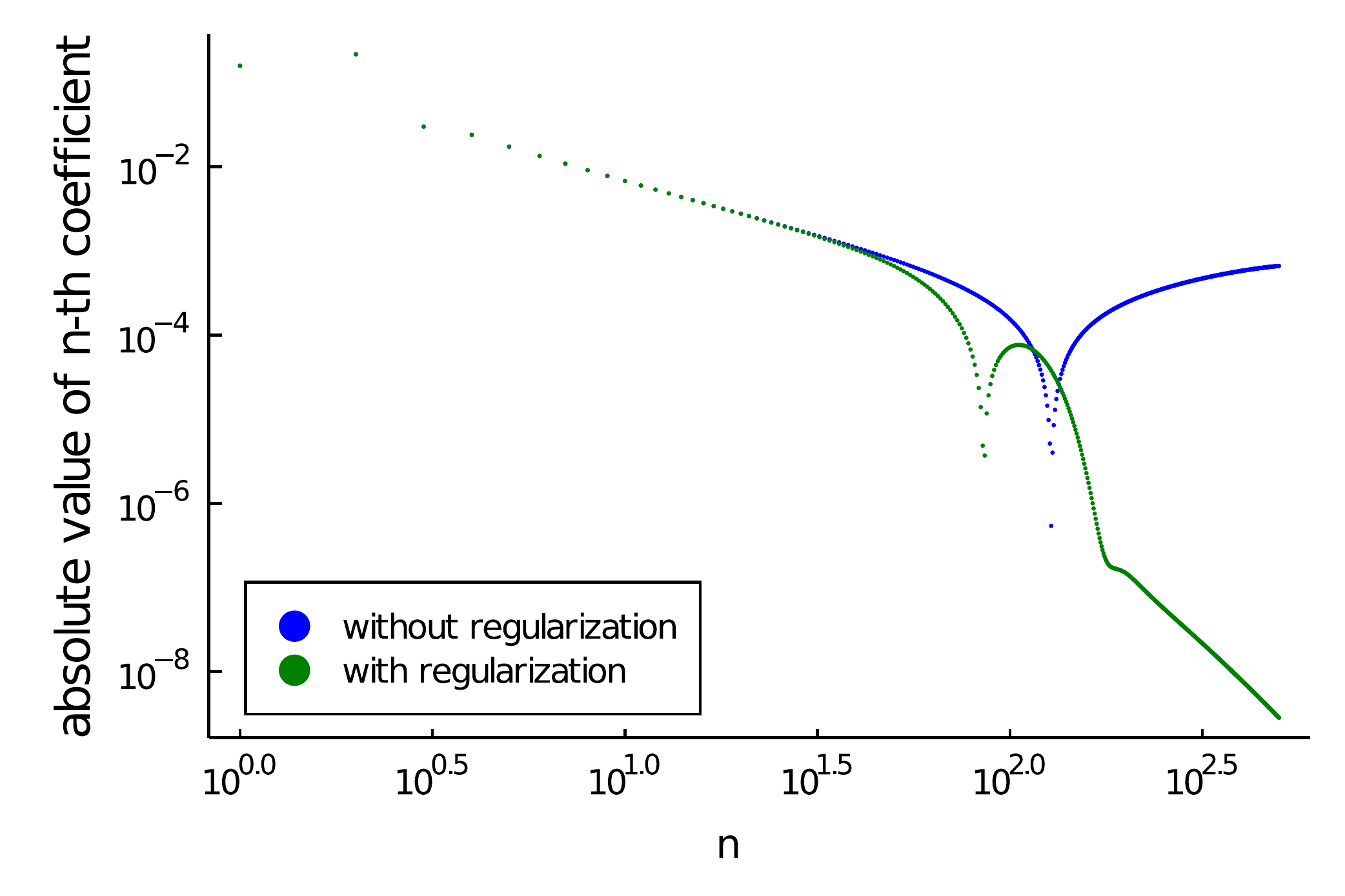} }}
             \subfloat[measures in (b) near origin]
    {{ \centering \includegraphics[width=4.1cm]{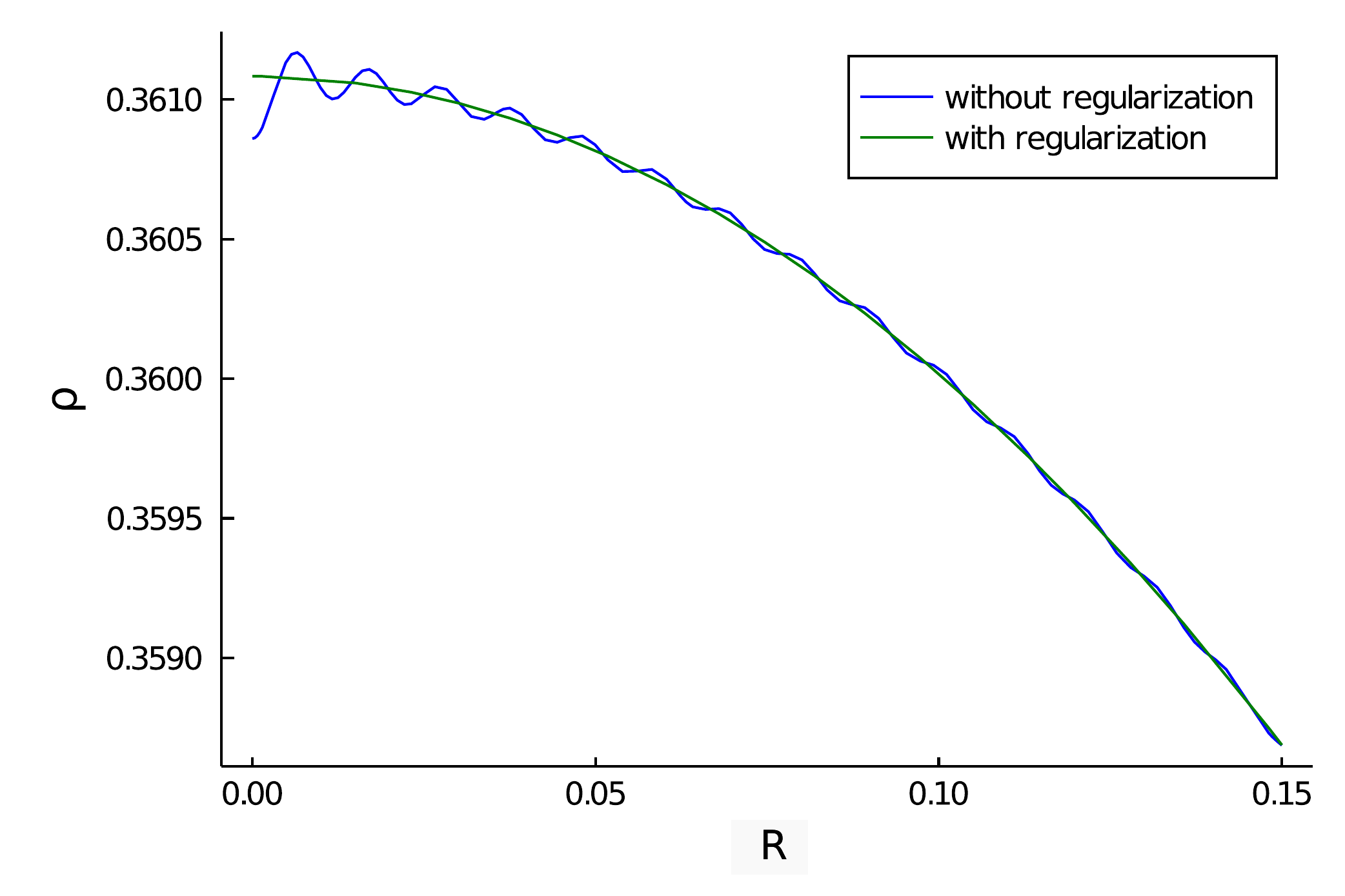} }}
    \caption{(a) and (b) show absolute value of the $n$-th coefficient in the computed solutions of the indicated equilibrium measure problems with and without using Tikhonov regularization. The observed instability in the coefficients without regularization in (b) is realized in the measure as incorrect oscillations near the origin, as seen in the zoomed-in segment in (c).}%
    \label{fig:coeffconvergence}%
\end{figure}
The convergence behaviour of the coefficients computed for $\rho$ may be changed by choosing a different basis under the integrals, in particular via Remark \ref{rem:genericweightparam}. The natural choices are those which lead to one of the two operators being banded such that the bandwidth is minimized. In Figure \ref{fig:coeffconvergence}(a) and (b) we plot the absolute value of the coefficients for two generic examples for the regularized and non-regularized methods. The regularized method as well as the early behaviour of the non-regularized method shows linear convergence in the coefficients for the generic attractive-repulsive. For large $n$, the instability of the coefficients in the non-regularized method is significant enough to be apparent in the resulting measures themselves. As seen in the example in Figure \ref{fig:coeffconvergence}(c) this instability causes incorrect oscillations to appear near the origin which grow with the order of approximation. These oscillations are not observed in the regularized method which converges consistently.\\
The linear decay behaviour of the coefficients with the regularized method as seen in Figure \ref{fig:coeffconvergence} is expected due to the mismatch of the singularity of the Jacobi polynomial basis with the singularity of the kernel in the generic attractive-repulsive case. In the above-mentioned cases showing exponential convergence we are able to exactly resolve the analytic singularities.\\
Finally, it should be noted that our method via Theorem \ref{theorem:hypergeomgeneral} and \ref{cor:generalrecurrence} does not inherently require the basis on the right hand side to be a particular set of Jacobi polynomials. As long as the operator mapping is accurately and consistently taken care of using the above results, the measure $\rho$ will be obtained in a weighted Jacobi polynomial basis with any choice of right-hand side basis functions. As the basis in which the measure $\rho$ is obtained does not change when doing this, only the computational efficiency in generating the operators may be affected by changing the right-hand side basis, not the $n$-convergence rate of the coefficients of $\rho$ itself.

\subsection{Comparison with particle swarm simulations}
As mentioned in the introduction, the current go-to methods for numerically approximating equilibrium measure solutions are particle swarm simulations, see e.g. \cite{carrillo2010particle} where this approach and the observed phenomenon of gap formation for certain parameter ranges are explored. In the interesting parameter ranges, these simulations require some thousands to ten thousands of particles in one-dimension to meaningfully converge. Already in two dimensions, particle swarm simulations are close to the limit of what is computationally feasible. In contrast, our method converges independent of the dimension, which allows investigating problems far out of reach of the conventional methods. In this section, we compare the results of our method with a standard version of these particle swarm simulations in two dimensions. We also present some comparisons with lower density particle simulations in three dimensions, which have to be interpreted with care. In all of these cases, the results of our method are consistent with what is obtained via particle simulations.\\
In Figure \ref{fig:particleanalyticcompare} we compare the particle solution to our spectral method for an example intentionally chosen to be a special case in which an analytic solution is also known via (\ref{eq:a2radius}-\ref{eq:a2measure}), allowing a three-way comparison via Figure  \ref{fig:analyticdiskcompare}. We present a further generic case one-dimensional and two dimensional example in Figure \ref{fig:particlecomparegeneric}, without known analytic solutions.
\begin{figure}
\centering
     \subfloat[$(\alpha, \beta, d)  = (2, -0.44, 2)$]
    {{ \centering \includegraphics[width=6cm]{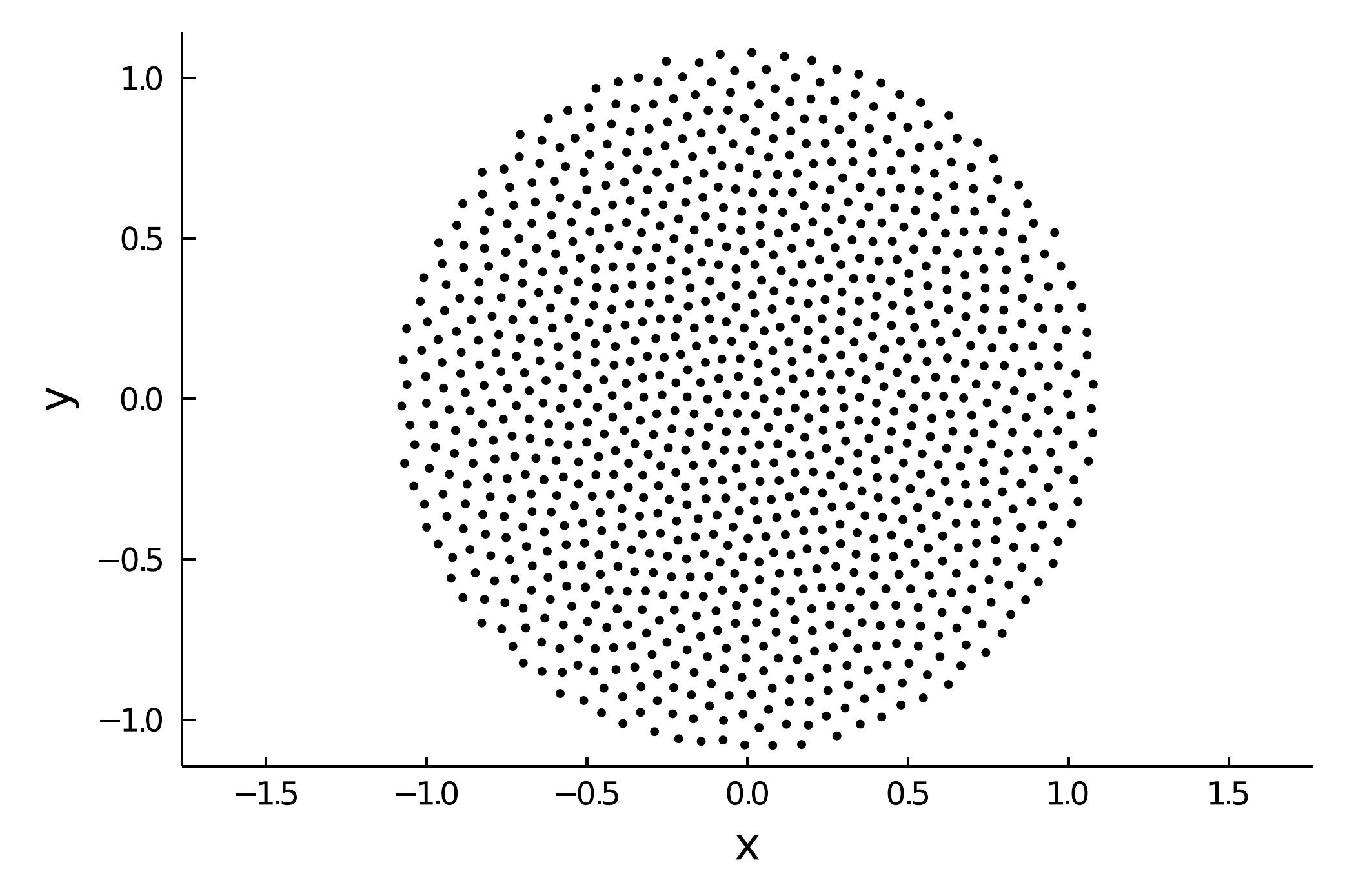} }}
         \subfloat[2D histogram based on particles in (a)]
    {{ \centering \includegraphics[width=6.1cm]{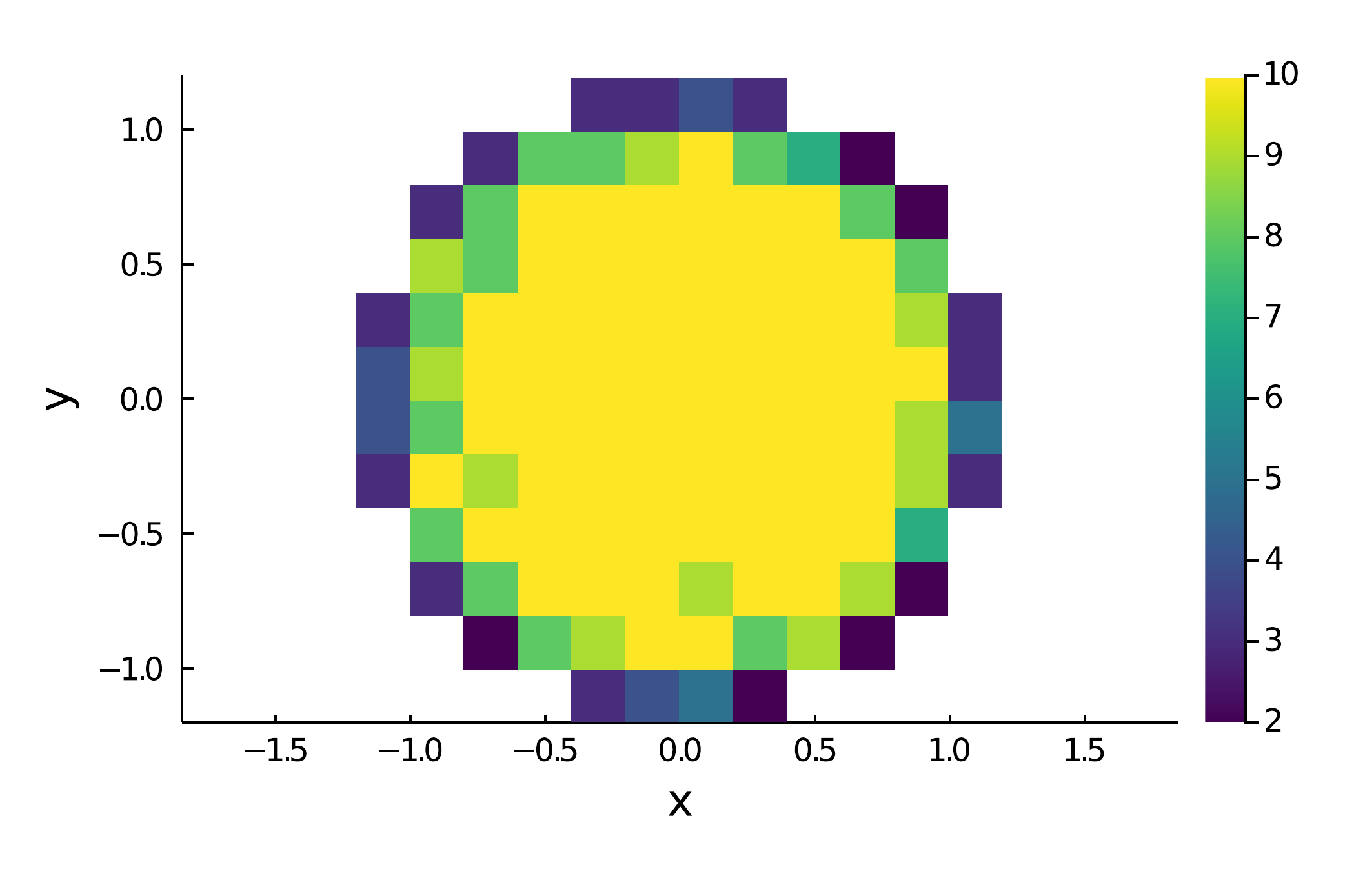} }}
    \caption{(a) shows the results of a particle simulation for the indicated equilibrium measure problem with $1000$ initially randomly distributed particles. (b) shows a two-dimensional histogram approximating the density based on the particle simulation. Note that the problem parameters here are intentionally the same as in Figure \ref{fig:analyticdiskcompare} and that an analytic solution is known in this special case.}%
    \label{fig:particleanalyticcompare}%
\end{figure}

\begin{figure}
\centering
     \subfloat[$(\alpha, \beta, d)  = (1.3, 1.1, 2)$]
    {{ \centering \includegraphics[width=4cm]{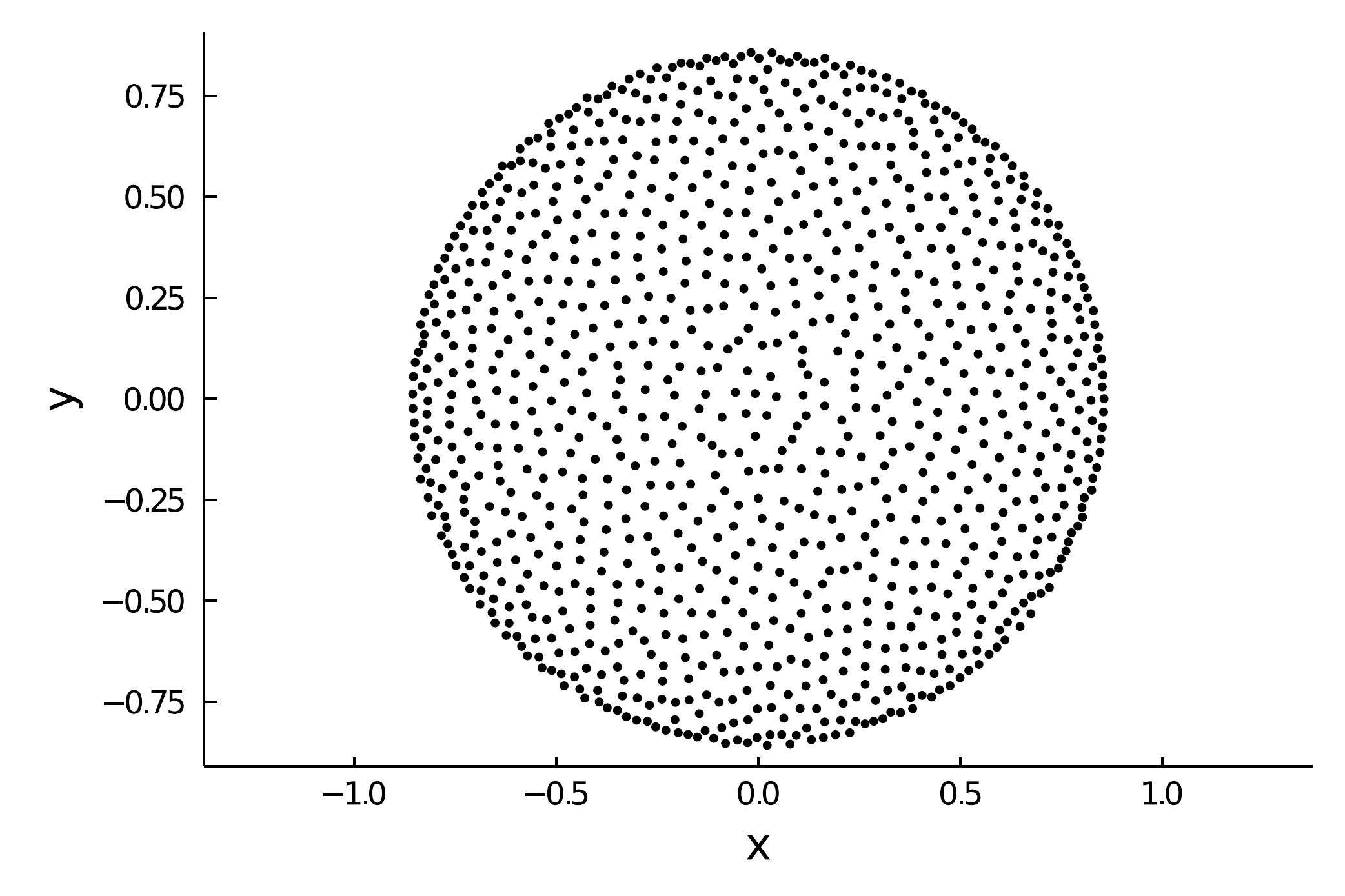} }}
         \subfloat[2D histogram based on (a)]
    {{ \centering \includegraphics[width=4.2cm]{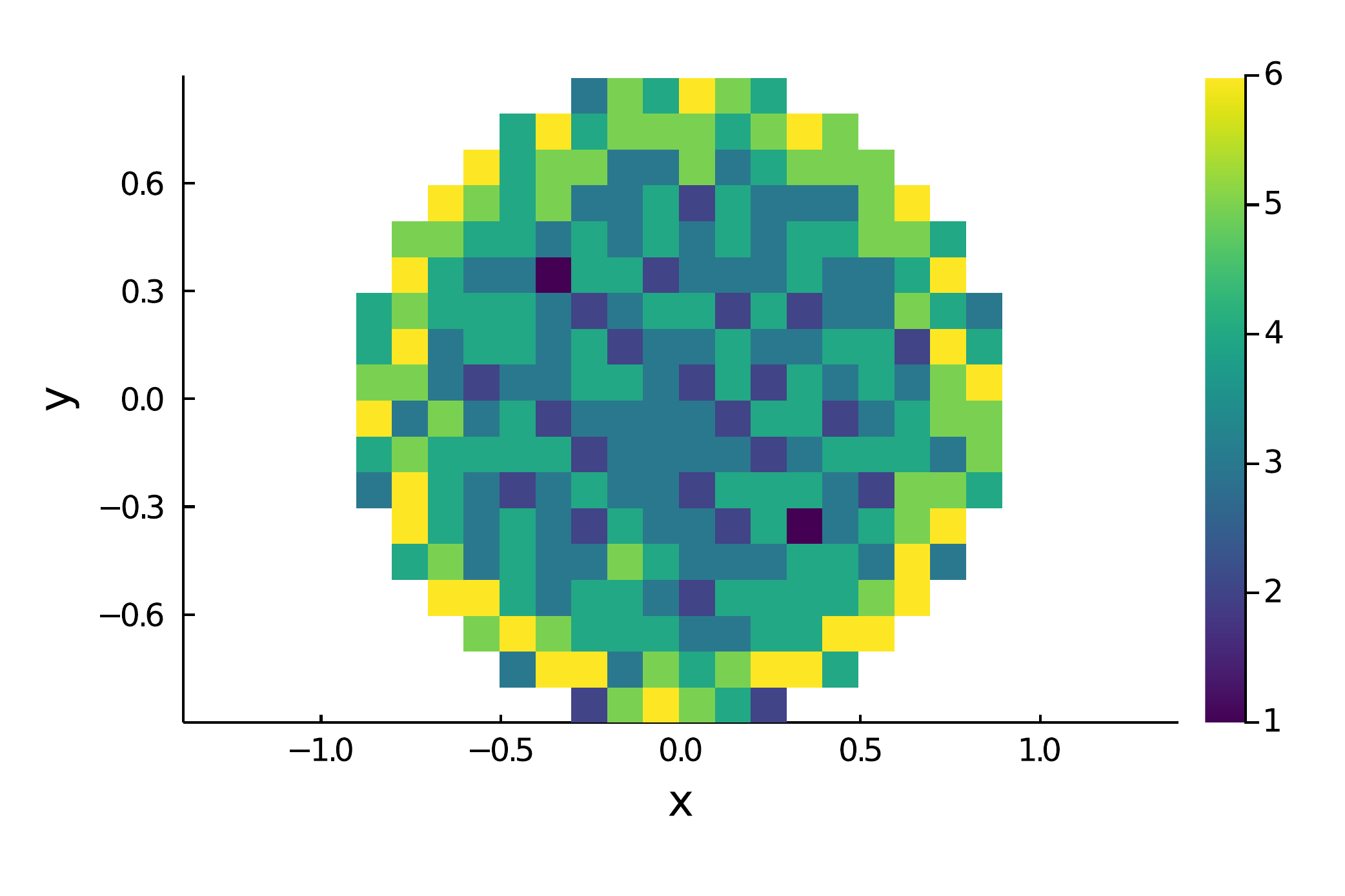} }}
             \subfloat[computed measure]
    {{ \centering \includegraphics[width=4.1cm]{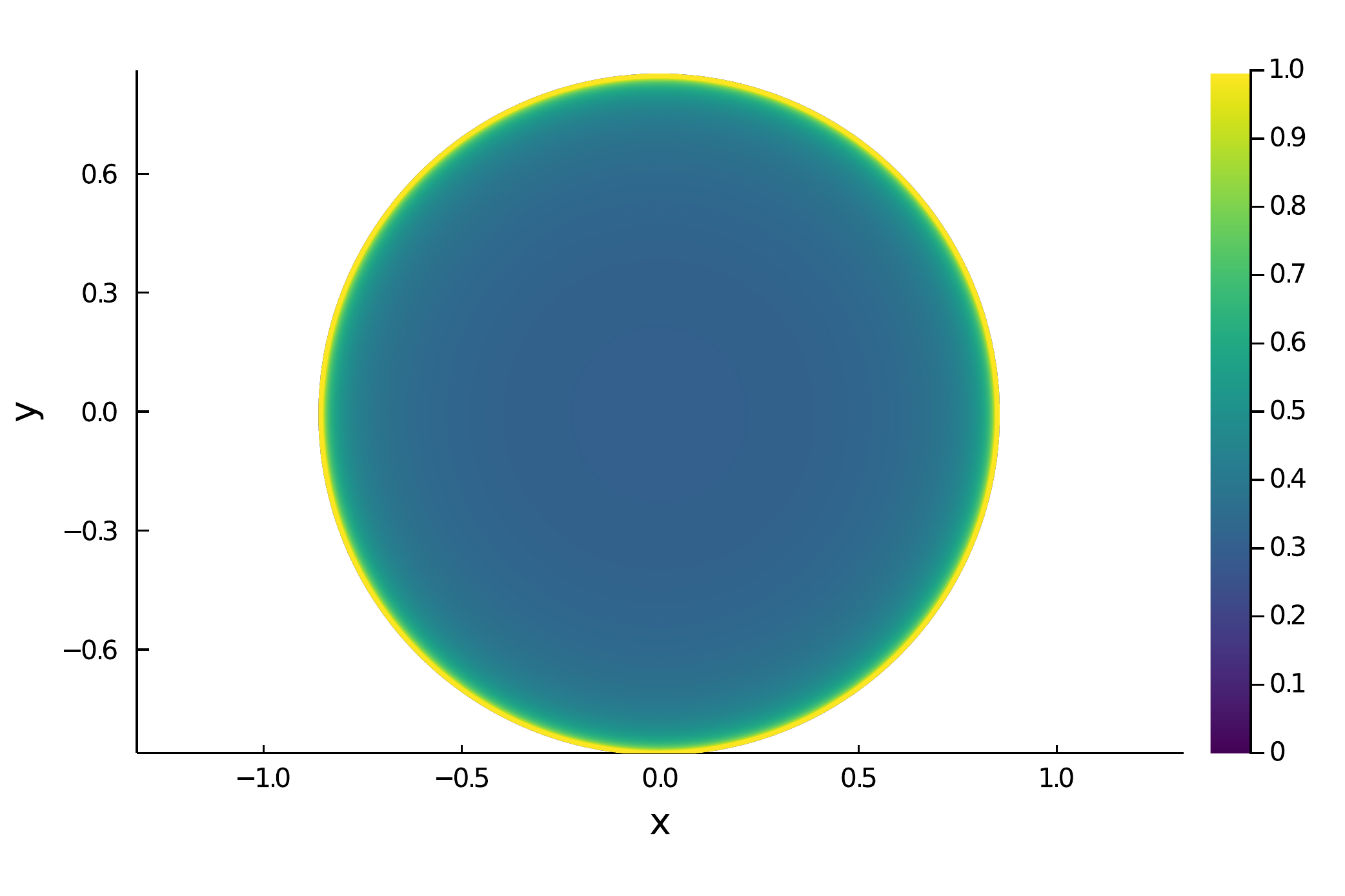} }}
    \caption{(a) shows the results of a particle simulation for the indicated equilibrium measure problem with $1000$ initially randomly distributed particles. (b) shows a two-dimensional histogram approximating the density based on the particle simulation. (c) shows the computed equilibrium measure using the method we introduced in this paper.}%
    \label{fig:particlecomparegeneric}%
\end{figure}

\subsection{Tracing the gap formation boundary}
It has been observed before \cite{balague2013dimensionality,carrillo_explicit_2017,gutleb_computing_2020} that for certain values of $\alpha$ and $\beta$, the assumption of the support of the equilibrium measure being a ball breaks down. For example, when $\alpha=4$, and $\beta > \frac{2+2d-d^2}{d+1}$ the analytic solution when assuming ball 
shaped support shows negative values at the origin \cite{carrillo_explicit_2017}, making the solution inadmissible. This gap formation phenomenon has been confirmed more generally by particle simulations in low dimensions \cite{balague2013dimensionality} and recently via an ultraspherical spectral method in one-dimension in \cite{gutleb_computing_2020}. In Figure \ref{fig:particlevoid} we show an example of this gap formation phenomenon in a two dimensional particle simulation. In \cite{gutleb_computing_2020} the authors developed a two interval approach method, which was able to obtain admissible solutions beyond the gap formation boundary. Likewise, it is expected that in higher dimensions the correct support assumption for the measure is an annulus in two dimensions and hyperspherical shells in higher dimensions. While in one-dimension this problem can be solved with a two interval approach, this is significantly more complicated in higher dimensions as hyperspherical shells are not simply comprised of two $d$-dimensional balls. \\
\begin{figure}
 \centering 
 \includegraphics[width=7.1cm]{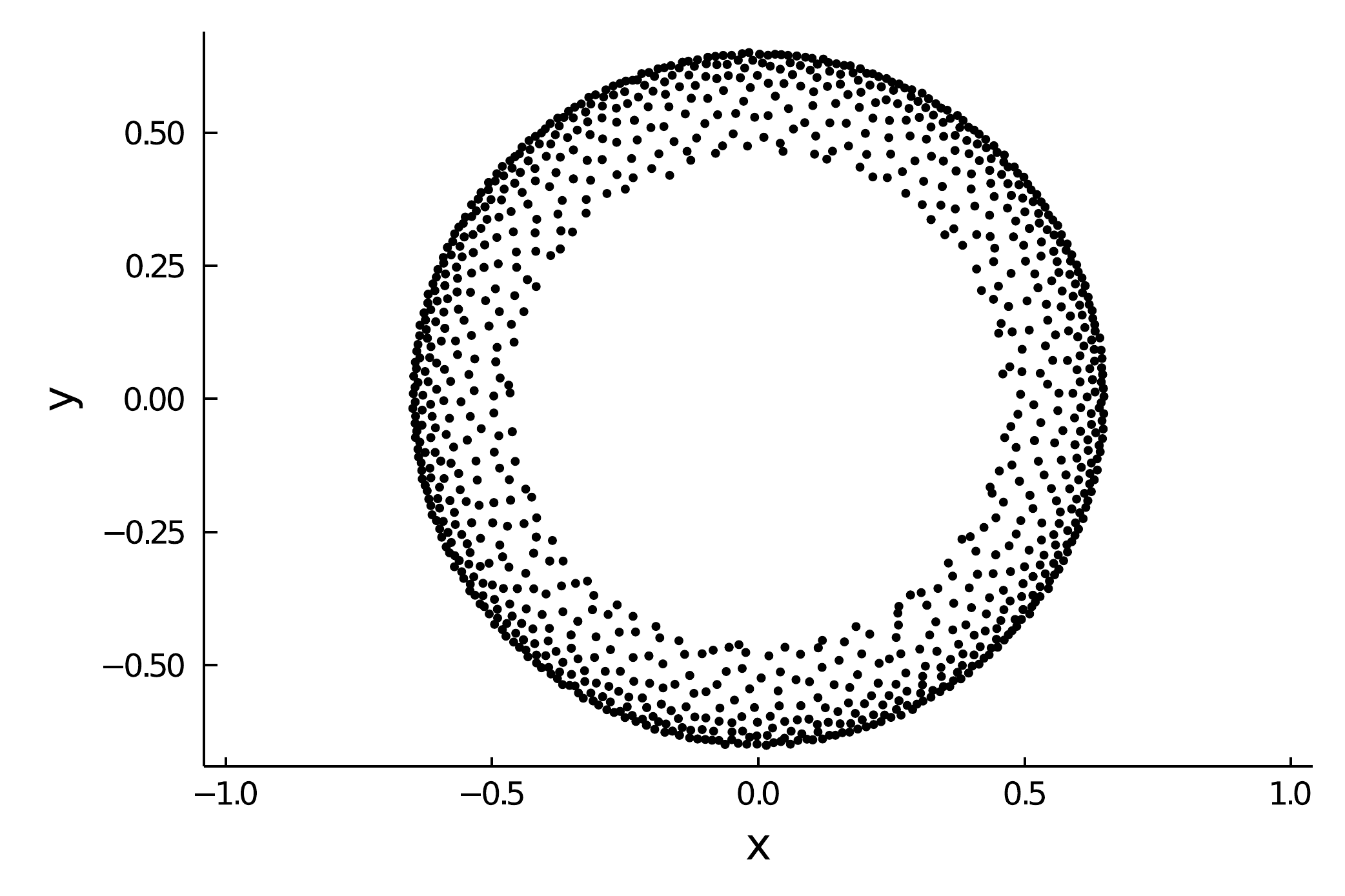}
    \caption{Particle simulation with $1000$ particles for the equilibrium measure problem with parameters $(\alpha, \beta, d, M) = (4.2, 0.85, 2, 1)$ showing an example of the gap or void formation behaviour around the origin previously observed in \cite{balague2013dimensionality,carrillo_explicit_2017,gutleb_computing_2020}.}%
    \label{fig:particlevoid}%
\end{figure}
As is, the method proposed in this paper can thus not solve equilibrium measure problems past the gap formation boundary, although we intend to work on an extension of this approach for that case in the future. The present method can, however, help with understanding the shape of said gap formation boundary similarly to what was done in \cite{gutleb_computing_2020} in one-dimension, as the obtained measures will begin to show negative values at the origin. While the boundary point is known for a very limited number of special cases, such as $\alpha=4$ mentioned above, the general form of this boundary is presently not understood and is very difficult to impossible to explore with particle simulations even in just one or two dimensions due to slow convergence in the regions of low density around the origin immediately preceeding gap formation. In Figure \ref{fig:voidcontours} we show an exploration of the boundary using our method which can be performed in any dimension.

\begin{figure}
\centering
     \subfloat[$d = 2$]
    {{ \centering \includegraphics[width=6cm]{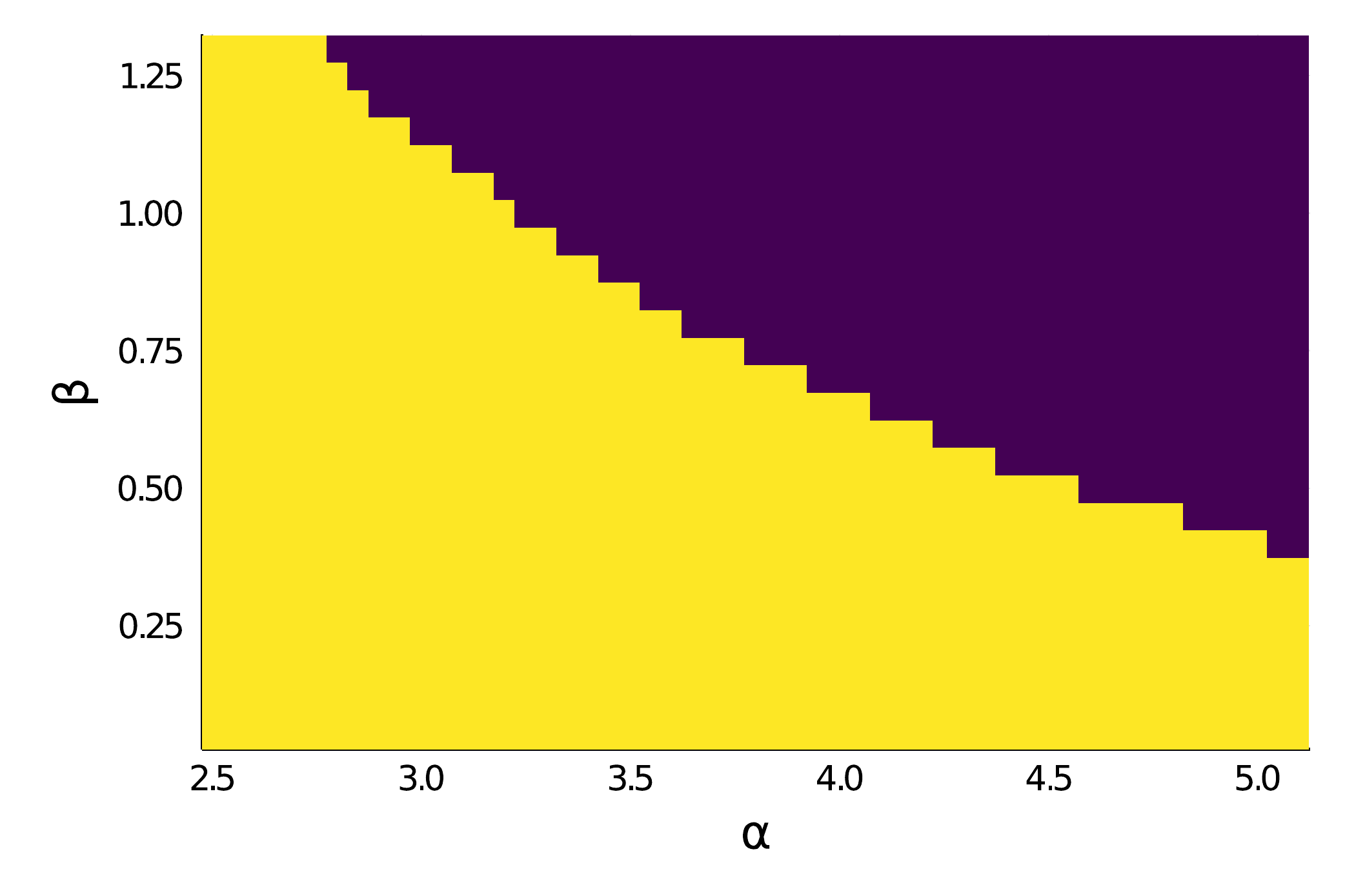} }}
         \subfloat[$d = 3$]
    {{ \centering \includegraphics[width=6.1cm]{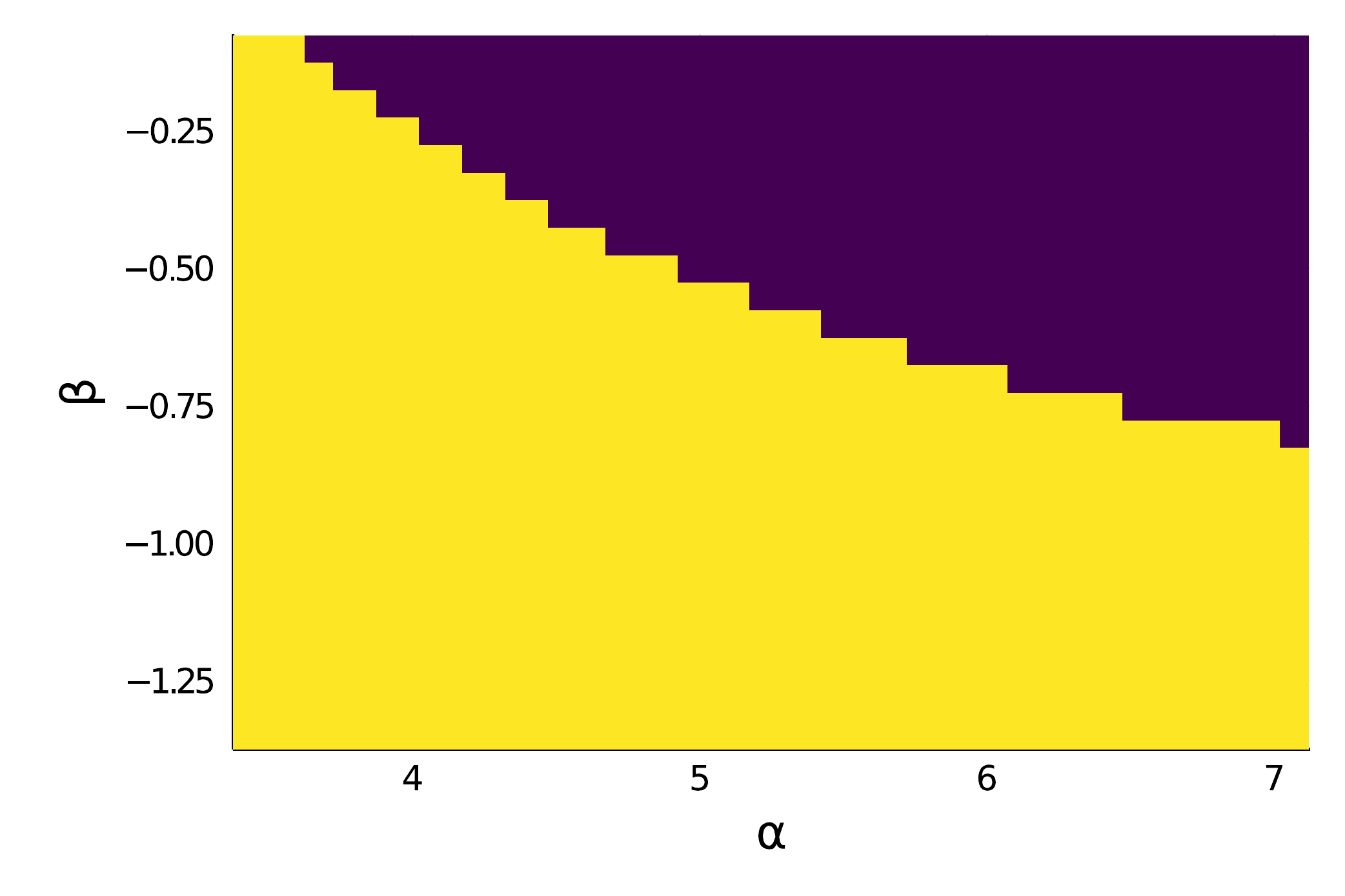} }}
    \caption{Visualization of parts of the boundary in values of $\alpha$ and $\beta$ for which no single ball non-negative measures can be found in two and three dimensions, cf. the one-dimensional equivalent discussed in \cite{gutleb_computing_2020}. Bright areas indicate existence of single ball non-negative measures, while dark areas indicate that no such measure could be found. The resolution of the measure search and heatmap was set to steps of $0.05$ but may be performed at arbitrary precision.}%
    \label{fig:voidcontours}%
\end{figure}

\section{Discussion} \label{sec:discussion}
In this paper, we have introduced a new approach to numerically solving power law equilibrium measures in arbitrary dimension for the case of ball-shaped support. Our method is based on both existing as well as to our knowledge new recurrence results for radial Jacobi polynomials and Gaussian hypergeometric functions under the action of Riesz potentials. Aspects of these results, such as Theorem \ref{theorem:hypergeomgeneral} may eventually be useful in analytic or computer-based proof aproaches to power law equilibrium measure problems, cf. \cite{carrillo_explicit_2017}. In constrast to the conventional approach -- particle simulations -- our method's complexity is independent of the dimension $d$, making convergent arbitrary precision arithmetic computations possible in any dimension and for all parameter ranges where the assumption of ball-shaped support holds. Our method also allows the exploration of where this assumption breaks down, by tracing the gap formation boundary for parameters for which negative values around the origin appear. Future research will generalize this approach to solve equilibrium measure problems past the gap formation boundary on annuli and hyperspherical shells.

\section*{Acknowledgments}

The authors were partially supported by the EPSRC grant number EP/T022132/1.
JAC was supported the Advanced Grant Nonlocal-CPD (Nonlocal PDEs for Complex Particle Dynamics: 	Phase Transitions, Patterns and Synchronization) of the European Research Council Executive Agency (ERC) under the European Union's Horizon 2020 research and innovation programme (grant agreement No. 883363). SO was supported by the Leverhulme Trust Research Project Grant RPG-2019-144.

\section{Appendices}
\subsection{A -- Extension of a Dyda-Kuznetsov-Kwaśnicki formula for fractional Laplacians}\label{app:dydaformula}
As mentioned in section \ref{sec:laplace} Dyda, Kuznetsov and Kwaśnicki proved that certain weighted Jacobi polynomials form a complete basis for the eigenfunctions of the fractional Laplacian on the $d$-dimensional unit ball \cite{dyda_fractional_2017}. We stated their result in \eqref{eq:eigenjacobi}. The domain of validity they prove for their formula is $\gamma>0$ but as we will show in this section, the formula also holds in the range $\gamma \in (-2,0)$. To prove this, we follow effectively the same proof Dyda, Kuznetsov and Kwaśnicki used, but at a critical point use a different theorem of theirs. The theorem we prove in this section is thus effectively an extension of Theorem 3 in \cite{dyda_fractional_2017}.\\
As in \cite{dyda_fractional_2017}, in the following $V_{l,m}(x)$ denotes a linear basis of the finite dimensional space with dimension
\begin{align*}
M_{d,l} = \frac{d+2l-2}{d+l-2} \begin{pmatrix}
2+l-2\\l
\end{pmatrix}
\end{align*}
spanned by solid harmonic polynomials with degree $l\geq 0$. See section \ref{sec:spectralondisk} as well as \cite{dyda_fractional_2017,dunkl_orthogonal_2014} for more details.
\begin{theorem}
Assume that $\gamma \in (-2,0)$, $l,n \geq 0$ and $1\leq m \leq M_{d,l}$. Then
\begin{align*}
(-\Delta)^{\frac{\gamma}{2}} (1-|x|^2)^{\frac{\gamma}{2}}&V_{l,m}(x)P_n^{\left(\frac{\gamma}{2},\frac{d-2}{2}+l\right)}(2|x|^2-1) \\&= \tfrac{2^\gamma \Gamma\left(1+\frac{\gamma}{2}+n \right)\Gamma \left( \frac{\delta+\alpha}{2}+n \right)}{n! \Gamma \left( \frac{\delta}{2}+n \right)}V_{l,m}(x)P_n^{\left(\frac{\gamma}{2},\frac{d-2}{2}+l\right)}(2|x|^2-1),
\end{align*}
for all $x$ such that $|x|<1$ and where $\delta := d+2l$.
\end{theorem}
\begin{proof}
We begin by following the same proof structure as in \cite[Theorem 3]{dyda_fractional_2017}: We write $V(x)=V_{l,m}(x)$. Then by the explicit hypergeometric function representation of the shifted radial Jacobi polynomials we find
\begin{align*}
\frac{n!}{\Gamma(1+\frac{\gamma}{2}+n)} (1-|x|^2)^{\frac{\gamma}{2}}&V(x)P_n^{\left(\frac{\gamma}{2},\frac{d-2}{2}+l\right)}(2|x|^2-1) \\&= V(x)(1-|x|^2)^{\frac{\gamma}{2}} {}_2 F_1\left(\begin{matrix}-n,\quad \frac{\delta+\gamma}{2}+n \\ 1+\frac{\gamma}{2}\end{matrix};1-|x|^2 \right)
\\&= V(x) G^{2,0}_{2,2}\Big( 
\,|x|^2;\begin{matrix}
1+\tfrac{\gamma}{2}+n, & 1-\tfrac{\delta}{2}-n \\ 1-\frac{\delta}{2}, &0
\end{matrix}\,
\Big ),
\\&= (-1)^n V(x)G^{1,1}_{2,2}\Big( 
\,|x|^2;\begin{matrix}
1-\tfrac{\delta}{2}-n, & 1+\tfrac{\gamma}{2}+n \\ 0, &1-\frac{\delta}{2}
\end{matrix}\ 
\Big ).
\end{align*}
Above, the second and third equalities follow from known relations for the Meijer-G function, see \cite[(8.4.49.22)]{prudnikov2003integrals} and \cite[Equation (51)]{dyda_fractional_2017}. At this point Dyda, Kuznetsov and Kwaśnicki apply the fractional Laplacian $(-\Delta)^{\frac{\gamma}{2}}$ and use Theorem $2$ in their paper which has range of validity $\gamma>0$ to ultimately arrive at \eqref{eq:eigenjacobi}. Instead of doing this, we will instead use Theorem $1$ from their paper, which is an analogous result to their Theorem $2$ but for Riesz potentials and has range of validity $\gamma \in (-d,0)$. Using \cite[Theorem 1]{dyda_fractional_2017} in the first step and then applying cancellation formulas for the Meijer-G function, cf. \cite[Equation (21--22)]{dyda_fractional_2017} yields
\begin{align*}
(-\Delta)^{\frac{\gamma}{2}} &\frac{n!}{\Gamma(1+\frac{\gamma}{2}+n)} (1-|x|^2)^{\frac{\gamma}{2}}V(x)P_n^{\left(\frac{\gamma}{2},\frac{d-2}{2}+l\right)}(2|x|^2-1)\\
&= (-1)^n 2^\gamma V(x)G^{2,2}_{4,4}\Big( 
\,|x|^2;\begin{matrix}1-\tfrac{\delta+\gamma}{2}, &
1-\tfrac{\delta+\gamma}{2}-n, & 1+n, & -\tfrac{\gamma}{2} \\0, & -\tfrac{\gamma}{2}, &1-\frac{\delta+\gamma}{2}, & 1-\tfrac{\delta}{2}
\end{matrix}\,
\Big )\\
&=(-1)^n 2^\gamma V(x)G^{1,1}_{2,2}\Big( 
\,|x|^2;\begin{matrix}
1-\tfrac{\delta+\gamma}{2}-n, & 1+n \\ 0, &1-\frac{\delta}{2}
\end{matrix}\,
\Big )\\
&=\tfrac{(-1)^n 2^\gamma V(x) \Gamma \left( \frac{\delta+\gamma}{2}+n \right)}{n!} {}_2 F_1\left(\begin{matrix}\tfrac{\delta+\gamma}{2}+n, \quad -n \\ \tfrac{\delta}{2} \end{matrix} ; |x|^2 \right),
\end{align*}
which is valid for $|x|<1$. Multiplying with the appropriate constants to simplify the left-hand side and converting the hypergeometric function into Jacobi polynomial form then finally yields the stated result. This last step restricts us to $\gamma \in (-2,0)$ for classical Jacobi polynomials which require both parameters to be greater than $-1$.
\end{proof}
\begin{remark}
Just as noted for \cite[Theorem 3]{dyda_fractional_2017} for the fractional Laplacian, this result implies that the stated Jacobi polynomials form a complete orthogonal system of eigenfunctions for the weighted Riesz operator in the stated ranges.
\end{remark}

\subsection{B -- Derived properties of radially shifted Jacobi polynomials}\label{app:radialjacobi}
For ease of reference, this section concisely lists some of the basic properties for the radial Jacobi polynomials $P_n^{(a,b)}(2|x|^2-1)$, where $|x|\in(0,1)$. They follow directly from the respective properties of the ordinary Jacobi polynomials, cf. \cite[18.9]{nist_2018}.
\subsubsection*{Classical recurrence relationship}
\begin{equation}\label{eq:radialjacobirecclassical}
P_{n+1}^{(a,b)}(2|x|^2-1)=(2A_{n}|x|^2+(B_{n}-A_{n}))P^{(a,b)}_{n}(2|x|^2-1)-C_{n}P^{(a,b)}_{n-1}(2|x|^2-1),
\end{equation}
with $A_n$, $B_n$ and $C_n$ as in section \ref{sec:jacobi}.
\subsubsection*{Explicit representations}
\begin{align}\label{eq:radialseriesrep}
P^{(a,b)}_{n}\left(2|x|^2-1\right)&= \sum_{k=0}^{n} (-1)^{n+k} \tfrac{{\left(n+a+%
b+1\right)_{k}}{\left(b+k+1\right)_{n-k}}}{k!\;(n-k)!}%
|x|^{2k},\\\nonumber
&=\tfrac{\Gamma (a+n+1)}{n!\,\Gamma (a+b+n+1)} \sum_{k=0}^n (-1)^k \binom{n}{k} \tfrac{\Gamma (a + b + n + k + 1)}{\Gamma (a + k + 1)} \left(1-|x|^2\right)^k,\\
&= \frac{(a+1)_n}{n!} {}_2F_1\left(\begin{matrix}-n,\quad n+a+b+1 \\ a+1 \end{matrix} ;1-|x|^2 \right)\nonumber \\
&= (-1)^n \frac{(b+1)_n}{n!} {}_2F_1\left( \begin{matrix} -n, \quad n+a+b+1 \\ b+1 \end{matrix} ;|x|^2 \right).\nonumber
\end{align}
\subsubsection*{Symmetry}
\begin{align}\label{eq:radialsymm}
P_n^{(a,b)}(1-2|x|^2) = (-1)^n P_n^{(b,a)}(2|x|^2-1).
\end{align}
\subsubsection*{Basis conversion}
\begin{align}\label{eq:radialbasisconversionjacobi1}
P^{(a,b)}_{n}\left(2|x|^2-1\right)=\tfrac{(n+a+b+1)}{(2n+a+b+1)}P^{(%
a+1,b)}_{n}\left(2|x|^2-1\right)-\tfrac{(n+b)}{(2n+a+b+1)}P^{(a+1,b)}_{n-1}\left(2|x|^2-1%
\right),\\
P^{(a,b)}_{n}\left(2|x|^2-1\right)=\tfrac{(n+a+b+1)}{(2n+a+b+1)}P^{(%
a,b+1)}_{n}\left(2|x|^2-1\right)+\tfrac{(n+a)}{(2n+a+b+1)}P^{(a,b+1)}_{n-1}\left(2|x|^2-1%
\right).\nonumber
\end{align}
\begin{align}
|x|^2P^{(a,b+1)}_{n}\left(2|x|^2-1%
\right)=K_n P^{(a,b)}_{n+1}\left(2|x|^2-1\right)+\tfrac{(n+b+1)}{2(n+\frac{a}{2}+\frac{b}{2}+1)}P^{(a,%
b)}_{n}\left(2|x|^2-1\right),\label{eq:radialjacobiraisingB}\\
(|x|^2-1)P^{(a+1,b)}_{n}\left(2|x|^2-1%
\right)=K_n P^{(a,b)}_{n+1}\left(2|x|^2-1\right)-\tfrac{(n+a+1)}{2(n+\frac{a}{2}+\frac{b}{2}+1)}P^{(a,%
b)}_{n}\left(2|x|^2-1\right),\nonumber
\end{align}
with $K_n = \tfrac{(n+1)}{2(n+\frac{a}{2}+\frac{b}{2}+1)}$.

\bibliographystyle{siamplain}
\bibliography{references}

\end{document}